\documentclass[11pt]{amsart}

\usepackage{soul}
\usepackage{amssymb,latexsym,amsmath,amsthm,graphicx}
\usepackage[numbers]{natbib}
\usepackage{fancyhdr}
\usepackage{sidecap}
\usepackage[all,cmtip]{xy}
\usepackage[pdfencoding=auto]{hyperref}

\renewcommand{\AA}{\mathbb{A}}
\newcommand{\ZZ}{\mathbb{Z}}

\newcommand{\RR}{\mathbb{R}}
\newcommand{\PP}{\mathbb{P}}
\newcommand{\MM}{\mathbb{M}}
\newcommand{\KK}{\mathbb{K}}
\newcommand{\CC}{\mathbb{C}}
\renewcommand{\SS}{\mathbb{S}}

\newcommand{\cA}{\mathcal{A}}

\newcommand{\cC}{\mathcal{C}}
\newcommand{\cD}{\mathcal{D}}
\newcommand{\cE}{\mathcal{E}}
\newcommand{\cF}{\mathcal{F}}

\newcommand{\cH}{\mathcal{H}}

\newcommand{\cL}{\mathcal{L}}
\newcommand{\cM}{\mathcal{M}}

\newcommand{\cO}{\mathcal{O}}
\newcommand{\cP}{\mathcal{P}}

\newcommand{\cT}{\mathcal{T}}

\newcommand{\cX}{\mathcal{X}}

\DeclareMathOperator{\phase}{phase}
\DeclareMathOperator{\rk}{rk}

\DeclareMathOperator{\Ext}{Ext}
\DeclareMathOperator{\Hom}{Hom}

\DeclareMathOperator{\Mod}{Mod}

\DeclareMathOperator{\Coh}{Coh}

\DeclareMathOperator{\Stab}{Stab}
\DeclareMathOperator{\PreStab}{PreStab}
\DeclareMathOperator{\RelStab}{RelStab}

\DeclareMathOperator{\SL}{SL}

\newcommand{\glue}{\mathrm{glue}}
\newcommand{\cut}{\mathrm{cut}}

\newcommand{\del}{\partial}

\newtheorem{theorem}{Theorem}
\newtheorem{proposition}[theorem]{Proposition}
\newtheorem{corollary}[theorem]{Corollary}
\newtheorem{lemma}[theorem]{Lemma}
\newtheorem*{theorem*}{Theorem}

\theoremstyle{definition}
\newtheorem{definition}[theorem]{Definition}

\theoremstyle{remark}
\newtheorem*{remark}{Remark}
\newtheorem*{example}{Example}

\author{Alex Takeda}
\title{Relative stability conditions on Fukaya categories of surfaces}

\begin{document}
\begin{abstract}
It is shown that there is a useful notion of a relative Bridgeland stability condition on the partially wrapped Fukaya category of a marked surface, relative to some part of the surface's boundary. This construction has nice functorial properties, obeying cutting and gluing relations. This reduces the calculation of stability conditions on the Fukaya category of any fully stopped surface into three types of base cases. Calculations of these cases shows that every Bridgeland stability condition on such categories can be described by flat surfaces. In other words, the map constructed by Haiden-Katzarkov-Kontsevich from the moduli of flat surfaces to the stability space of the Fukaya category is a global homeomorphism when the surface is fully stopped.
\end{abstract}
\maketitle

\section{Introduction}
In the article \citep{Bri1}, T. Bridgeland defines a notion of stability conditions on triangulated categories, having as inspiration the stability of D-branes in string theory and SCFTs \citep{aspinwall2002d,douglas2005stability}. This definition generalizes the classical concept of slope-stability for vector bundles and is quite remarkable; in particular the space of such structures naturally carries the structure of a complex manifold with a natural action by the group of automorphisms of the category.

The space $\Stab(\cD)$ of Bridgeland stability conditions (in this paper just `stability conditions') on a triangulated category $\cD$ has been well-understood for many cases of geometric interest. For instance, on the `B-side' of mirror symmetry, the initial example to be examined by Bridgeland is the calculation of $\Stab(\cD)$ when $\cD$ is the derived category of coherent sheaves on the elliptic curve \citep{Bri1}. Following this we have Macr\`i's calculation for higher-genus curves \citep{Mac1} and Okada's description of $\Stab(\Coh(\PP^1))$ \cite{Oka}. The complete description of stability conditions on compact surfaces is also known, due to the work of Bridgeland \citep{Bri2}, Toda \citep{toda2014stability}, Okada \citep{okada2006stability} and others, and the difficult case of smooth projective threefolds  \citep{bayer2011bridgeland,bayer2011bridgeland2,maciocia2015fourier} has been a subject of recent developments, with the construction of a family of stability conditions on the quintic threefold \citep{li2018stability}. 

The analogous questions for noncompact spaces \citep{bridgeland2006stability,huybrechts2008stability,Bay2} are often more tractable, and so are the cases of categories defined by quivers and other representation-theoretic data \citep{bridgeland2014stability,king2015exchange,qiu2015stability,ikeda2017stability,dimitrov2016bridgeland}. In these cases, it is often possible to construct families of stability conditions since one has explicit exceptional collections \citep{dimitrov2015bridgeland}; however, constraining all the components of stability spaces requires strong finiteness conditions on $\cD$.

On the other side (A-side) of mirror symmetry there have been many indications that stability conditions on Fukaya categories can recover some geometric data, in particular information related to special Lagrangian geometry \citep{Joy,Smi}. In \citep{HKK}, Haiden, Katzarkov and Kontsevich look at stability conditions on the partially wrapped Fukaya category of a marked surface $\Sigma$, and show that the space of stability conditions on $\cF(\Sigma)$ is related to the geometry of quadratic differentials on $\Sigma$. The relation between moduli spaces of quadratic differentials and spaces of stability conditions already appeared in the work of Bridgeland and Smith \citep{bridgeland2015quadratic}, albeit for a different category. The relation between these two appearances of quadratic differentials for these related categories is explained in \cite{ikeda2018q}.

In more detail, the authors of \citep{HKK} construct a map
\[ \cM(\Sigma) \to \Stab(\cF(\Sigma)) \]
to the space of stability conditions from a moduli space $\cM(\Sigma)$ of ``marked flat structures'' on $\Sigma$, whose points are given by a Riemann surface $X$ diffeomorphic to a compactification of $\Sigma$ together with a meromorphic quadratic differential $\phi$ on $X$, with singularities type prescribed by the marking data of $\Sigma$. This map is proven to be a homeomorphism to a union of connected components of $\Stab(\cF(\Sigma))$; we will call the stability conditions in these components \emph{HKK stability conditions}. In some small cases (disk and annulus), it is shown that this image is in fact the whole stability space, using finiteness properties of these categories.

Those examples display a recurring feature of the existing calculations of $\Stab(\cD)$; it is easier to make statements about individual components of these spaces than to know them in their entirety, since it is \emph{a priori} possible that there might be exotic stability conditions that do not correspond to intuitive geometric structures, living in components of $\Stab(\cD)$ that cannot be accessed by deformations from known stability conditions of geometric origin.

One of the reasons for this recurring difficulty is the relative lack of general tools for constructing stability conditions in a functorial manner. The two constituent parts of a stability condition, the central charge and the slicing, have opposite functoriality, and it is not obvious that stability conditions should exhibit any sheaf- or cosheaf-like behavior. It appears then that one must start with knowing the `global' behavior of the geometry to study stability conditions; all the cases cited above rely heavily on knowledge of the global behavior of morphisms between objects.

The initial motivation for this paper is the observation that \citep{HKK} provides an interesting counterexample to that general principle, since it builds stability conditions on $\cF(\Sigma)$ from geometric objects with nice functorial properties, namely flat structures that glue along nicely under a certain type of decompositions of surfaces. For example, given a decomposition of a marked surface $\Sigma$ into two pieces $\Sigma_1$ and $\Sigma_2$ mutually overlapping along a rectangular strip $R$, and a flat structure on $\Sigma$, restricting the flat structure to each side gives a flat structure (with a new boundary `at infinity'). Moreover, once one defines the appropriate notion of compatibility between flat structures along the strip, one can glue compatible flat structures on $\Sigma_1$ and $\Sigma_2$ into a flat structure on $\Sigma$.

This paper is an effort towards abstracting this idea of cutting and gluing flat surfaces in terms of stability conditions on their Fukaya categories. The appropriate local pieces of this construction are presented in Section \ref{sec:relstab}, where we introduce the definition of \emph{relative stability conditions} on a marked surface. A relative stability condition (Definition \ref{def:relativestabcond}) on $\Sigma$ with respect to one of its unmarked boundary arcs $\gamma$ is an ordinary stability condition on an extended surface $\tilde\Sigma$, obtained from $\Sigma$ by an appropriate modification along the part of the boundary isotopic to $\gamma$.

Let $\RelStab(\Sigma,\gamma)$ denote the set of relative stability conditions on $\Sigma$ with respect to $\gamma$. This set inherits a topology from spaces of (ordinary) Bridgeland stability conditions $\Stab(\cF(\tilde\Sigma))$, and also a compatible ``generalized metric'' coming from Bridgeland's generalized metric on stability space. The resulting topological space is infinite-dimensional, but can be shown (Proposition \ref{prop:hausdorff}) to be a Hausdorff space.

Consider a decomposition $\Sigma = \Sigma_L \cup_\gamma \Sigma_R$ of a marked surface $\Sigma$ into two surfaces glued along boundary arcs. Our main technical result is about the existence of cutting and gluing maps relating stability conditions on $\Sigma$ and relative stability conditions on $\Sigma_L$ and $\Sigma_R$.
\begin{theorem}\label{thm:cutglue}
There is a relation of compatibility along $\gamma$ defining a subset $\Gamma \subset \RelStab(\Sigma_L,\gamma)\times \RelStab(\Sigma_R,\gamma)$ and continuous maps
\[ \Stab(\cF(\Sigma)) \xrightarrow{\cut_\gamma} \Gamma \xrightarrow{\glue_\gamma} \Stab(\cF(\Sigma)) \]
which, when restricted to the locus of stability conditions whose stable objects are all supported on intervals, are inverse homeomorphisms.
\end{theorem}
Moreover, the cutting map $\cut_\gamma$ behaves nicely with respect to Bridgeland's generalized metric on the stability space; in particular it never sends points in different connected components of $\Stab(\cD)$ (that is, at an infinite distance with respect to the generalized metric) to points at a finite distance in the relative stability spaces (Lemma \ref{lem:distance}).

Consider now any marked graded surface $\Sigma$ that is `fully stopped', ie. every boundary circle has at least one marked interval. In Section \ref{sec:calculations}, we describe a procedure for reducing the calculation of $\Stab{\cF(\Sigma)}$ to the calculation of (ordinary) stability conditions on three types of surfaces with marked boundary: the disk, the annulus and the punctured torus. In all of these cases it can be shown that every stability condition is an HKK stability condition, ie. the map $\cM(\Sigma) \to \Stab(\cF(\Sigma))$ is an isomorphism. The cases of the disk and of the annulus are dealt with in \citep{HKK}, but the calculation for the case of the punctured torus is new. 

These calculations, together with Theorem \ref{thm:cutglue} and the metric properties of the cutting map, imply the following result:
\begin{theorem}\label{thm:noexotic}
If $\Sigma$ is fully stopped, every Bridgeland stability condition on $\cF(\Sigma)$ is an HKK stability condition, ie. given by a flat structure on $\Sigma$.
\end{theorem}

The cutting and gluing procedures can be used to give an explicit description of the spaces $\Stab(\cF(\Sigma))$ in this fully stopped case. In the forthcoming work \cite{takeda2021future}, we use this technique to give a combinatorial description of (a generalized form of) wall-and-chamber structures in these stability spaces.

Recent work of Lekili-Polishchuk \cite{lekili2020derived} (see also Opper-Plamondon-Schroll \cite{opper2018geometric} for a similar result in a slightly different context) establishes a two-way dictionary between graded marked surfaces and smooth $\ZZ$-graded \emph{gentle algebras}; under this equivalence, fully-stopped surfaces correspond to smooth and finite-dimensional gentle algebras. Therefore as a corollary we can also use relative stability conditions to study the structure of stability spaces of representation categories of such algebras.

\section*{Acknowledgments}
I would like to especially thank Tom Bridgeland, Fabian Haiden and my graduate advisor Vivek Shende for very helpful conversations and great patience in explaining technical points to me. I would also like to thank Dori Bejleri, Benjamin Gammage, Tatsuki Kuwagaki, David Nadler, Ky\=oji Sait\=o, Ryan Thorngren and Gjergji Zaimi for helpful discussions. A crucial part of this work was conducted during a working visit to the IPMU in Japan, and I would like to thank the faculty and staff of that institute for providing a great working environment. This project was supported in part by NSF CAREER DMS-1654545.

\section{Background}

\subsection{Fukaya categories of graded marked surfaces}
We fix a field $\KK$ of characteristic zero. A graded marked surface (or just surface for brevity) is a smooth oriented surface $\Sigma$ with boundary $\del\Sigma$ and a set of \emph{marked boundaries} $\MM$, and a grading. The set $\MM$ has as elements intervals contained in $\del\Sigma$; the intervals in the complement $\del\Sigma \setminus \MM$ will be the \emph{unmarked boundaries}. Let us assume throughout this paper that each component of $\del \Sigma$ has at least one marked boundary.

A grading on $\Sigma$ is a line field $\eta \in \Gamma(\Sigma,\PP T\Sigma)$. The set of gradings on $\Sigma$, up to graded diffeomorphism isotopic to the identity, is a torsor over $H^1(\Sigma,\ZZ)$. Curves immersed in $\Sigma$ can be graded with respect to the line field $\eta$; this defines the integer degree of a point of intersection between curves. We can equivalently interpret the choice of data $(\Sigma,\MM,\eta)$ above as a stopped surface $(\Sigma,\Lambda,\eta)$ (in the sense of \cite{auroux2010fukaya,lekili2020derived}) by collapsing each unmarked boundary interval to a point, giving a stop; our standing assumption then implies we only work with \emph{fully stopped surfaces}. The data of a marked surface or a stopped surface are equivalent.

An arc in $\Sigma$ is an embedded interval with ends on marked boundaries, and an arc system $\cA$ is a collection of pairwise disjoint and non-isotopic arcs.

\begin{figure}[h]
    \centering
    \includegraphics[width=0.45\textwidth]{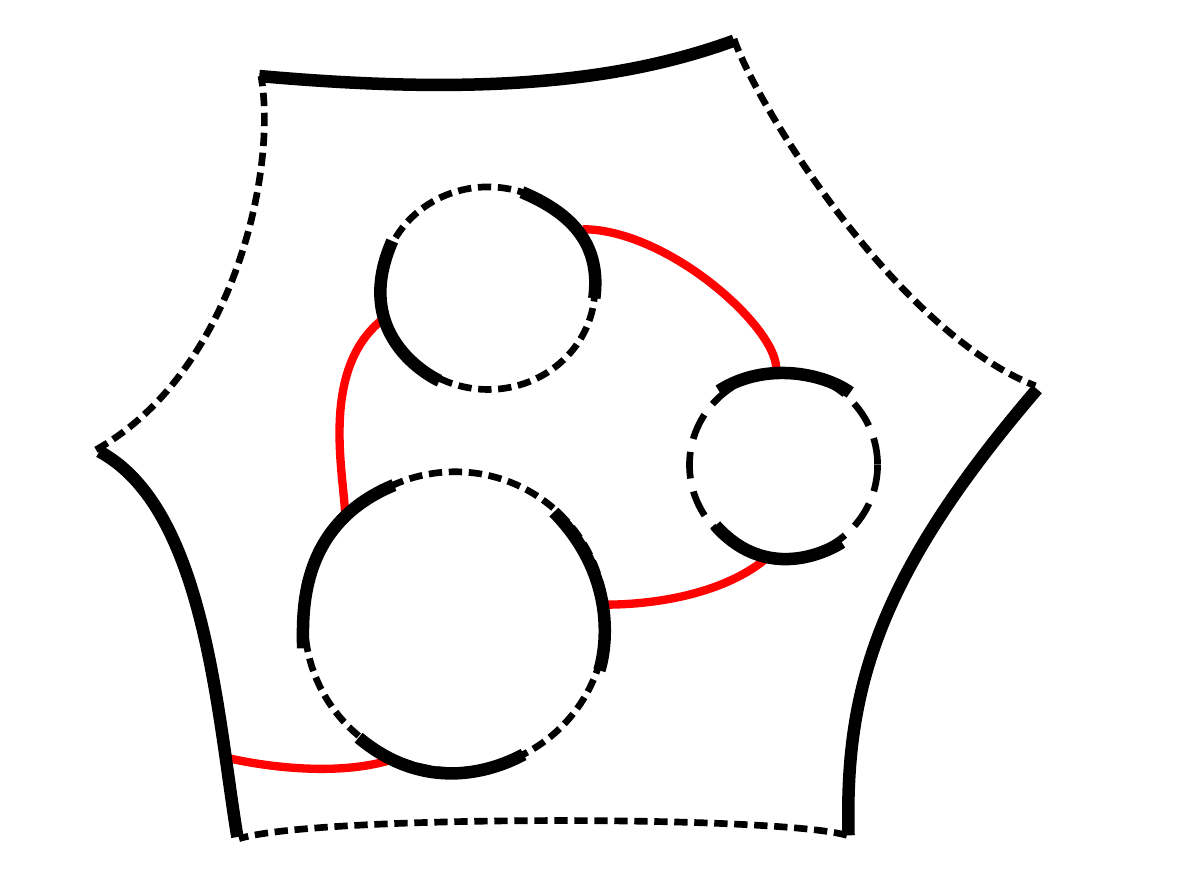}
    \caption{A marked surface with a system of arcs in red. The marked boundary intervals are denoted by solid black lines and the unmarked ones by dotted black lines.}
\end{figure}

Given any full arc system (that is, dividing the surface into polygons and including arcs isotopic to all unmarked boundaries) $\cA$, one can define an $\ZZ$-graded $A_\infty$-category $\cF_\cA(\Sigma)$; the triangulated Fukaya category can then be defined \citep{HKK} as the the category $\mathrm{Tw} \cF_\cA(\Sigma,\MM,\eta)$ of twisted complexes. This is a triangulated $A_\infty$-category, which is proven to be independent of the choice of $\cA$, up to equivalence. In this paper, we will denote by $\cF(\Sigma) = H^0(\mathrm{Tw} \cF_\cA(\Sigma,\MM,\eta))$ its homotopy category; this is a $\KK$-linear triangulated category in the usual sense. For conciseness, we will simply refer to this triangulated category as `the Fukaya category'.

From its description as the homotopy category of the category of twisted complexes, one might think that general indecomposable objects of $\cF(\Sigma)$ could be hard to classify. Nevertheless, it can be shown to that any indecomposable object admits a description in terms of certain immersed curves.

An \emph{admissible graded curve} $\gamma$ in $\Sigma$ is either an immersed interval ending at marked boundaries or an immersed circle, equipped with a grading. To be admissible, the immersed curve $\gamma$ is not allowed to bound any teardrops \citep[Sec.2.1]{HKK}, or to be nullhomotopic, or isotopic to one of the marked boundaries.
\begin{theorem}\citep[Theorem 4.3]{HKK}\label{thm:geometricity}
Every isomorphism class of indecomposable objects in $\cF(\Sigma)$ can be represented by an admissible graded curve endowed with an indecomposable local system, uniquely up to graded isotopy.
\end{theorem}
By a local system on an admissible curve, it is meant a local system on the source (either an interval or a circle) of the immersion. The geometricity theorem above admits generalizations to $\ZZ/2$-graded and non-exact cases, as it was recently explained in \citep{auroux2020fukaya}.

Another result of \citep{HKK} is a description of $K_0(\cF(\Sigma))$: the grading on $\Sigma$ gives a double cover $\tau$, which corresponds to a local system of abelian groups $\ZZ_\tau = \ZZ \otimes_{\ZZ/2} \tau$.
 \begin{theorem}\citep[Theorem 5.1]{HKK}\label{thm:K0}
There is a natural isomorphism of abelian groups $K_0(\cF(\Sigma)) \cong H^1(\Sigma,\MM; \ZZ_\tau)$.
\end{theorem}

\subsection{Intersection and morphisms}
Besides the combinatorial description above in terms of twisted complexes of arc systems, the category $\cF(\Sigma)$ admits also a description in terms of symplectic geometry as (the triangulated completion) of the \emph{partially wrapped Fukaya category} of the corresponding stopped surface $(\Sigma,\Lambda,\eta)$ \cite{auroux2010fukaya}. Objects in this category are given by unobstructed Lagrangians $L \subset \Sigma$, endowed with finite-rank local systems, which are allowed to end on the unstopped boundary. Morphisms between two such objects are given by the partially wrapped Floer complex $CW^*(L_1,L_2)$, generated by the intersection points $p \in L_1 \cap L_2$ (after appropriate perturbation) \emph{and} by clockwise boundary paths $L_1 \rightsquigarrow L_2$ which avoid the stopped part $\Lambda$. The differential on $CW^*(L_1,L_2)$ is given by counting bigons between $L_1$ and $L_2$.

As in the combinatorial description, one can take the category of twisted complexes on these Lagrangians and then its derived category; this symplectic approach should furnish an equivalent triangulated category. We will not need to use this explicit equivalence; instead we will now establish some facts about the morphism spaces and extensions of $\cF(\Sigma)$, inspired by the symplectic description.

Let us first generalize some topological definitions to include the data of the marked boundary. Let $\gamma_1,\gamma_2$ be two admissible curves in $\Sigma$, intersecting transversely. We modify the notion of intersection in the following way.
\begin{definition}
The set of \emph{directed intersections} $\gamma_1 \cap' \gamma_2$ is equal to the disjoint union of the set of points $\gamma_1 \cap \gamma_2$ and the set of (isotopy classes of) boundary paths $\gamma_1 \rightsquigarrow \gamma_2$, i.e. paths from an end of $\gamma_1$ to an end of $\gamma_2$, along some marked boundary $M$, keeping the interior of $\Sigma$ to its right.
\end{definition}

Note that in this definition the order of the intervals matters. We also modify the notion of bigon bounded by $\gamma_1$ and $\gamma_2$ in a similar way.
\begin{definition}
A \emph{generalized bigon} between $\gamma_1$ and $\gamma_2$ is either an embedded bigon in $\Sigma$ bound by $\gamma_1$ on one side, $\gamma_2$ on the other, and two intersection points $p,q \in \gamma_1 \cap \gamma_2$ at each end, \emph{or} an embedded triangle in $\Sigma$ bound by $\gamma_1$, $\gamma_2$ and an interval $I \subset M$ contained in a marked boundary; in other words a `bigon' with a genuine intersection point at one end, and some directed intersection in $\gamma_1 \cap' \gamma_2$ or $\gamma_2 \cap' \gamma_1$.
\end{definition}

Let us now define a notion of representatives with minimal intersection among admissible curves in some isotopy class.
\begin{definition}
We say that the admissible curves $\gamma_X, \gamma_Y$ are in \emph{minimal position} if:
\begin{enumerate}
    \item they only intersect transversely,
    \item there are no generalized bigons bound by each of $\gamma_X$ and $\gamma_Y$,
    \item there are no generalized bigons between $\gamma_X$ and $\gamma_Y$.
\end{enumerate}
\end{definition}
Note that in the two definitions above the order of the two curves does not matter. Note also that in the definition of generalized bigon we \emph{do not} include `bigons' bound by boundary paths on both ends (i.e. embedded quadrilaterals between the curves); otherwise two isotopic immersed intervals could never be put in minimal position. By `embedded bigons', we do allow bigons whose interior intersects other parts of the curves $\gamma_X,\gamma_Y$. 

Let us state now an auxiliary lemma, that will be used in this paper to reduce questions about circle objects to questions about interval objects.
\begin{lemma}\label{lem:circletointerval}
Let $X$ be an object of $\cF(\Sigma)$ supported on an immersed circle $\gamma$, bounding no embedded bigons. Then there are objects $X_1$ and $X_2$, supported on immersed intervals $\gamma_1,\gamma_2$ ending on marked boundaries $M,N$ giving extension maps $\eta_M,\eta_N \in \Hom(X_1,X_2[1])$, such that:
\begin{enumerate}
    \item there is a distinguished triangle $X_2 \to X \to X_1 \xrightarrow{\eta_M + \eta_N} \dots$
    \item the curves $\gamma_1$ and $\gamma_2$ are in minimal position, and
    \item there is a bijection between the set $\gamma \pitchfork \gamma$ of self-intersections of the immersed curve $\gamma$ and the set $(\gamma_1 \pitchfork \gamma_1) \cup (\gamma_2 \pitchfork \gamma_2) \cup (\gamma_1 \cap' \gamma_2).$
\end{enumerate}
Note that the marked boundaries $M$ and $N$ may not be distinct.
\end{lemma}
\begin{proof}
This claim follows from the fact that roughly speaking, there is at least one marked boundary to each side of any immersed circle, embedded or not. Let us be more explicit: consider first the embedded case. If $\gamma$ is non-separating, then all of $\Sigma$ sits on both sides of $\gamma$, and since our surfaces have at least one marked boundary, we can pick paths from marked boundaries $M,N$ (not necessarily distinct) to $\gamma$ and draw embedded intervals $\gamma_1,\gamma_2$.

\begin{figure}
    \centering
    \includegraphics[width=\textwidth]{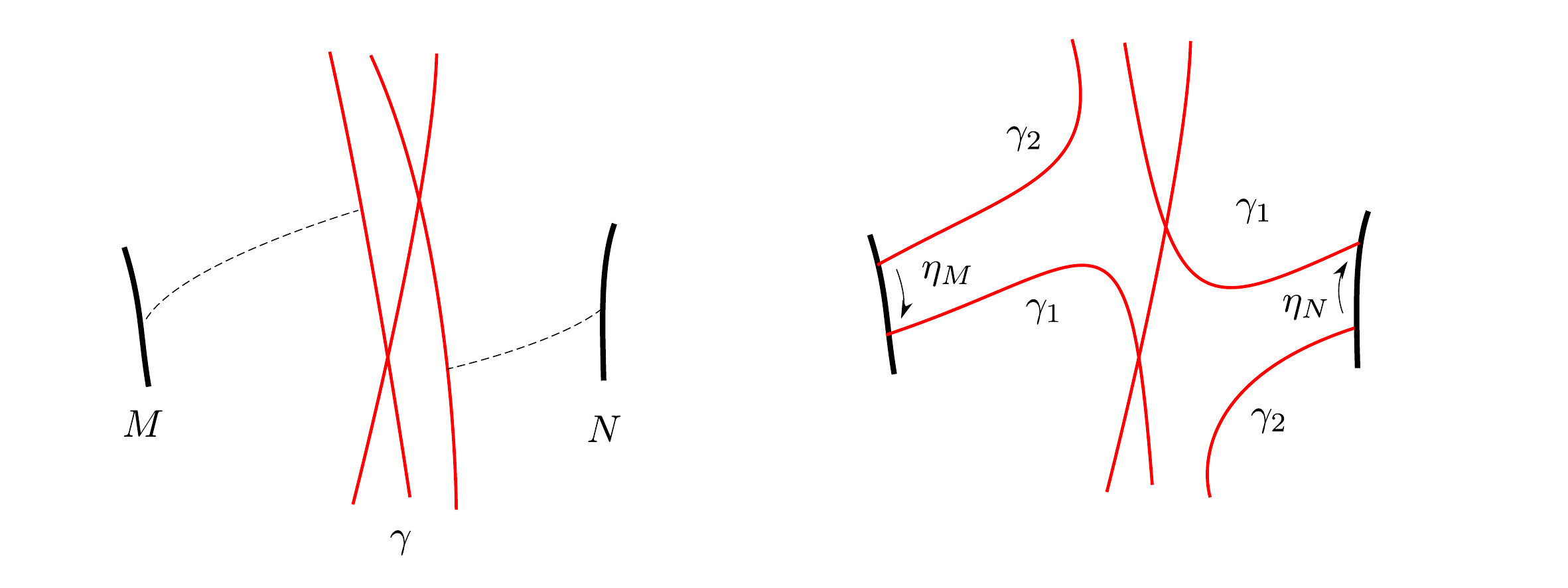}
    \caption{Replacing the circle object supported on $\gamma$ by two interval objects, with extension morphisms $\eta_M,\eta_N$.}
    \label{fig:circletointerval}
\end{figure}

Suppose instead that the embedded circle $\gamma$ separates $\Sigma$ into two connected components $\Sigma_l,\Sigma_r$. We argue that each of those components must have at least one boundary circle. Say $\Sigma_l$ does not have a boundary (apart from the newly created $\gamma$). Then it has the topological type of the complement $\Sigma_g \setminus D$ of a disk in a genus $g$ surface. But the simple closed curve giving the boundary of such a surface is never gradable, so this curve does not support an object.

We endow these intervals with local systems with equal rank to the local system in $X$, and choose the extension maps $\eta_M,\eta_N$ corresponding to directed boundary paths along $M,N$, respectively, such that the induced monodromy given by $\eta_M \circ \eta_N$ is conjugate to the monodromy of $X$. This gives condition (1); conditions (2) and (3) follow from the lack on intersections between $\gamma_1,\gamma_2$ in the interior.

Suppose now that $\gamma$ has non-trivial self-intersections. We choose a cover $\hat\Sigma \to \Sigma$ to which $\gamma$ lifts to an embedded circle $\hat\gamma$, and find marked boundaries to each side of $\hat\gamma$. Their image in $\Sigma$ gives marked boundaries $M,N$ to each side of $\gamma$; they might be identical if $\gamma$ is non-separating.

Now we look at the complement $\Sigma \setminus \gamma$. This a disjoint union of many open surfaces, whose boundaries are pieces of $\gamma$. We look at the components containing $M$ and $N$, and pick paths connecting them to $\gamma$ on either side, as shown in Figure \ref{fig:circletointerval}.

This gives the desired objects as in the embedded case, and satisfies conditions (2) and (3) since the paths we chose from $M,N$ to $\gamma$ do not intersect $\gamma$ in their interior, so in modifying $\gamma$ into $\gamma_1,\gamma_2$ no intersections are created or cancelled.
\end{proof}

Now let $X$ and $Y$ be rank one objects as above, supported on a pair $(\gamma_X,\gamma_Y)$ in minimal position. The following two lemmas appear to be well-known facts about Fukaya categories of surfaces, or at least used implicitly in some of the literature, but we include their proofs here for completeness.
\begin{lemma}\label{lemma:morphisms}
There is a basis of $\Hom(X,Y)$ whose elements are in bijection with the points in $\gamma_X \cap' \gamma_Y$.
\end{lemma}
\begin{proof}
The case for embedded curves in surfaces without boundary is \citep[Cor.2.11]{auroux2020fukaya}. Let us argue three cases in sequence: when $\gamma_X$ and $\gamma_Y$ are intervals, when one is an interval and the other is a circle, and when both are circles.

Suppose $\gamma_X$ and $\gamma_Y$ are both intervals. In that case, we can consider the universal cover $\hat\Sigma \to \Sigma$ (as a graded marked surface) and pick lifts $\hat\gamma_X$ and $\hat\gamma_Y$ (as graded curves). As in \citep[Sec.5.5]{HKK}, there is an equivalence
\[ \Hom_{\cF(\Sigma)}(X,Y) = \bigoplus_{g\in \pi_1(\Sigma)}\Hom_{\cF(\hat\Sigma)}(\hat\gamma_X,g\cdot \hat\gamma_Y). \]
Since the cover $\hat\Sigma\to\Sigma$ is a local homeomorphism, if $\gamma_X,\gamma_Y$ are in minimal position, then so are $\hat\gamma_X$ and $g\cdot \hat\gamma_Y$ for any $g$, so they are embedded intervals intersecting once, and the result follows from a calculation on the disc with four marked boundaries, whose category is equivalent to the (derived) representation category of the $A_3$ quiver.

Suppose now that $\gamma_X$ is an immersed circle and $\gamma_Y$ is an immersed interval. We take the annular cover $\hat\Sigma\to \Sigma$, on which $\gamma_X$ lifts to an embedded circle $\hat\gamma_X$. We again pick a lift $\hat\gamma_Y$ of $\gamma_Y$ and consider the intersections between the circle $\hat\gamma_X$ and the intervals $g\cdot\hat\gamma_Y$. Since they intersect minimally, they either intersect only once or not at all; in either case the result follows from a calculation on the annulus. The case where $\gamma_X$ is an immersed interval and $\gamma_Y$ is an immersed circle is analogous.

The most complicated case is when $\gamma_X$ and $\gamma_Y$ are both embedded circles; in which case they do not easily lift to the same cover as in the two previous cases. We use Lemma \ref{lem:circletointerval} to split $Y$ into an extension of two interval objects extended at marked boundaries $M,N$. We can choose the resulting intervals $\gamma_1, \gamma_2$ to intersect $\gamma_Y$ transversely some number of times, as indicated in Figure \ref{fig:threecolors}. Let us label the generating degree one morphisms along $M,N$ as $m,n$; then the object $Y$ can be recovered by a distinguished triangle
\[ Y_2 \to Y \to Y_1 \xrightarrow{Am + Bn} \dots \]
in $\cF(\Sigma)$ where $A,B \in \KK^*$; note that $AB$ is the monodromy of the rank one local system on $Y$.

\begin{figure}
    \centering
    \includegraphics[width=\textwidth]{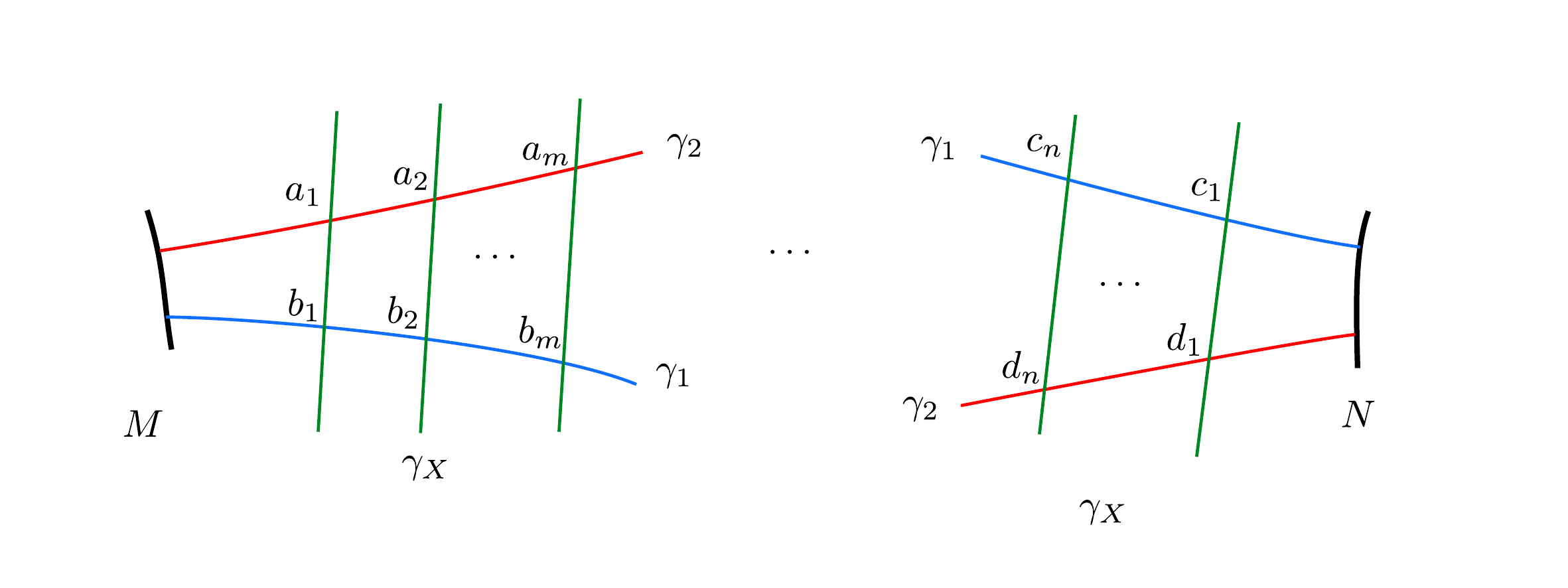}
    \caption{The two ends of the immersed intervals $\gamma_1, \gamma_2$. These intervals are chosen to intersect $\gamma_X$ transversely, at additional points $a_i,b_i,c_i,d_i$, in addition to the already-existing intersections between $\gamma_X,\gamma_Y$.}
    \label{fig:threecolors}
\end{figure}

Note that in comparison to $\gamma_X \pitchfork \gamma_Y$, there is now an even number $2(m+n)$ of new intersections in $(\gamma_X \cap \gamma_1) \cup (\gamma_X \pitchfork \gamma_2)$, which we label
\[ a_1, \dots, a_m, b_1,\dots, b_m, c_1, \dots, c_n, d_1, \dots, d_n \]
as in Figure \ref{fig:threecolors}.

We already know the statement of the lemma for $\Hom(X,Y_1)$ and $\Hom(X,Y_2)$ since $Y_1,Y_2$ are interval objects. We now apply $\Hom(X,-)$ to the distinguished triangle composed of $Y,Y_1,Y_2$ to get a distinguished triangle
\[ \Hom(X,Y_2) \to \Hom(X,Y) \to \Hom(X,Y_1) \xrightarrow{(Am+Bn) \circ} \dots \]
Note that the extension map on Hom-spaces is given by postcomposition with the degree one morphism $Am+Bn \in \Ext^1(Y_1,Y_2)$. The claim then follows for $\Hom(X,Y)$ by the calculation that the only non-trivial compositions are given by the `triangles'
\[ m \circ a_i = b_i \qquad n \circ c_i = d_i \]
appearing in the figure above; by the assumption of no embedded bigons between the original $\gamma_X,\gamma_Y$, those are the only triangles which appear with $M$ and $N$ at a vertex.
\end{proof}

The argument above also gives the degrees of the generators of the morphism space $\Hom_{\cF(\Sigma)}(X,Y)$. Each generalized intersection $p \in \gamma_X \cap \gamma_Y$ carries an intersection index $i_p(X,Y) \in \ZZ$ \citep[Sec.2.1]{HKK}; the corresponding morphism $p$ in $\Hom_{\cF(\Sigma)}(X,Y)$ sits in degree $i_p(X,Y)$.

We now prove a relation between extensions of indecomposable objects and intersections between their representing admissible (possibly immersed) curves.
\begin{lemma}
Let $X,Y$ be two objects with rank one as above, supported on a pair of immersed curves $(\gamma_X,\gamma_Y)$ in minimal position, and $p \in \gamma_X \cap \gamma_Y$ be an intersection point with index $i_p(X,Y) = 1$. Then there is an object $Z \in \cF(\Sigma)$ together with a distinguished triangle
\[ Y \to Z \to X \overset{p}{\dashrightarrow} \]
such that $Z$ is supported on a (possibly disconnected) immersed curve obtained by smoothing $\gamma_X \cup \gamma_Y$ at $p$.
\end{lemma}
\begin{proof}
Let us first address the case where $X$ and $Y$ are supported on embedded intervals intersecting only once, and do not share any marked boundaries. In that case, $X, Y$ and the object $Z = Z_1 \oplus Z_2$ given by smoothing the intersection as in Figure \ref{fig:intervalstriangle} are in the image of an exact functor $\cF(\Delta_4) \to \cF(\Sigma)$, where $\Delta_4$ is the disk with 4 marked boundaries, and the result follows from a calculation there.

\begin{figure}
    \centering
    \includegraphics[width=0.4\textwidth]{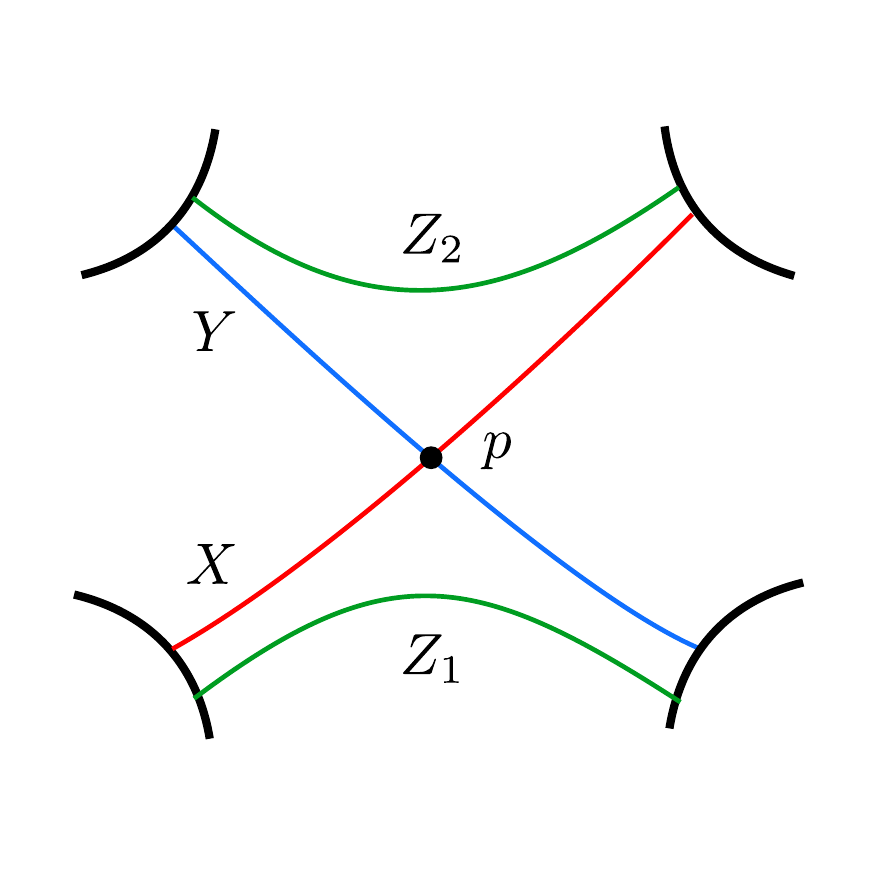}    \caption{The curves representing a distinguished triangle $Y \to Z \to X \xrightarrow{p} \dots$, where $Z = Z_1 \oplus Z_2$ is supported on two intervals disjoint from each other.}
    \label{fig:intervalstriangle}
\end{figure}

We then prove the case for immersed intervals, possibly with multiple intersections, by again lifting to the universal cover $\pi:\hat\Sigma\to\Sigma$. Let us fix two lifts $\hat\gamma_X,\hat\gamma_Y$ of the immersed curves; there is then some group element $g \in \pi_1(\Sigma)$ such that $\hat\gamma_X$ and $g \cdot \hat\gamma_Y$ intersect at a unique point $\hat p$, which is a lift of $p$; by the disk calculation we have a distinguished triangle $\hat Y \to \hat Z \to \hat X \dashrightarrow$, corresponding to some functor $\cD \to \cF(\hat\Sigma)$, which we compose down to $\cF(\Sigma)$ to get the desired object as the image of $\hat Z$.

Now we move on to the case where $\gamma_X$ is an immersed circle, and $\gamma_Y$ is an immersed interval. We lift to the annular cover $\hat\Sigma' \to \Sigma$ with respect to $\gamma_X$; then $\gamma_X$ lifts to an embedded circle $\hat\gamma_X$. Let us pick a lift $\hat\gamma_Y$ of $\gamma_Y$ which intersects $\gamma_X$ at a lift $\hat p$ of $p$. By the minimality assumption, there are no bigons, and the only intervals that satisfy these conditions in the annulus are embedded intervals intersecting $\hat\gamma_X$ exactly once; the result then follows from a calculation on the annulus $\Delta^*_{1,1}$ with one marked boundary on each boundary circle, whose category is equivalent to $D^b(\mathrm{Coh}(\PP^1))$.

The remaining case is where $\gamma_X$ and $\gamma_Y$ are both immersed circles. We use Lemma \ref{lem:circletointerval} to split $Y$ into two interval objects $Y_1,Y_2$, supported along $\gamma_1,\gamma_2$ extended at two common marked boundaries $M,N$. Let $\lambda_X,\lambda_Y \in \KK^*$ be the monodromies of the rank one local systems associated to $X,Y$. Let $Z$ be the circle object supported along the `smoothing at $p$' $\gamma_Z$ with rank one local system of monodromy $\lambda_Z = \lambda_X \lambda_Y$.

Suppose without loss of generality that $\gamma_X$ intersects $\gamma_1$ near $p$ at a point $q$. We then consider the interval object $Z'$ given by the distinguished triangle
\[ Y_1 \to Z' \to X \xrightarrow{q} \dots \]
By the calculation in the annulus above, the object $Z'$ is supported on an immersed interval $\gamma_{Z'}$ obtained by smoothing at $q$. We can then arrange the objects as indicated in Figure \ref{fig:complicated}.

\begin{figure}
    \centering
    \hspace*{-1.2cm}
    \includegraphics[width=1.2\textwidth]{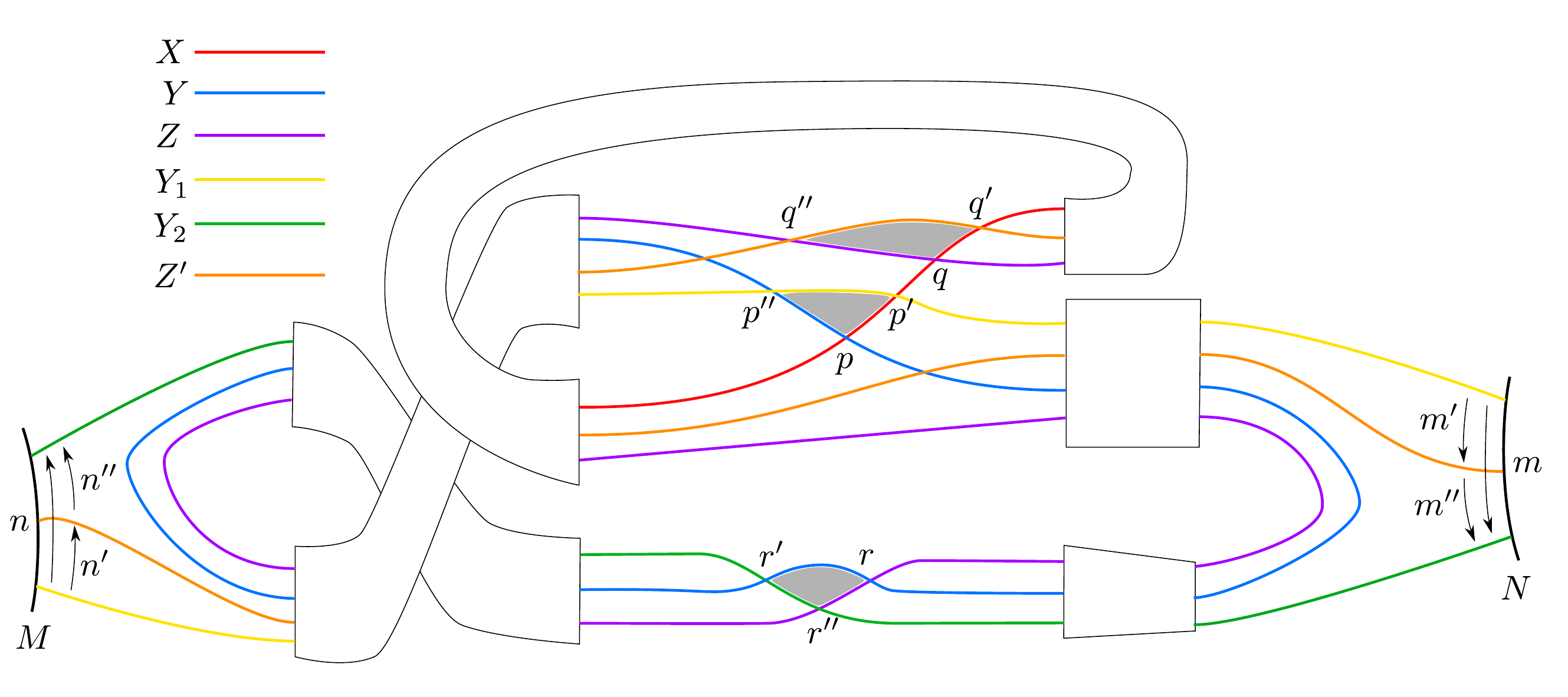}
    \caption{The intersections and boundary paths appearing in the distinguished triangles involving the immersed circle objects $X,Y,Z$ and interval objects $Y_1,Y_2,Z'$. The three shaded triangles correspond to the three compositions appearing in the octahedral diagram below. The five white strips represent embedded rectangles in the surface along which some number of strands run parallel, without crossing each other; note that they are allowed to cross strands in other white strips, depending on how the immersed circles $\gamma_X,\gamma_Y$ intersect.}
    \label{fig:complicated}
\end{figure}

We can fit the morphisms corresponding to the intersection points and marked boundaries in the following octahedron:
\[\xymatrix{
                                            & X \ar@{-->}[ld]_{p} \ar@{-->}[ldd]^{p'}   & \\
Y \ar[rr]^{r} \ar[d]_{p''}    &                                           & Z \ar[lu]_{q} \\
Y_1 \ar[rr]_{\lambda_Z m'+n'} \ar@{-->}[rd]_{\lambda_Y m+n}     &                                           & Z' \ar@{-->}[ld]^{\lambda_X m''+n''} \ar[u]_{q''} \ar[luu]^{q'}\\
                                            & Y_2 \ar[uur]^{r''} \ar[luu]_{r'}                & \\
}\]
In the diagram above the dotted arrows indicate morphisms of degree $+1$. Each face of the octahedron is given by a calculation in the annulus; we have distinguished triangles
\[ Y_2 \to Y \to Y_1 \to \dots, \quad Z' \to Z \to Y_2 \to \dots, \quad Y_1 \to Z' \to X \to \dots \]
and compositions
\[ p' \simeq p'' \circ p, \quad r'' \simeq r \circ r', \quad q' \simeq q\circ q''. \]
We deduce the Lemma from the octahedral axiom satisfied by the triangulated category $\cF(\Sigma)$.
\end{proof}

\section{Lemmas about stability conditions}\label{sec:lemmas}
In this section we collect some lemmas about (Bridgeland) stability conditions on general triangulated categories, and also about the specific case where $\cD = \cF(\Sigma)$ is the Fukaya category of a marked surface $\Sigma$.

As usual, $\Stab(\cD)$ will denote the space of stability conditions on a triangulated category $\cD$ satisfying the so-called \emph{support property} \citep{KS,Bay2} (in the original paper \citep{Bri1} these are called full locally finite stability conditions). In all of our cases, $K_0(\cD)$ is finite-rank so we will use the lattice $\Lambda = K_0(\cD)$. As shown by Bridgeland, $\Stab(\cD)$ has the structure of a $(\rk K_0(\cD))$-dimensional complex variety.

In \citep{HKK}, it is shown that one can construct stability conditions on $\cF(\Sigma)$ using flat surfaces, or equivalently, quadratic differentials with exponential-type singularities. Namely, there is a moduli space $\cM(\Sigma)$ whose points are roughly pairs $(X,\phi)$ of a Riemann surface $X$ diffeomorphic to $\Sigma$ and a quadratic differential $\phi$ on $X$ with certain singularities prescribed by the data of the marked boundary $\MM$ and the line field $\eta$. The horizontal foliation of $\phi$ gives $\Sigma$ the structure of a flat surface with conical singularities, of both finite and infinite angular defect. In our case of fully stopped surfaces, all these singularities will be exponential-type singularities, or equivalently, they will only have infinite-angle conical singularities.

\subsection{Stability conditions and genericity}
Let us precisely state our genericity assumptions, which will play an important role in later proofs. We first recall the \emph{support property} \citep{KS,Bay2}:
\begin{definition}\label{def:supportproperty}
A stability condition $\sigma = (Z,\cP)$ satisfies the support property if
\[ \inf_{0 \neq X \mathrm{semistable}} \frac{|Z(X)|}{\lVert [X] \rVert} = C > 0,\]
where $\lVert\cdot\rVert$ is a norm on $\Lambda \otimes \RR$.
\end{definition}

From now on, we will only consider stability conditions satisfying the support property. The space $\Stab(\cD)$ of such stability conditions is a complex manifold and the map $\Stab(\cD) \to \Hom_\ZZ(\Lambda,\CC)$, given by forgetting the slicing $\cP$, is a local homeomorphism \citep{Bri1}. To express genericity we need to define walls in this space, following \citep{bridgeland2015quadratic}. Let us fix a class $\gamma \in \Lambda$, and consider other classes $\alpha$ such that $\alpha$ and $\gamma$ are not both multiples of the same class in $\Lambda$.
\begin{definition}
The wall $W_\gamma(\alpha) \subset \Stab(\cD)$ is the subset of stability conditions such that there is a phase $\phi \in \RR$ and objects $A,G$ with respective classes $\alpha,\gamma$ such that $A \subset G$ in the abelian category $\cP_\phi$.
\end{definition}

Each wall $W_\gamma(\alpha)$ is contained within a codimension one subset of $\Stab(\cD)$ where $Z(\alpha)/Z(\gamma)$ is real, and we have the following local finiteness result:
\begin{lemma}\citep[Lemma 7.7]{bridgeland2015quadratic}
If $B\subset \Stab(\cD)$ is compact then for a fixed $\gamma$ only finitely many walls $W_\gamma(\alpha)$ intersect $B$.
\end{lemma}

Note that this is not true if we consider the whole collection of walls for all $\gamma$; the union of all walls for all classes $\alpha$ can be dense in $\Stab(\cD)$.
\begin{definition}\label{def:xistable}
Let $\Xi \subset \Lambda$ be a finite subset of classes. Take the union
\[ W_\Xi = \bigcup_{\gamma \in \Xi, \alpha \in \Lambda} \overline{W_\gamma(\alpha)} \]
of all closures of walls for classes in $\Xi$; we will say a stability condition $\sigma$ is  $\Xi$-generic if $\sigma \in \Stab(\cD) \setminus \bar W_\Xi$.
\end{definition}

By local finiteness, $W_\Xi$ is a locally-finite union of closed subsets so $\Xi$-genericity is an open condition. The connected components of $\Stab(\cD) \setminus W_\Xi$ will be called the $\Xi$-chambers.

\subsection{Stability conditions on Fukaya categories of surfaces}
Now we turn our attention to stability conditions on the categories $\cF(\Sigma)$. Importantly, we do not make \emph{a priori} assumptions about whether such stability conditions are describable by flat surfaces; the lemmas here follow from the general axioms of stability conditions and the properties of such categories.

An important role will be played by objects that can be represented by embedded curves. Let us from now let us say an object is a \emph{(embedded) interval object} if it can be represented by an (embedded) interval, and a \emph{(embedded) circle object} if it can be represented by an (embedded) circle. Note that any interval object, embedded or not, must carry a trivial rank one local system, if it is indecomposable.

We now prove the following proposition, which constrains the type of geometric objects. This will play an important role throughout this paper.
\begin{proposition}\label{prop:stableobjects}
For any stability condition $\sigma \in \Stab(\cF(\Sigma))$, every stable object is either an embedded interval object or an embedded circle object.
\end{proposition}

\begin{proof}
Since $L$ is indecomposable its support cannot have more than one connected component. Thus the only objects we have to rule out are objects whose representatives all have self-intersections; we will call these \emph{truly immersed objects}.

A stable object $L$ must have $\Ext^i(L,L) = 0$ for $i < 0$. Let $L$ be a truly immersed objects and pick a representative of $L$ which bounds no generalized bigons, supported on an immersed curve $\gamma_L$. Perturbing $L$ to calculate endomorphisms, we see that a self-intersection point $p$ of $\gamma_L$ contributes classes to $\Ext^*(L,L)$ in degrees $i_p$ and $1-i_p$, where $i_p$ is the degree of intersection at $p$. These classes are nonzero by minimality of self-intersections, so if there is a self-intersection point with $i_p \neq 0,1$, one of these degrees is negative and therefore $L$ cannot be semistable.

\begin{figure}[h]
    \centering
    \includegraphics[width=0.8\textwidth]{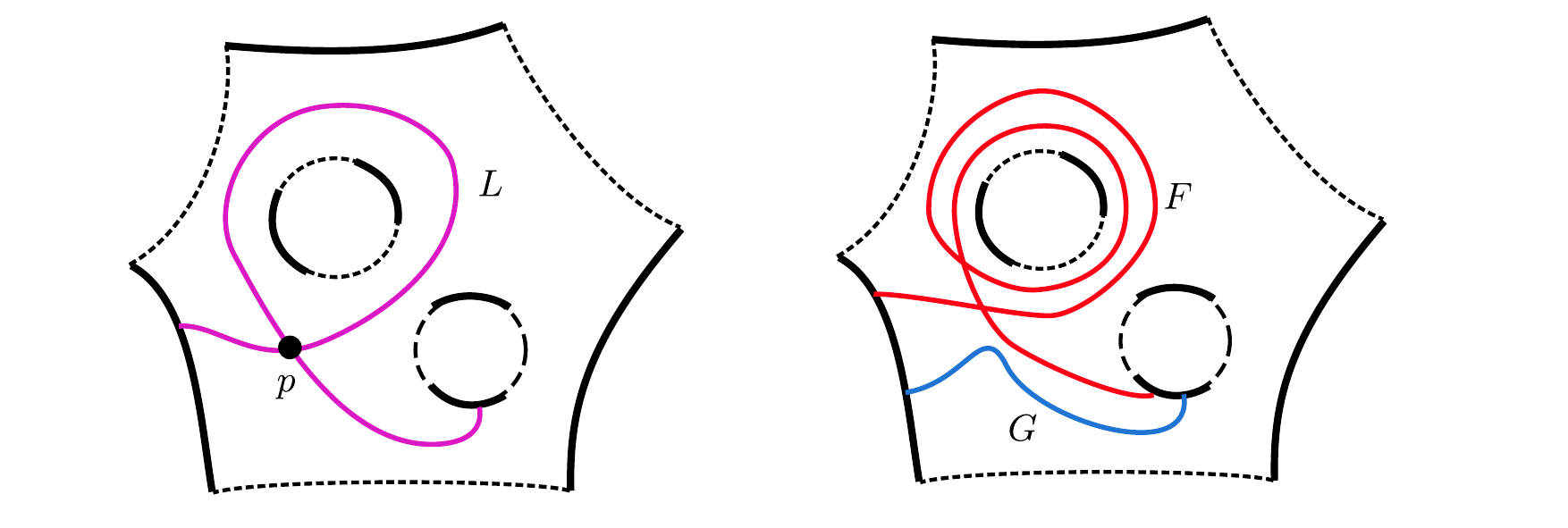}
    \caption{A truly immersed Lagrangian $L$. The self-extension $L \to E \to L$ at the self-intersection point $p$ splits as a direct sum $E = F \oplus G$.}
\end{figure}

The only case left to consider is when $\gamma_L$ only has self-intersection points of degree $0$ and $1$; each one of these points gives nonzero classes in $\Hom(L,L)$ and $\Ext^1(L,L)$. Let us pick one of these points $p$, and consider the corresponding nontrivial extension $L \to E \to L$. Note that the support of $E$ is given by two superimposed curves so we have a direct sum decomposition $E = F \oplus G$. But by assumption $L \in \cP_\phi$ for some $\phi$, and by \citep[Lemma 5.2]{Bri1} each $\cP_\phi$ is abelian, so $E, F, G \in \cP_\phi$ as well. Since the stability condition is locally finite, the abelian category $\cP_\phi$ is finite length; therefore the Jordan-H\"older theorem applies \citep{joyce2006configurations}. Since the length of $E$ is 2, $F$ and $G$ are length one, and by uniqueness of the simple objects in the Jordan-H\"older filtration (up to permutations) we must have $F \cong G \cong L$. But this is impossible because $E$ is a nontrivial extension so $E \neq L \oplus L$.
\end{proof}

\begin{remark}\label{rem:annulus}
Note that the proof above does not preclude a self-intersecting object $L$ from being \emph{semistable}; it just cannot be simple in $\cP_{\phi_L}$. In fact this even happens generically: take $\Sigma$ to be the annulus with one marked interval on each boundary circle and grading such that the nontrivial embedded circle is gradable; by mirror symmetry the category $\cF(\Sigma)$ is equivalent to $D^b(\Coh(\PP^1))$. Under this equivalence, the rank one circle object with monodromy $z \in \CC^\times$ gets mapped to the skyscraper sheaf $\CC_z$ on $\PP^1$, and the interval object $I$ with both ends on the outer boundary, wrapping the annulus once, gets mapped to the skyscraper sheaf $\CC_\infty$ on $\PP^1$.

The space of stability conditions on this category is known to be isomorphic to $\CC^2$ as a complex manifold \citep{Oka}, and there is a geometric (top dimensional) chamber in $\Stab(\PP^1)$ where all the rank one skyscraper sheaves are stable. In particular, the nontrivial extension $I \to L \to I$, represented by an immersed Lagrangian with one self-intersection as in Figure \ref{fig:annuluslag}, is semistable. So self-intersecting semistable objects do appear generically, that is, on open loci in stability spaces, but they always must have Jordan-H\"older decompositions into embedded objects.
\end{remark}

\begin{figure}[h]
    \centering
    \includegraphics[width=0.9\textwidth]{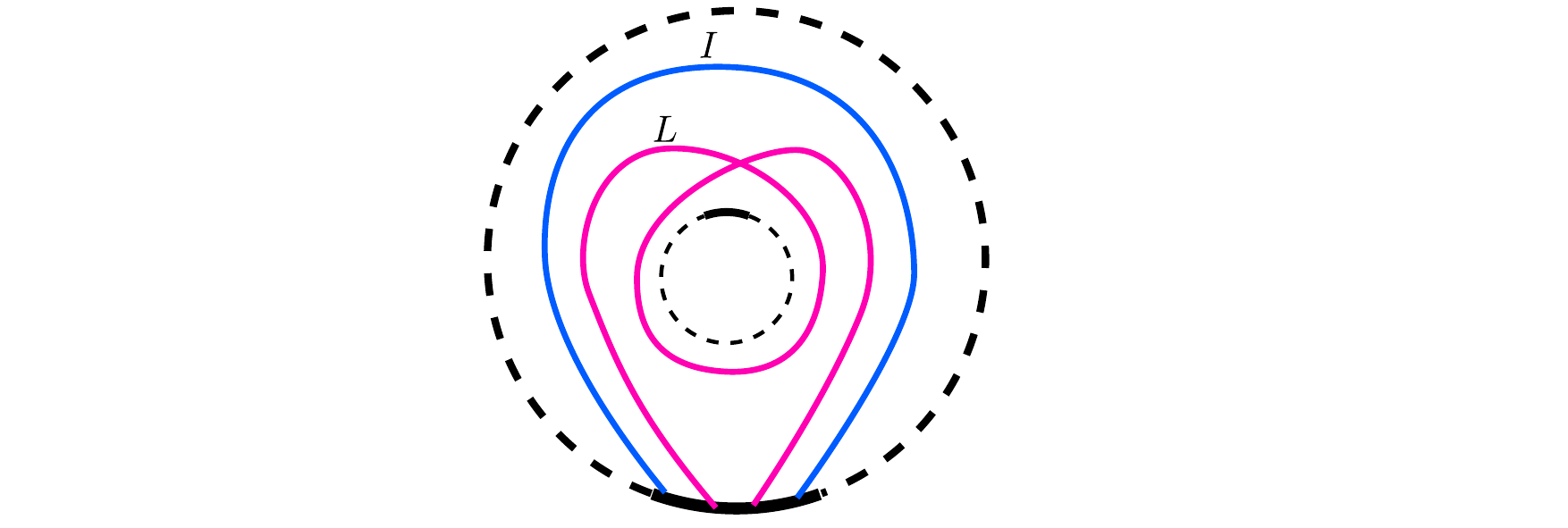}
    \caption{The annulus mirror to $D^b(\Coh(\PP^1))$. For a geometric stability condition on $\PP^1$, the truly immersed object $L$ (corresponding to an irreducible rank 2 skyscraper sheaf $\cO_{x^2}$) is semistable.}
    \label{fig:annuluslag}
\end{figure}

The result above characterizes which objects can be stable, namely embedded intervals and embedded circles with indecomposable local systems. It turns out that similar index computations also allows us to constrain the form of the HN decompositions of objects.

\begin{definition}(Chain of stable intervals)
Let us fix a stability condition $\sigma \in \Stab(\cF(\Sigma))$ and consider an indecomposable object $X$ in $\cF(\Sigma)$. We say that $X$ has a chain of stable intervals decomposition (\textsc{cosi} decomposition) under $\sigma$ if there is
\begin{itemize}
    \item A sequence of stable (therefore embedded) interval objects $X_1,\dots,X_N$, with respective phases $\phi_1,\dots,\phi_N$ and a sequence of marked boundary intervals $M_0,\dots,M_N$, where the support $\gamma_i$ of the object $X_i$ has ends on $M_{i-1}$ and $M_i$,
    \item Extension morphisms $\eta_i \in \Ext^1(X_i,X_{i+1})$ when $\phi_i \le \phi_{i+1}$  or $\eta_i \in \Ext^1(X_{i+1},X_i)$ when $\phi_i \ge \phi_{i+1}$ corresponding to the shared $M_i$ marked boundary (including an extension at $M_0 = M_N$ if $X$ is a circle object),
\end{itemize}
such that the iterated extension by all the $\eta_i$ is isomorphic to $X$.
\end{definition}

\begin{figure}[h]
    \centering
    \includegraphics[width=0.8\textwidth]{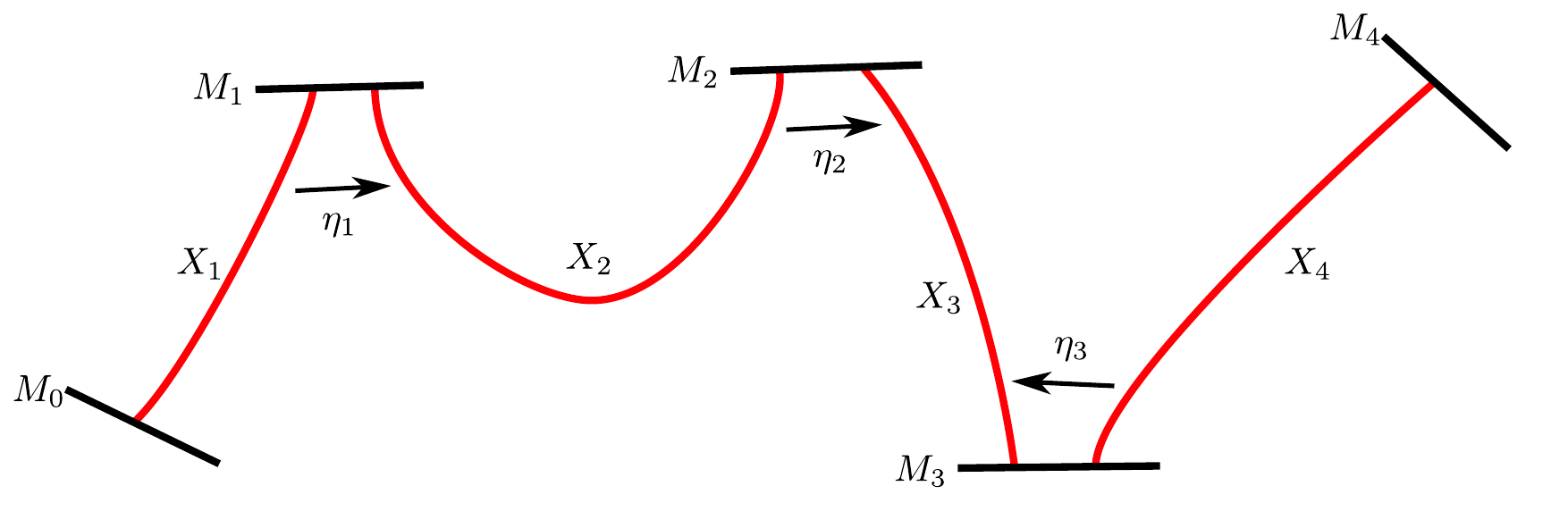}
    \caption{A chain of stable intervals with N=4.}
\end{figure}

\begin{remark}
Note that the order $X_1,\dots,X_N$ here is not the ordering of semistable objects in the HN decomposition of $X$: the extension maps are allowed to go either way, and the same phase may appear several times.
\end{remark}

Note that if $X$ has a \textsc{cosi} decomposition then its HN decomposition can be produced from it by grouping together all stable interval objects of the same phase.

\begin{lemma}
If $X$ has a \textsc{cosi} decomposition under $\sigma$, then it is essentially unique, ie. the sets $\{X_i\}$ and $\{M_i\}$ are uniquely defined up to isomorphism.
\end{lemma}
\begin{proof}
Follows from the uniqueness of the HN filtration and the uniqueness (up to permutation) of the Jordan-H\"older filtration on each finite-length abelian category $\cP_\phi$.
\end{proof}

This decomposition also captures the isotopy class of the object $X$. Let us produce an immersed curve $\gamma$ from this data as follows: for each $i$, if the extension map $\eta_i$ belongs to $\Ext^1(X_i,X_{i+1})$ we connect $\gamma_i$ to $\gamma_{i+1}$ counterclockwise (ie. by a boundary path following $M_i$ and keeping $\Sigma$ to the right), and if $\eta_i \in \Ext^1(X_{i+1},X_i)$ we use the corresponding clockwise path from $\gamma_i$ to $\gamma_{i+1}$. From the geometricity result (Theorem \ref{thm:geometricity}) we can deduce that the object $X$ can be represented by the curve $\gamma$ endowed with an appropriate local system.

The following lemma will be central to our proofs later, and essentially means that \textsc{cosi} decompositions are not allowed to cross each other. From now on, we will leave the extension morphisms implicit and denote a \textsc{cosi} decomposition by its stable intervals.

\begin{lemma}\label{lemma:noncrossing}
Let $X$ and $Y$ be two objects with respective \textsc{cosi} decompositions $(X_1,\dots,X_m)$ and $(Y_1,\dots,Y_n)$. We choose representatives in minimal positions for every pair of those stable intervals. Then on the surface $\Sigma$ there are none of the following arrangements:
\begin{enumerate}
    \item Polygons bounded by the two chains and two transversal crossings between stable intervals.
    \item Polygons bounded by the two chains and two common marked boundary intervals (with boundary paths inside the polygon).
    \item Polygons bounded by the two chains, one transversal crossing and one common marked boundary interval.
\end{enumerate}
\end{lemma}

\begin{figure}[h]
    \centering
    \includegraphics[width=\textwidth]{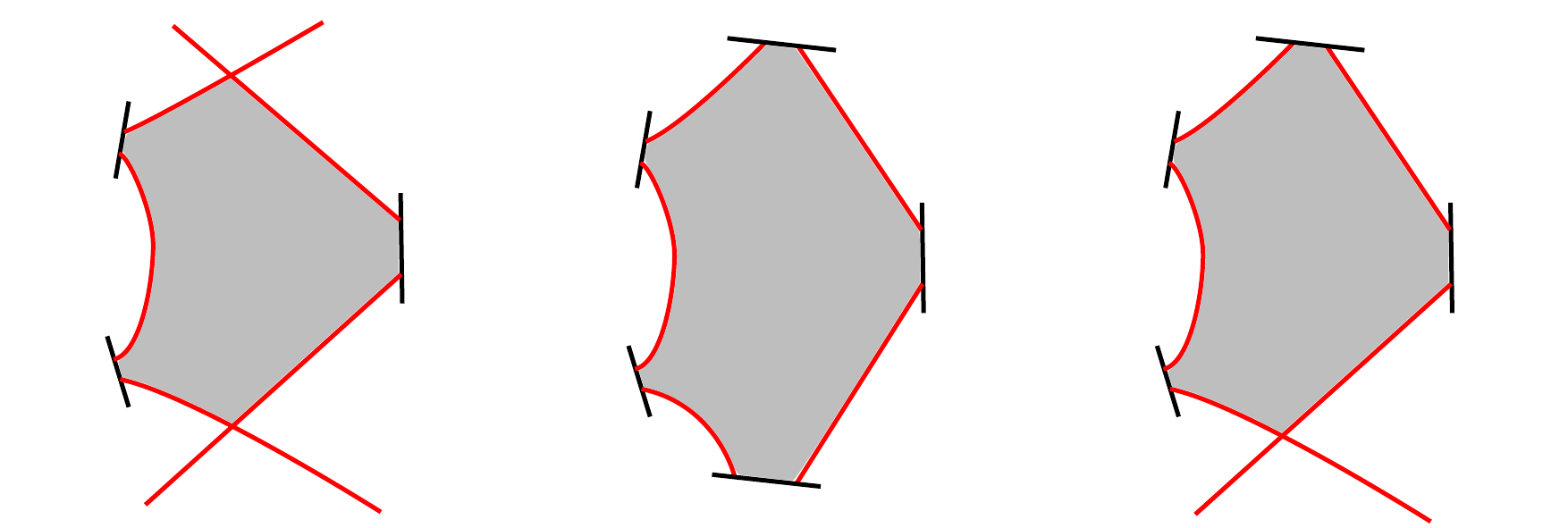}
    \caption{The three kinds of polygons of stable intervals that cannot appear by Lemma \ref{lemma:noncrossing}. Here we have polygons with $k=3$ sides on the left and $l = 2$ sides on the right. The shaded interior means that these polygons bound disks inside of $\Sigma$.}
\end{figure}

\begin{figure}[h]
    \centering
    \includegraphics[width=\textwidth]{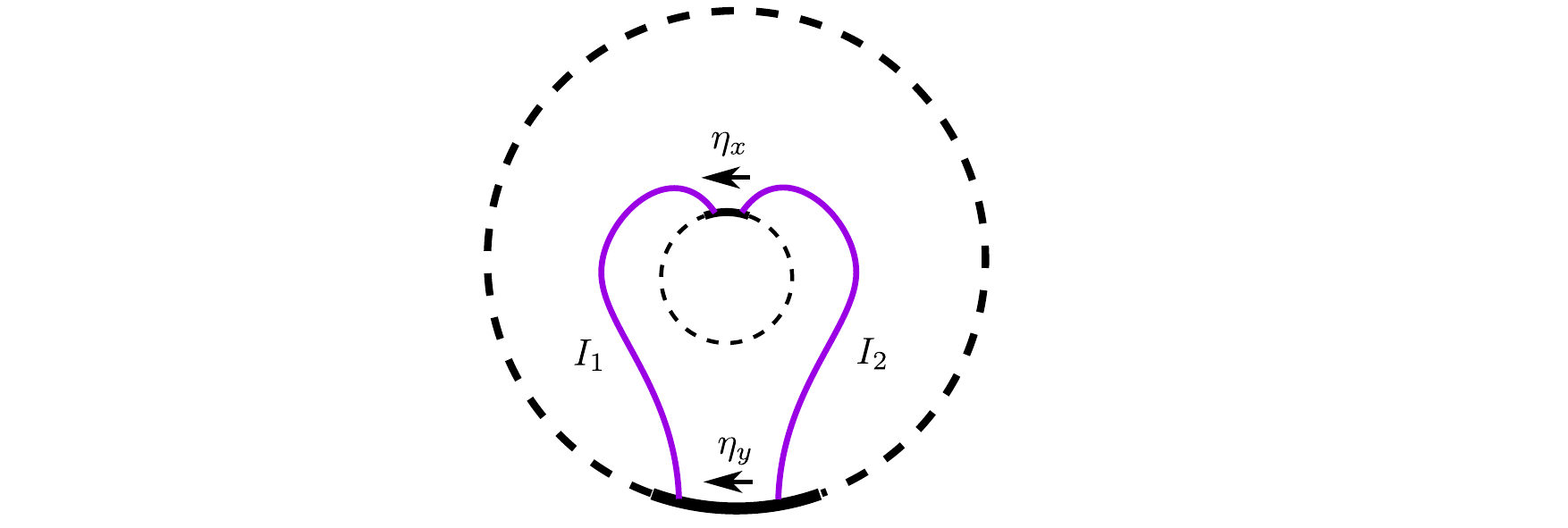}
    \caption{The annulus $\Delta^*_{(1,1)}$ mirror to $\PP^1$. Under this correspondence $\cF(\Delta^*_{(1,1)}) \cong D^b(\Coh(\PP^1))$ we have $I_1 \cong \cO(1)$ and $I_2 \cong \cO$ with $\Hom(I_2,I_1)$ spanned by $\eta_x,\eta_y$. Note that is not a counterexample to case (2) of Lemma \ref{lemma:noncrossing} since the intervals do not bound an embedded polygon.}
\end{figure}

\begin{proof}
Let us first argue that it is sufficient to prove the statement for adequately generic $\sigma$. By standard arguments, the locus of $\Stab(\cD)$ in which the all the objects $X_i,Y_i$ are stable is open. Consider now the collection $\Xi \subset \Lambda$ containing all the classes of these objects; the corresponding union of walls $\bar W_\Xi$ is a locally-finite union of closed subsets of positive codimension. So we can find some other stability condition $\sigma'$, arbitrarily close to $\sigma$, where $X_i,Y_i$ still give \textsc{cosi} decompositions of $X,Y$, and where the phases of any $X_i$ and $Y_j$ are pairwise distinct when $[X_i]$ and $[Y_j]$ are not proportional. If the noncrossing statement of the lemma is true for $\sigma'$ it is also true for $\sigma$.

We start with the first type of polygon. Assume the polygon has $k$ edges on the right and $l$ edges on the left, and for ease of notation we label the intervals in this polygon starting by 1 on both sides. Without loss of generality shift the grading of $X$ such that the intersection point $p$ has index $i_p(X_1,Y_1) = 1$. By minimality of crossings $p$ contributes nonzero classes in $\Ext^1(X_1,Y_1)$ and in $\Hom(Y_1,X_1)$. Since both are stable objects, this implies that
\[ \phase(Y_1) \le \phase(X_1) \le \phase(Y_1) + 1. \]

Smoothing out each one of the chains of intervals separately, one gets a bigon with vertices at $p$ and $q$; the existence of the embedded bigon constrains the index of $q$ to be $i_q(X_k,Y_l) = 0$, and by the same argument we have
\[ \phase(X_k) \le \phase(Y_l) \le \phase(X_k) + 1. \]

By assumption, all the other vertices of this polygon give, on the left hand side, extension maps $X_i \xrightarrow{+1} X_{i+1}$, and on the right hand side, extension maps $Y_{i+1} \xrightarrow{+1} Y_i$. Since all these maps appear in HN decompositions we must have the following inequalities between phases
\begin{align*} \phase(X_i) & \le \phase(X_{i+1}) \text{\ for all\ } 1 \le i \le k-1, \\ \phase(Y_j) & \ge \phase(Y_{j+1}) \text{\ for all\ } 1 \le j \le l-1,
\end{align*}
that, together with the previous inequalities, imply that the phases are all equal. But since we excluded the degenerate polygons, at least two of the $K_0$ classes of this object these objects are not multiples of the same class so by $\Xi$-genericity of $\sigma'$ they have distinct phases. The three other cases are proven by small variations of this same argument.
\end{proof}

\begin{remark}
Note that the two chains might still share a common stable interval; this is not ruled out by the argument above and in fact happens generically. Similarly, note that our definition of chain-of-intervals decomposition above does not exclude the possibility that the chain of intervals overlaps with itself. Again, in the annulus example consider some algebraic stability condition such that the stable objects are two intervals $I_1,I_2$ connecting the outer and inner boundary, and consider the embedded interval object also connecting the two boundaries but wrapping around more times; this object has a \textsc{cosi} decomposition given by multiple copies of $I_1$ and $I_2$.
\end{remark}

Self-overlapping chains of intervals will pose some serious technical difficulties later on, so we will rule them out with the following criterion. Let $X$ be an indecomposable object with a \textsc{cosi} decomposition $(X_1,\dots,X_N)$, with $X_i$ supported on $\gamma_i$.

\begin{definition}
This is a \emph{simple} \textsc{cosi} decomposition if for each pair $\gamma_i, \gamma_j, i \neq i$, $\gamma_i$ and $\gamma_j$ are in pairwise distinct isotopy classes, and $\gamma_i \cap \gamma_j = \emptyset$.
\end{definition}
In other words, the decomposition is simple if the set of arcs given by all the $\gamma_i$ can be extended to an arc system on $\Sigma$. In general, objects will not have a simple \textsc{cosi} decomposition, but the following topological condition is sufficient.

\begin{lemma}\label{lemma:simplechain}
Let $X$ be an object with a \textsc{cosi} decomposition, supported on an embedded interval $\gamma$ separating the surface $\Sigma$ into two connected components, such that the two ends of $\gamma$ belong to distinct marked boundary intervals. Then $X$ has a simple \textsc{cosi} decomposition.
\end{lemma}

\begin{proof}
Let us write as before $\gamma_1,\dots,\gamma_N$ for the intervals and $M_1,\dots,M_{N-1}$ for the marked boundary intervals between them. We would like to rule out the possibility of having repeated intervals.

Suppose that the subsequence
\[ M_i, \gamma_{i+1}, M_{i+1}, \dots, M_{i+k-1}, \gamma_{i+k}, M_{i+k} \]
repeats itself, ie. all those intervals and marked boundary components are isomorphic to
\[ M_j, \gamma_{j+1}, M_{j+1}, \dots, M_{j+k-1}, \gamma_{j+k}, M_{j+k} \]
for some other $j$. For simplicity assume that $j > i+k$ so there's no overlap; and let us assume that $k$ is maximal. Let us also assume that $i > 0$ and $j+k < N$ so that we are in the middle of the chain and not at the ends, and that $j$ is the smallest index possible with these properties (because this sequence could in principle repeat many times).

There are then four possibilities for the extension maps at $M_i$ and $M_{i+k}$, as below:
\begin{figure}[h]
    \centering
    \includegraphics[width=\textwidth]{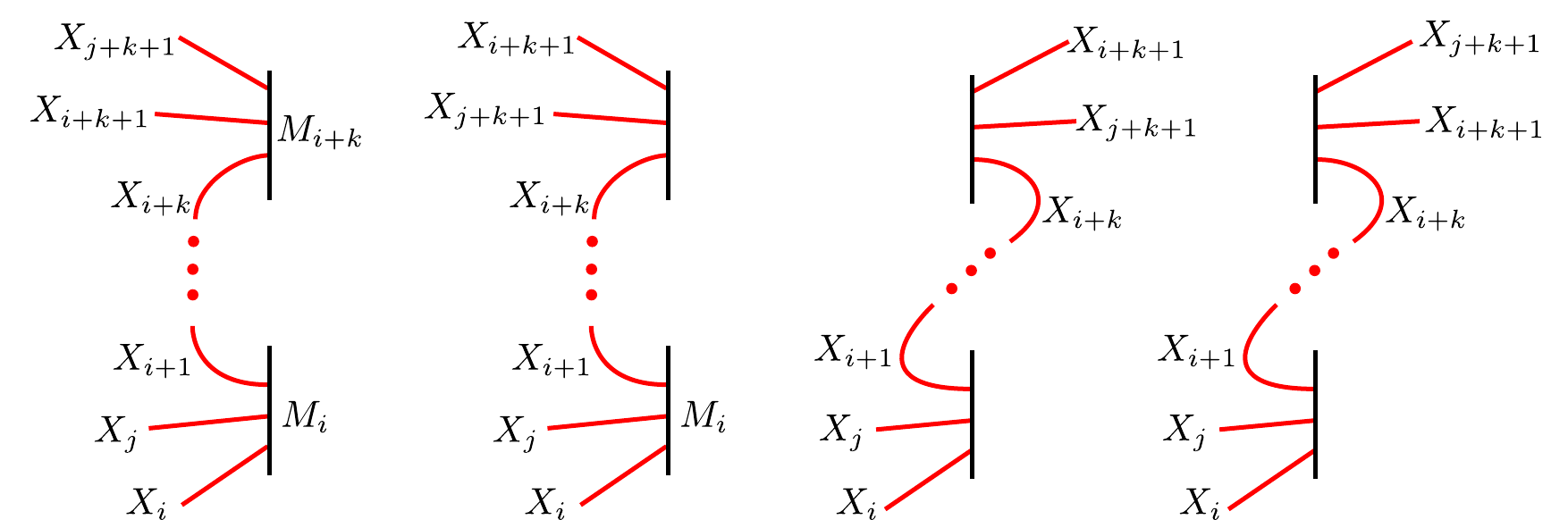}
    \caption{Four possible cases for extensions within a self-overlapping chain.}
\end{figure}

If we are in the first case or third case, note that concatenating the chain by those boundary walks leads to a self-crossing of $\gamma_X$. This self-crossing cannot be eliminated by isotopy, because due to Lemma \ref{lemma:noncrossing} there are no polygons of stable intervals bound by the chain. Since we assumed that $X$ is an embedded interval object this is impossible.

As for the second case and fourth case, note that concatenating the chain by those boundary walks leads to an embedded interval that does not separate the surface into two parts, contradicting the topological condition.

The special cases to be dealt with are when this repeated sequence is at one end of the chain; in this case it is easy to see that the concatenation is always non-trivially self-intersecting, unless the overlap is just a single boundary component $M_0 = M_N$ which we also excluded by assumption. The more general case of repeated intersections, nested intersections etc. poses no essential difficulties and can be argued by repeating the argument above recursively.
\end{proof}

With these lemmas, we prove the following proposition constraining the form of the HN decomposition of an object.

\begin{proposition}\label{prop:intervalchain}
Let $X$ be an rank one indecomposable object of $\cD = \cF(\Sigma)$ and $\sigma \in \Stab(\cD)$ any stability condition. Then $X$ is either a semistable circle or has a chain of stable intervals decomposition under $\sigma$.
\end{proposition}

\begin{proof}
Note that for any object, being stable is an open condition in stability space, therefore it is enough to assume that the stability condition $\sigma$ is appropriately generic (that is, for any finite set of classes $\Xi$ as in Definition \ref{def:xistable}).

Suppose first that $X$ is not a semistable circle. Consider the HN decomposition of $X$ under $\sigma$ and further decompose each semistable factor of phase $\phi$ using the Jordan-H\"older filtration on the abelian category $\cP_\phi$. We get then a total filtration
\[\xymatrix@=4pt{
0 \ar[rr] & 				& X_1 \ar[ld]^{\pi_1} \ar[rr] & 		& X_2 \ar[ld]^{\pi_2} \ar[r] 	& \dots \ar[r]	& X_{N-1}  \ar[rr]	&			& X_N = X \ar[ld]^{\pi_n} \\
& A_1 \ar@{.>}[lu]^{\epsilon_1=0}	& & A_2 \ar@{.>}[lu]^{\epsilon_2}	&						&				& & A_N \ar@{.>}[lu]^{\epsilon_n}	&
}\]
where each factor $A_i$ is stable but the phases $\phi_i$ might repeat.

We will prove by induction on the total length $N$. The case $N=1$ is obvious. Assume now that the statement is true for any object of total length $N-1$, and take an object $X$ as above.

Consider the extension $X_{N-1} \to X_N \to A_N$. Since the object $A_N$ is stable, by Lemma \ref{prop:stableobjects} it is either representable either by an embedded interval or an embedded circle. We will treat these cases separately.

If $A_N$ is an interval object supported on a embedded interval $\alpha_N$, and $X_{N-1}$ is supported on some collection of immersed curves $\gamma_{N-1}$. Note that we can also express $X_{N-1}$ as an extension \[A_N[-1] \to X_{N-1} \to X_N\],
so we conclude that $X_{N-1}$ is either supported on a single immersed curve (interval or circle) or a direct sum of two intervals.

\begin{figure}[h]
    \centering
    \includegraphics[width=\textwidth]{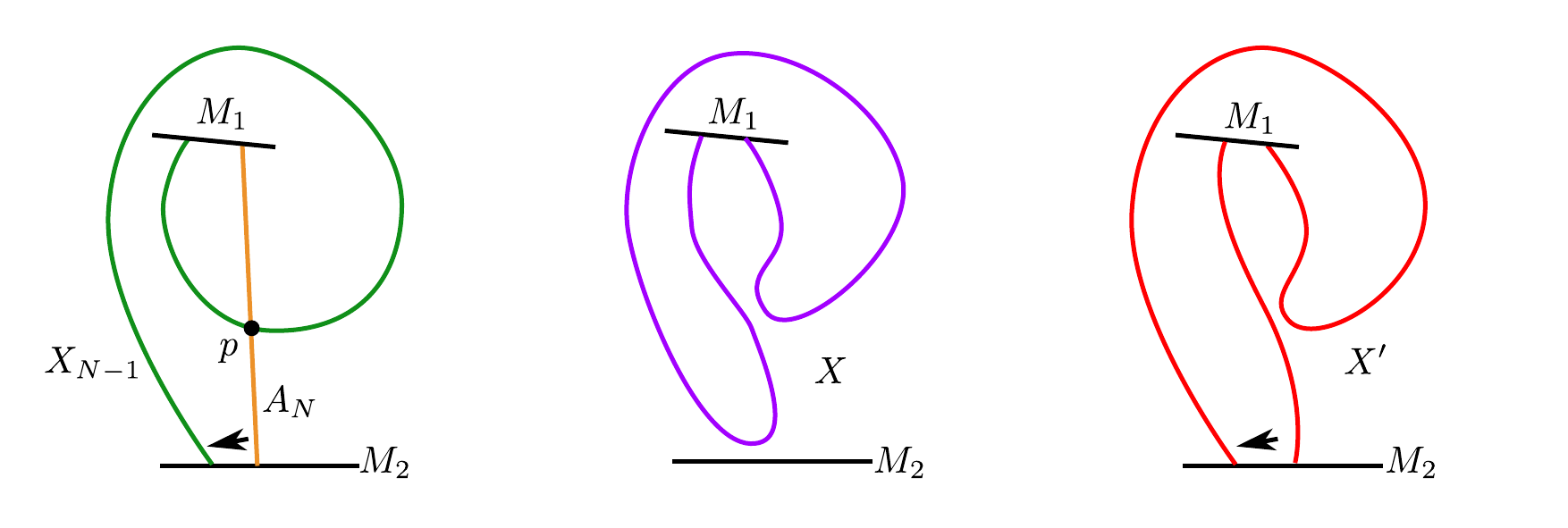}
    \caption{One example where $A_N$ extends $X_{N-1}$ with an extension map $c_2 M_2 + c_p p$. Using only the extension at $p$ we obtain $X'$ which is the sum of two interval objects (each of smaller total length), which can be extended at $M_2$ to give $X$. In this case $X_{N-1}$ and $A_N$ shared the other boundary too; this does not have to be the case in general}
\end{figure}

We choose $\alpha_N$ and $\gamma_{N-1}$ to be in minimal position. The extension map $\eta \in \Ext^1(A_N,X_{N-1})$ comes from a linear combination of classes corresponding to generalized intersections in $\alpha_N \cap' \gamma_{N-1}$. Let us write
\[ \eta = c_1 M_1 + c_2 M_2 + \sum_p c_p p \]
where $M_1, M_2$ are extension maps given by the marked boundary intervals at the end of $A_N$ and $p$ labels extension maps coming from intersection points. Note that the coefficients $c_1,c_2,c_p$ are not uniquely defined.

We see that it is impossible to have $c_1 = c_2 = 0$. If the extension happens only at transverse intersection points, then this extension is supported on two (or more) superimposed curves which is impossible since we assumed $X_N = X$ was indecomposable.

Consider then the modified extension map
\[ \eta' = \sum_p c_p p \]
and the corresponding extension $X_{N-1} \to X' \to A_N$. This is supported on a set of curves that share the marked boundary intervals $M_1$ and/or $M_2$ and moreover can be extended at those to obtain the original object $X$. This topologically constrains $X'$ to be of one of three types:
\begin{enumerate}
    \item $X' = I_1 \oplus I_2$, two intervals which can be extended at a common boundary to form the interval object $X$,
    \item $X' = I_1 \oplus I_2 \oplus I_3$, three intervals which can be extended at two common boundaries to form the interval object $X$,
    \item $X' = I_1 \oplus I_2$, two intervals which can be extended at both common boundaries to form a circle object $X$.
\end{enumerate}

Whichever case we are in, since total length is additive, the indecomposable factors $I_1,I_2,I_3$ are all of length $\le N-1$ so by the induction hypothesis they have \textsc{cosi} decompositions, which can then be composed at the shared marked boundaries to give a \textsc{cosi} decomposition for $X$.

It remains to deal with the case where $A_N$ is a circle object. Since there is no boundary, the extension map $\eta \in \Ext^1(A_N,X_{N-1})$ must be given by a linear combination
\[ \eta = \sum_p c_p p \]
of the classes given by transverse intersections $p$ between $\alpha_N$ and $\gamma_{N-1}$. Assume first that $N \ge 3$; then $N-1 \ge 2$ and therefore $X_{N-1}$ is not a semistable circle so by the induction hypothesis it has a \textsc{cosi} decomposition coming from concatenating intervals $\alpha_1,\dots,\alpha_{N-1}$.

We see that every transverse intersection of index $1$ between $\alpha_N$ and $\gamma_{N-1}$ must come from one or more transverse intersections of index $1$ between $\alpha_N$ and another $\alpha_i$. However this gives a nonzero class in $\Hom(A_i,A_N)$ which cannot happen if $\phi_{A_i} \ge \phi_{A_N}$, so the only possibility is that these have the same phase, which can be discarded by the genericity condition. The only last case to deal with is when $N = 2$ and $X$ is an extension of two stable circle objects $A_1,A_2$; by the same argument as above this can only happen if the two circles have the same phase, which does not happen by the genericity assumption.
\end{proof}

One easy consequence of this result is that the monodromy of the rank one local system carried by the curve does not matter for its semistability.
\begin{corollary}
Fix any stability condition $\sigma$ as above, and $X$ any rank one object supported on a embedded circle $\gamma$. If $X$ is semistable under $\sigma$, then any other rank one object $X'$ supported on $\gamma$ is also semistable under $\sigma$.
\end{corollary}
\begin{proof}
Suppose otherwise; then $X'$ has a \textsc{cosi} decomposition into stable intervals, not all of the same phase. But the same chain of intervals can be concatenated to give $X$ as well, by taking different multiples of the extension classes between the intervals in the chain, contradicting the uniqueness of the HN decomposition.
\end{proof}

The only indecomposable objects not covered by Theorem \ref{prop:intervalchain} are circle objects with higher rank local systems, but this will cause no further problems:
\begin{lemma}
Let $X$ be an indecomposable object supported on a circle $\gamma$ with higher-rank local system. Then there are two possibilities for $X$:
\begin{enumerate}
    \item $X$ is a semistable circle whose stable components are all rank one objects supported on $\gamma$,
    \item $X$ has a decomposition as as chain of semistable intervals, ie. similar to a \textsc{cosi} decomposition except that every piece is a direct sum of stable intervals instead of a single stable interval.
\end{enumerate}
\end{lemma}
\begin{proof}
Suppose $X$ carries a rank $r$ indecomposable local system $\cL$. If the rank one objects supported on $\gamma$ are stable, then we pick $r$ such objects with monodromies given by the eigenvalues of $\cL$; using the self-extension of the circle we can present $X$ as an iterated extension of these objects, proving that $X$ is semistable, so we are in case (1).
Otherwise, these rank one objects have a \textsc{cosi} decomposition; again we take $r$ copies of this chain of stable intervals and extend them appropriately to construct the local system $\cL$, and we are in case (2).
\end{proof}

Combining the results above, we conclude that certain kinds of embedded intervals always have simple \textsc{cosi} decompositions.
\begin{corollary}\label{corollary:chainofobjects}
Let $X$ be an object of $\cF(\Sigma)$ represented by an embedded interval $\gamma_X$ with trivial rank one local system, such that $\gamma_X$ cuts the surface into two, and has ends on distinct marked boundary intervals. Then $X$ has a simple \textsc{cosi} decomposition under any stability condition on $\cF(\Sigma)$.
\end{corollary}

\section{Relative stability conditions}\label{sec:relstab}
In this section, we present a notion of stability conditions on a surface $\Sigma$ \emph{relative} to part of its boundary. This construction will exhibit functorial behavior and satisfy cutting and gluing relations. First we will give some presentations of the category $\cF(\Sigma)$ that will be useful in stating that definition.

\subsection{Pushouts}\label{subsec:pushout}
In \citep{HKK}, it is shown that given a full system of arcs on $\Sigma$, one can define a graph $G$ dual to it and a constructible cosheaf $\cE$ of $A_\infty$-categories on $G$ such that:
\begin{theorem}\citep[Theorem 3.1]{HKK}
The category $\cF(\Sigma)$ represents global sections of the cosheaf $\cE$, ie. is the homotopy colimit of the corresponding diagram of $A_\infty$-categories.
\end{theorem}

We will describe how to use this result to express $\cF(\Sigma)$ as certain useful homotopy colimits. Let $\gamma$ be some embedded interval dividing $\Sigma$ into two surfaces, $\Sigma_L$ and $\Sigma_R$. Suppose that we have a chain of intervals $\gamma_1,\dots,\gamma_N$ in distinct isotopy classes connecting $n+1$ distinct marked boundary intervals $M_0,\dots,M_n$, such that their concatenation gives the interval $\gamma$.

\begin{lemma}\label{lemma:systemofarcs}
$\Sigma$ admits a full system of arcs $\cA = \cA_L \sqcup \cA_\gamma \sqcup \cA_R$ such that every arc in $\cA_L$ has a representative contained in $\Sigma_L$, every arc in $\cA_R$ has a representative contained in $\Sigma_R$, and $\cA_\gamma = \{\gamma_1,\dots,\gamma_N\}$.
\end{lemma}
\begin{proof}
Consider a (non-full) system of arcs $\overline{\cA}_\gamma$ given by the `closure' of $\cA_\gamma =\{\gamma_1,\dots,\gamma_N\}$; that is containing also a chain of arcs connecting all the marked boundary intervals to the left of the chain $\gamma$, and the analogous chain to the right of it.

\begin{figure}[h]
    \centering
    \includegraphics[width=\textwidth]{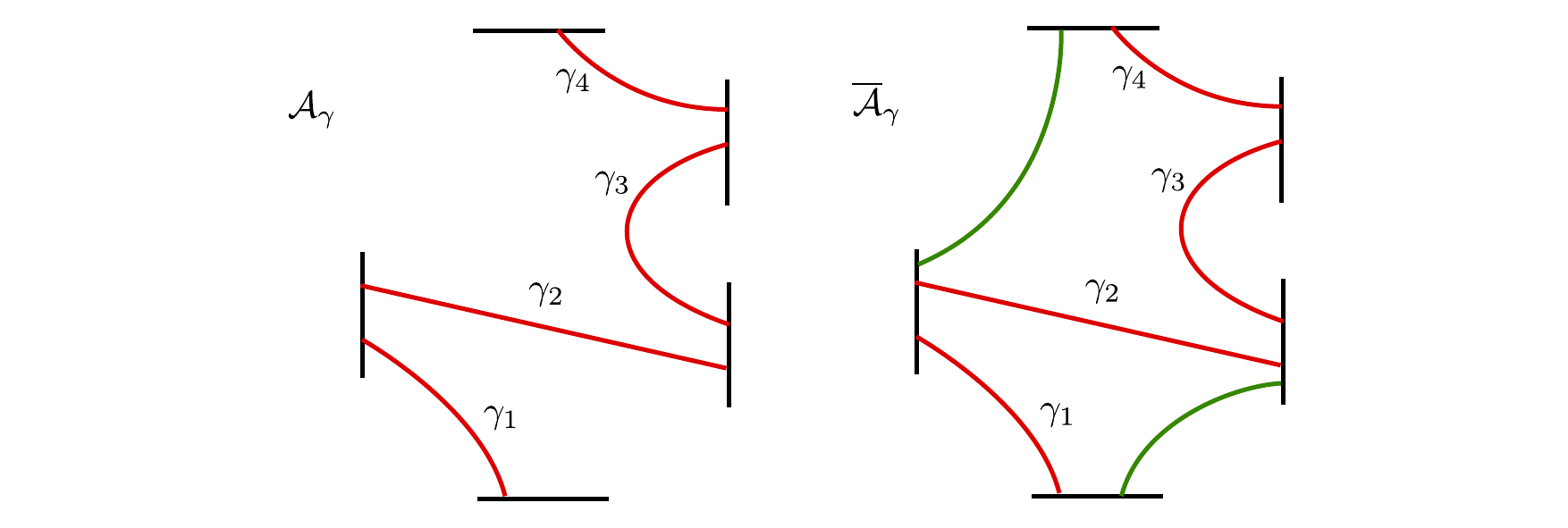}
    \caption{The (non-full) system of arcs $\cA_\gamma$ and its closure $\overline{\cA}_\gamma$. The green arcs are elements of $\overline{\cA}_\gamma \setminus \cA_\gamma$.}
\end{figure}

Since all the intervals in $\overline{\cA}_\gamma$ are non-intersecting and not pairwise isotopic there is some full arc system $\cA$ of $\Sigma$ containing them; and since $\gamma$ (and therefore the chain made by the $\gamma_i$) cuts the surface into two we can partition the arcs $\cA$ that are not among the $\gamma_i$ into left and right subsets $\cA_L$ and $\cA_R$. By construction every arc in $\cA_L$ is contained in $\Sigma_L$ and every arc in $\cA_R$ is contained in $\Sigma_R$.
\end{proof}

Consider this arc system $\cA$. Let us define $\tilde\Sigma_L$ to be the smallest marked surface with an inclusion into $\Sigma$ that contains all the arcs in $\cA_L \sqcup \cA_\gamma$; we define $\tilde\Sigma_R$ analogously.

We see that topologically, $\tilde\Sigma_L,\tilde\Sigma_R$ can being constructed from $\Sigma_L,\Sigma_R$ by attaching a disk along $\gamma$, that is
\[ \tilde\Sigma_L = \Sigma_L \cup_\gamma \Delta_{m}, \qquad \tilde\Sigma_R = \Sigma_R \cup_\gamma \Delta_{n}, \]
where $\Delta_k$ is the disk with $k$ marked boundary intervals. By minimality of these surfaces, we must have $(m-2) + (n-2) = N-1$.

\begin{figure}[h]
    \centering
    \includegraphics[width=0.8\textwidth]{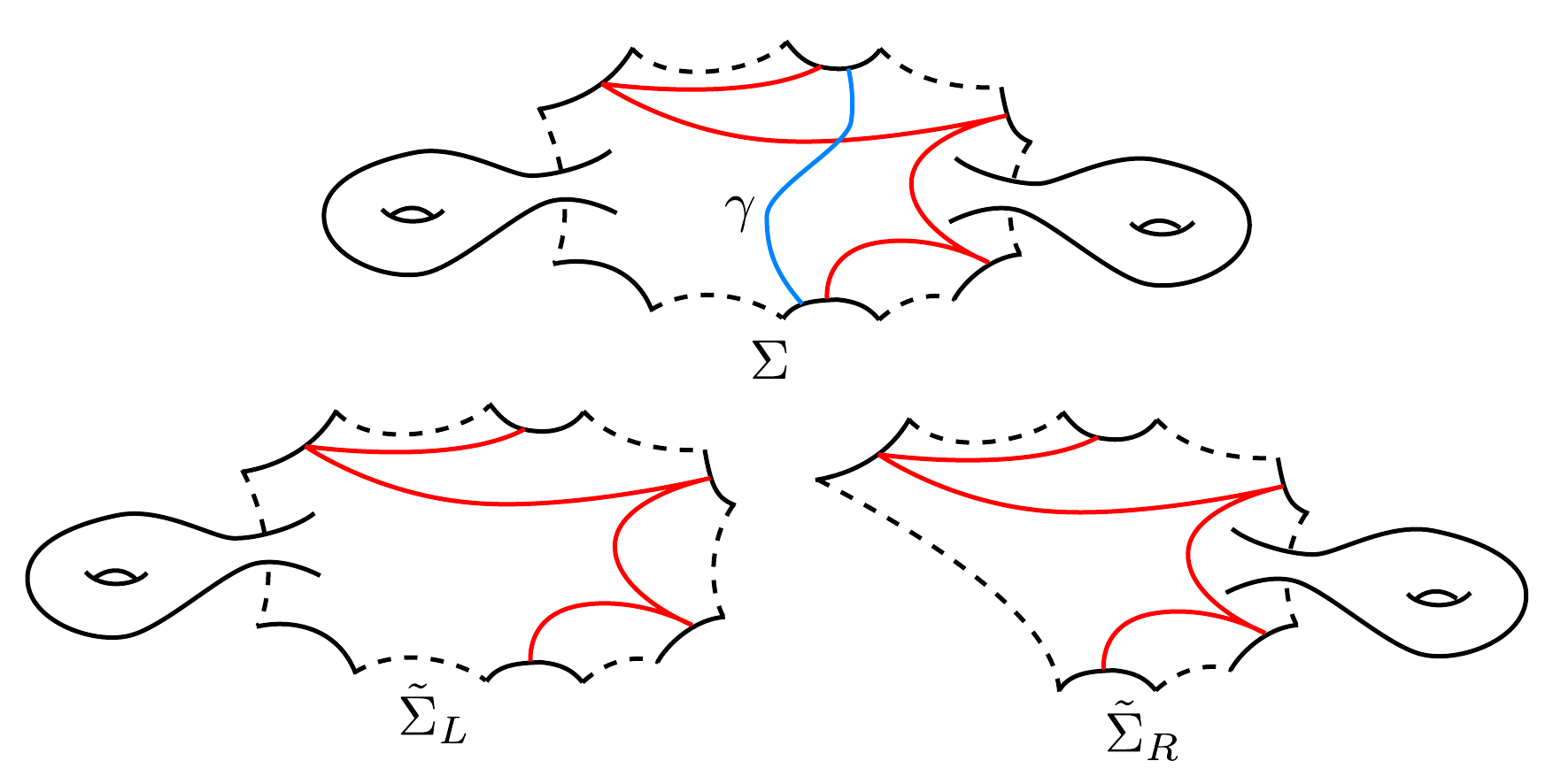}
    \caption{The two `modified' surfaces $\tilde\Sigma_L$ and $\tilde\Sigma_R$. Each one is obtained from $\Sigma_L,\Sigma_R$ respectively by adding more marked intervals ($m$ and $n$ of them) along the boundary according to the chain. In this example $N=4, m=2, n=1$.}
\end{figure}

Let us denote the triangulated closure of the object represented in an arc system by $\langle \cA \rangle$. Then we have $\cF(\tilde\Sigma_L) = \langle \cA_L \sqcup \cA_\gamma \rangle$ and $\cF(\tilde\Sigma_R) = \langle \cA_R \sqcup \cA_\gamma \rangle$. Using the cosheaf description above we can assemble all these categories into the following cube diagram:
\[\xymatrix@=18pt{
\langle \cA_\gamma \rangle \ar[rrr] \ar[ddd]    &                                               &                               & \cF(\tilde\Sigma_R) \ar[ddd] \\
                                                & \langle \gamma \rangle \ar[r] \ar[lu] \ar[d]  & \cF(\Sigma_R) \ar[ru] \ar[d]  & \\
                                                & \cF(\Sigma_L) \ar[ld] \ar[r]                  & \cF(\Sigma) \ar[rd]^\cong     & \\
\cF(\tilde\Sigma_L) \ar[rrr]                    &                                               &                               & \cF(\Sigma)
}\]
where the inner and outer squares, and the top and left sides are all pushouts (ie. homotopy colimits).

\subsection{Main definitions}
Consider now some surface $S$ with an embedded interval $\gamma$ which connects two adjacent marked boundary intervals $M,M'$, and runs parallel to the unmarked boundary interval between them (for example we can take $(S,\gamma) = (\Sigma_L,\gamma)$ as above).

\begin{definition}\label{def:relativestabcond}
A \emph{relative stability condition} on the pair $(S, \gamma)$ is the data of:
\begin{itemize}
    \item An integer $n \ge 2$, and
    \item A stability condition $\tilde\sigma \in \Stab(\cF(\tilde S))$, where $\tilde S = S \cup_\gamma \Delta_n$ is the ``extended surface'' obtained by gluing a disk $\Delta_n$ to $\gamma$ along one of its unmarked boundaries.
\end{itemize}
\end{definition}

Note that the embedded interval $\gamma \subset \tilde S$ cuts the surface into two, so by Lemma \ref{lemma:simplechain} any indecomposable object $C$ supported on $\gamma$ has a simple \textsc{cosi} decomposition under $\tilde\sigma$.

Fix a relative stability condition $\sigma= (Z,\cP)$ and let us denote by $C_1,\dots,C_N$ the corresponding chain of stable intervals in the decomposition of $C$, supported on arcs $\gamma_1,\dots,\gamma_N$. As in the previous subsection, we can take $(\Sigma_L,\Sigma_R) = (S, \Delta_n)$; this defines an arc system $\cA_L \sqcup \cA_\gamma \sqcup \cA_R$ on $\tilde S$.

\subsection{Restricting stability conditions and minimality}
Consider now the central charges
\[ Z_L = Z|_{\langle \cA_L \sqcup \cA_\gamma \rangle}, \qquad  Z_R = Z|_{\langle \cA_\gamma \sqcup \cA_R \rangle} \]
and the `candidates for slicings' $\cP_L,\cP_R$, given by intersecting the full triangulated subcategories $\cP_\phi$ with the full triangulated subcategories $\langle \cA_L \sqcup \cA_\gamma \rangle, \langle \cA_\gamma \sqcup \cA_R \rangle$, respectively.

\begin{lemma}
$\sigma|_L = (Z_L,\cP_L)$ and $\sigma|_R = (Z_R,\cP_R)$ give stability conditions on the subcategories $\langle \cA_L \sqcup \cA_\gamma \rangle$ and $\langle \cA_\gamma \sqcup \cA_R \rangle$.
\end{lemma}
\begin{proof}
The compatibility between the central charges and filtrations is obvious by construction; we only need to check that $\cP_L,\cP_R$ do in fact give slicings, ie. that every object in either category has an HN decomposition by objects in each restricted slicing. This can be checked on indecomposable objects and follows from Lemma \ref{lemma:noncrossing}; every indecomposable object on either side can be represented by some immersed curve keeping to the same side of the chain $\gamma$, so therefore its HN decomposition under the original stability condition $\sigma$ cannot cross to the other side.
\end{proof}

Note that this construction $\sigma \to (\sigma|_L,\sigma|_R)$ does not give a map from $\Stab(\cF(\tilde S))$ to any other fixed stability space; as $\sigma$ varies, the target categories $\langle \cA_L \sqcup \cA_\gamma \rangle$ change since the decomposition of the interval object $C$ changes as we cross a wall. However, this only happens across some specific kinds of walls, defined by the following condition:

\begin{definition}
The relative stability condition $\sigma$ is \emph{non-reduced} if there are two interval objects $C_i,C_{i+1}$ extended on the right (ie. by an extension map $C_{i+1} \xrightarrow{+1} C_i$), with the same phase. Otherwise, we will say $\sigma$ is \emph{reduced}.
\end{definition}

By standard results \citep{bridgeland2015quadratic}, the subset of non-reduced stability conditions is contained in a locally finite union of walls of $\Stab(\cF(\tilde S))$ walls, so the subset of reduced stability conditions is composed of open chambers.
\begin{lemma}\label{lemma:constanttarget}
Within each chamber $\cC$ of reduced relative stability conditions, the target subcategory $\langle \cA_L \sqcup \cA_\gamma \rangle$ is constant and the map $\Stab(\cF(\tilde S)) \to \Stab(\langle \cA_L \sqcup \cA_\gamma \rangle)$ is continuous.
\end{lemma}
\begin{proof}
Within each reduced chamber $\cC$, the chain $\gamma$ is constant except for the (internal) walls on which two (or more) adjacent interval objects of the same phase $C_i,C_{i+1}$ are extended on the left (ie. by an extension map $C_i \xrightarrow{+1} C_{i+1}$). However, though the chain $\cA_\gamma$ changes across such a wall, by construction of $\cA_L$ we see that $\langle \cA_\gamma \sqcup \cA_L \rangle$ stays constant. Continuity follows from the fact that a small enough neighborhood of every stability condition on some category $\cD$ is isomorphic to $(K_0(\cD))^\vee = \Hom_\ZZ(K_0(\cD),\CC)$ and in that neighborhood the map $\Stab(\cF(\tilde S)) \to \Stab(\langle \cA_L \sqcup \cA_\gamma \rangle)$ is described by the projection dual to the inclusion $K_0(\langle \cA_L \sqcup \cA_\gamma \rangle) \to K_0(\cF(\tilde S))$.
\end{proof}

For our later uses, we would like to define a notion of minimality, in the sense that the integer $n$ of marked boundary intervals of $\Delta_n$ is as small as possible.
\begin{definition}
A relative stability condition $\sigma$ on $(S,\gamma)$ \emph{minimal} if every marked boundary interval of $\Delta_n$ appears in the simple chain of stable intervals decomposition of $C$.
\end{definition}

Another way of phrasing the minimality condition is:
\begin{lemma}\label{lemma:minimal}
$\sigma$ is minimal if and only if and $\langle \cA_R \rangle \subseteq \langle \cA_\gamma \rangle$.
\end{lemma}

\subsection{The space of relative stability conditions}\label{subsec:space}
For our purposes, the part of the stability condition `purely on the disk side' (that is, restricted to the subcategory generated by $\cA_R$) does not matter; we realize this by using an equivalence relation. Let $\sigma \in \Stab(\cF(\tilde S = S \cup_\gamma \Delta_m))$ and $\sigma' \in \Stab(\cF(\tilde S' = S \cup_\gamma \Delta_n))$ be two relative stability conditions on $(S,\gamma)$. As above, one can (non-uniquely) pick corresponding arc systems $\cA_L \sqcup \cA_\gamma \sqcup \cA_R$ and $\cA'_L \sqcup \cA'_\gamma \sqcup \cA'_R$ on $\tilde S$ and $\tilde S'$, and restrict stability conditions to each side.

We will see that we need to be careful about genericity when defining the correct equivalence relation. For motivation let us first define a naive notion of equivalence:
\begin{definition}(Naive equivalence)
$\sigma \sim_{\mathrm{naive}} \sigma'$ if there is some extended surface $\tilde S'' = S \cup_\gamma \Delta_\ell$ with inclusions $\tilde S \hookrightarrow \tilde S''$ and $\tilde S' \hookrightarrow \tilde S''$, such that there is an equality
\[ \langle \cA_L \sqcup \cA_\gamma \rangle \cong \langle \cA'_L \sqcup \cA'_\gamma \rangle \]
of subcategories of $\cF(\tilde S'')$ generated by (the images) of the arc systems $\cA_L \sqcup \cA_\gamma$ and $\cA'_L \sqcup \cA'_\gamma$, such that the restricted stability conditions $\sigma|_{\langle \cA_L \sqcup \cA_\gamma \rangle}$ and $\sigma'|_{\langle \cA'_L \sqcup \cA'_\gamma \rangle}$ agree.
\end{definition}
Note that to compare the two restricted stability conditions above, we use the equivalence induced by the identification of these categories with the same subcategory of $\cF(\tilde S'')$.

\begin{lemma}\label{lem:equivrelation}
The relation $\sim_\mathrm{naive}$ is an equivalence relation on the set of relative stability conditions on $(S,\gamma)$.
\end{lemma}
\begin{proof}
The identity and reflexive axioms are satisfied automatically; transitivity follows from the combinatorics of full arc systems on discs. More specifically, the data of the arc system $\cA_\gamma \sqcup \cA_R$ is given by the data of a full arc system on a disc $\Delta_{\bar m}$ for some $\bar m \ge m$; and so on for the other two arc systems. The transitive axiom is satisfied since any two surfaces $S \cup_\gamma \Delta_{\ell_1}$ and $S \cup_\gamma \Delta_{\ell_2}$ giving two relations $\sigma \sim_{\mathrm{naive}} \sigma'$ and $\sigma' \sim_{\mathrm{naive}} \sigma''$ can always be extended to a larger surface with an arc system on a disc $\Delta_{\bar \ell}$, with $\bar \ell \ge \bar \ell_1 + \bar \ell_2$.
\end{proof}

We would like to define the space of relative stability conditions as the quotient of the space
\[ \SS = \bigsqcup_{n \ge 2} \Stab(\cF(S \cup_\gamma \Delta_n) )\]
by the relation $\sim_{\mathrm{naive}}$, but it turns out that this quotient is ill-behaved; it does not give a Hausdorff space, because the graph  $\Gamma_{\sim \mathrm{naive}} \subset \SS \times \SS$ of the naive relation is not a closed subset.
\begin{example}
Take the simple example where $S \cong \Delta_2$ with unique (up to shift) indecomposable object $C$ and $\tilde S \cong \tilde S' \cong \Delta_3$, with objects $A,B,C$ as below.

\begin{figure}[h]
    \centering
    \includegraphics[width=\textwidth]{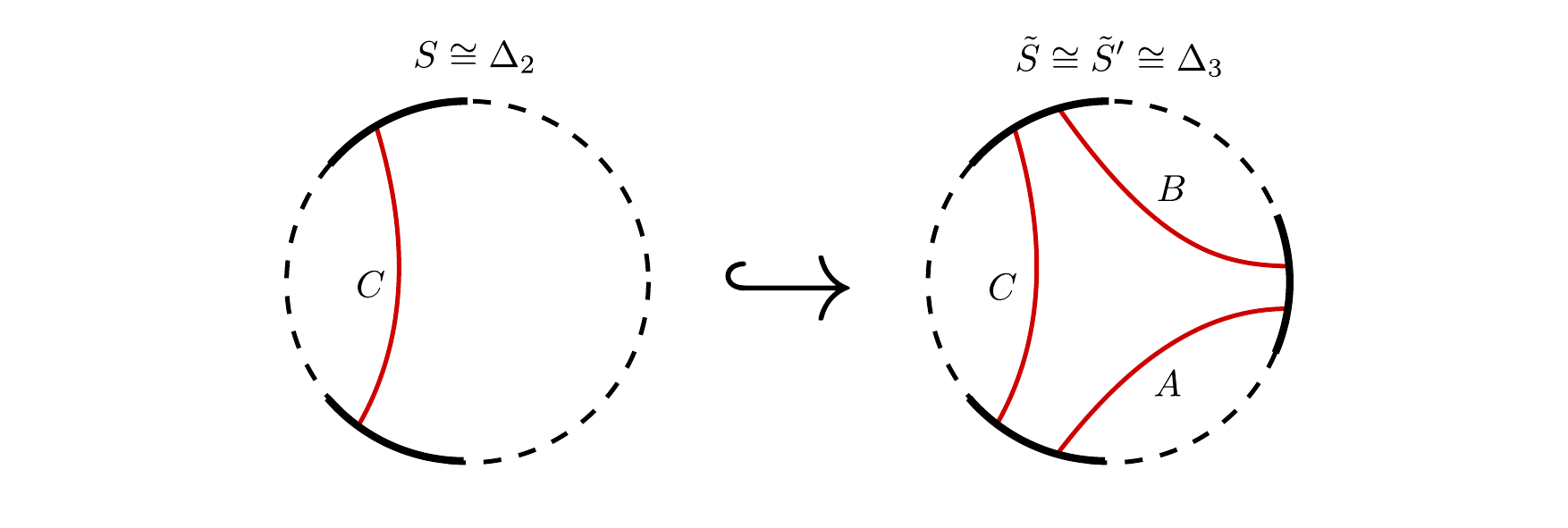}
    \caption{The surfaces $S \cong \Delta_2$ and $\tilde S \cong \tilde S' \cong \Delta_3$. The category $\cF(S)$ is equivalent to $\Mod(A_1)$ and $\cF(\tilde S)$ is equivalent to $\Mod(A_2)$.}
\end{figure}

We have a distinguished triangle $A \to C \to B$. Consider two infinite families of stability conditions on $\cF(\Delta_3)$, $\{\sigma_m = (Z_m,\cP_m)\}$ and $\{\sigma'_m = (Z'_m,\cP'_m)\}$ with $m\in \ZZ_+$, on $\cF(\Delta_3)$ given by the central charges
\[ Z_m(A) = \frac{1}{3} + i \frac{1}{m}, \qquad Z_m(B) = \frac{2}{3} - i \frac{1}{m} \]
\[ Z'_m(A) = \frac{2}{3} + i \frac{1}{m}, \qquad Z'_m(B) = \frac{1}{3} - i \frac{1}{m} \]
with $A,B$ and $C$ stable in all of them, picking phases for all these objects between $-1/2$ and $1/2$. Each one of these sequences converges in $\Stab(\cF(\Delta_3))$ respectively, to the stability conditions $\sigma_\infty, \sigma'_\infty$ with central charges
\[ Z_\infty(A) = \frac{1}{3}, \qquad Z_\infty(B) = \frac{2}{3}\]
\[ Z'_\infty(A) = \frac{2}{3}, \qquad Z'_\infty(B) = \frac{1}{3} \]
where $A,B$ are stable but $C$ is only semistable, with Jordan-H\"older factors $A,B$.

Seen as relative stability conditions on $(\Delta_2,\gamma)$, all the $\sigma_m,\sigma'_m$ for any $m$ are equivalent under $\sim_{\mathrm{naive}}$; the subcategory $\langle \cA_L \sqcup \cA_\gamma \rangle$ is $\cF(\Delta_2) = \langle C \rangle$ and the central charge of $C$ is 1 for all finite $m$. On the other hand, $\sigma_\infty$ and $\sigma'_\infty$ are not equivalent under $\sim_{\mathrm{naive}}$, since for those two $\langle \cA_L \sqcup \cA_\gamma \rangle$ is the whole category. Thus $(\sigma_\infty,\sigma'_\infty) \in \overline{\Gamma_{\sim \mathrm{naive}}} \setminus \Gamma_{\sim \mathrm{naive}}$.
\end{example}

As in the example above, the problem always arises when we have relative stability conditions which are non-reduced. Consider a relative condition $\sigma$ on $(S,\gamma)$ given by a stability condition on $\cF(\tilde S)$ for some $\tilde S = S \cup_\gamma \Delta_n$, where the object $C$ supported on $\gamma$ has a \textsc{cosi} decomposition $C_1,\dots,C_N$. Assume that $\sigma$ is non-reduced; this means that there is a nonempty set of indices $R \subset \{1,\dots,N\}$ such that the extension map is `on the right' (ie. $\in \Ext^1(C_{i+1},C_i)$) and $C_i$ and $C_{i+1}$ have the same phase. Let us suppose that the set $R$ is of the form $j,j+1,\dots,j+m$ for some $1 \le j \le j+m \le N-2$ with all objects $C_j,\dots,C_{j+m+1}$ having the same phase $\phi$; the general case (where $R$ is the disjoint union of a number of those subsets) will not be any more difficult.

Consider now the reduced arc system given by
\[ \cA^{\mathrm{red}}_\gamma = \{\gamma_1, \dots, \gamma_{j-1}, \tilde\gamma, \gamma_{j+m+2}, \dots, \gamma_N\}, \]
where $\tilde\gamma$ is obtained by concatenating the intervals $\gamma_j,\dots,\gamma_{j+m+1}$ at the $m$ marked boundaries $M_i$ with index $i \in R$. Let us now define a reduced restriction $\sigma^\mathrm{red}$ given by restricting the data of $\sigma$ to the subcategory $\langle \cA_L \sqcup \cA^{\mathrm{red}}_\gamma \rangle$, and then adding to the category $\cP_\phi$ the objects supported on $\tilde\gamma$.

\begin{lemma}\label{lemma:reduced}
$\sigma^\mathrm{red}$ is a stability condition on the category $\langle \cA_L \sqcup \cA^{\mathrm{red}}_\gamma \rangle$.
\end{lemma}
\begin{proof}
It suffices to prove that every object in the subcategory $\langle \cA_L \sqcup \cA^{\mathrm{red}}_\gamma \rangle$ has an HN decomposition into stable objects also in that same subcategory. Because of Lemma \ref{lemma:noncrossing}, the only way this could fail is if there is some indecomposable object $X$ of $\langle \cA_L \sqcup \cA^{\mathrm{red}}_\gamma \rangle$ in whose decomposition some but not all of the stable interval objects $C_j,\dots,C_{j+m+1}$ appear (if all of them appear we just replace that semistable object with the stable object $\tilde C$ supported on $\tilde\gamma$). But this cannot happen for phase reasons, following a similar argument as the proof of Lemma \ref{lemma:noncrossing}.
\end{proof}

For completeness let us define $\sigma^\mathrm{red} = \sigma|_L$ if $\sigma$ is reduced. With this definition we can now state the correct notion of equivalence.
\begin{definition}(Equivalence)
$\sigma \sim \sigma'$ if there is some extended surface $\tilde S'' = S \cup_\gamma \Delta_\ell$ with inclusions $\tilde S \hookrightarrow \tilde S''$ and $\tilde S' \hookrightarrow \tilde S''$, such that there is an equality
\[ \langle \cA_L \sqcup \cA^{\mathrm{red}}_\gamma \rangle \cong \langle \cA'_L \sqcup \cA'^{\mathrm{red}}_\gamma \rangle \]
of subcategories of $\cF(\tilde S'')$ generated by (the images) of the arc systems $\cA_L \sqcup \cA_\gamma$ and $\cA'_L \sqcup \cA'_\gamma$, such that the reduced restricted stability conditions $\sigma^\mathrm{red}$ and $\sigma'^{\mathrm{red}}$ agree.
\end{definition}

It is clear from the definition that $\sim$ is an equivalence relation on the set $\SS = \bigsqcup_{n \ge 2} \Stab(\cF(S \cup_\gamma \Delta_n))$, by the same argument as in Lemma \ref{lem:equivrelation}.
\begin{lemma}\label{lemma:uniqueminimal}
There is a unique minimal and reduced relative stability condition in each equivalence class of the equivalence relation $\sim$.
\end{lemma}
\begin{proof}
Consider some relative stability condition $\sigma$; as above it defines a stability condition $\sigma^\mathrm{red}$ on the subcategory $\langle \cA_L \sqcup \cA^{\mathrm{red}}_\gamma \rangle$. Note that this subcategory is also of the form $\cF(S \cup_\gamma \Delta_n)$, with $n = |\cA^\mathrm{red}_\gamma| + 1$, and also by construction $\sigma$ is equivalent to the reduced $\sigma^\mathrm{red}$ when both are viewed as relative stability conditions on $(S,\gamma)$.

Suppose now that we have two stability conditions $\sigma \sim \sigma'$ which are minimal and thus reduced; then the arcs in $\cA_R,\cA'_R$ can be generated by the other arcs so by compatibility we have
\[ \cF(\tilde S) \cong \langle \cA_L \sqcup \cA_\gamma \rangle \cong \langle \cA'_L \sqcup \cA'_\gamma \rangle \cong \cF(\tilde S'), \]
but it is easy to see that no two categories $\cF(S \cup_\gamma \Delta_n)$ are equivalent for different $n$ (for example by taking $K_0$) so $\tilde S \cong \tilde S'$ (compatibly with the embedding of $S$) with equivalent stability conditions.
\end{proof}

\begin{definition}(Space of relative stability conditions)
Let us define $\RelStab(S,\gamma)$ as the set of minimal and reduced stability conditions; this set is given the quotient topology by the identification $\RelStab(S,\gamma) = \SS/\sim$,
\end{definition}

\begin{proposition}\label{prop:hausdorff}
The space $\RelStab(S,\gamma)$ is Hausdorff.
\end{proposition}
\begin{proof}
This is equivalent to showing that the graph $\Gamma_\sim$ of the equivalence relation is closed in $\SS\times \SS$. Since $\SS$ is an disjoint union this is equivalent to showing $\Gamma_\sim$ is closed in each component $\Stab(\cF(\tilde S)) \times \Stab(\cF(\tilde S'))$.

The spaces $\Stab(\cF(\tilde S))$ have a wall-and-chamber structure where the walls are the locus of non-reduced stability conditions. By standard arguments, the union of all walls is a locally finite union of real codimension one subsets. The complement is composed of open chambers, and by Lemma \ref{lemma:constanttarget} the target subcategory $\cT = \langle \cA_L \sqcup \cA_\gamma \rangle$ is constant on each chamber.

In the interior of each chamber
\[ \cC = \cC_\rho \times \cC_\sigma \subset \Stab(\cF(\tilde S)) \times \Stab(\cF(\tilde S')), \]
the locus $\Gamma_\sim$ is the preimage of the diagonal $\Delta \subset \Stab(\cT) \times \Stab(\cT)$, so it is closed by continuity.

Let us look at the walls surrounding the chamber $\cC$, and start with a simple codimension one wall $W$, ie. the locus at the boundary of $\cC$ where the phases $\phi_i,\phi_{i+1}$ of two adjacent interval objects $C_i,C_{i+1}$ (with an extension to the right) agree. There are two possibilities: $\phi_i < \phi_{i+1}$ or $\phi_i > \phi_{i+1}$ inside of $\cC$. In the former case, comparing the target categories we see that the reduced target category $\cT^\mathrm{red}_W$ on the wall is equal to the usual target category $\cT_\cC$ in the interior of the chamber, so we can apply the same argument as inside the chamber and conclude that $\Gamma_\sim \cap W$ is closed.

In the latter case $\cT^{\mathrm{red}}_W$ is smaller than $\cT_\cC$, as it doesn't contain the objects $C_i,C_{i+1}$, only their extension. However, the closure $\overline{\Gamma_\sim \cap \cC}$ meets $W$ along a closed locus contained within $\Gamma_\sim \cap W$, as the reduced equivalence condition is strictly weaker than the naive equivalence condition on $W$. The general case for walls of higher codimension is essentially the same and can be obtained iteratively.

Now, over the entire space $\Stab(\cF(\tilde S)) \times \Stab(\cF(\tilde S'))$, since each point is surrounded by finitely many reduced chambers and $\Gamma_\sim$ is closed within the closure of each one of them, $\Gamma_\sim$ is the locally finite union of closed subsets.
\end{proof}

\begin{remark}
Unlike the space of stability conditions $\Stab(\cF(S))$, the space $\RelStab(S,\gamma)$ is not a complex manifold; it is an infinite-dimensional space obtained by gluing complex manifolds of unbounded dimension along real-analytic subsets (defined by inequalities of phases of stable objects).
\end{remark}

\subsection{Compatibility}
Consider now two surfaces $S$ and $S'$ with embedded intervals $\gamma,\gamma'$ and relative stability conditions $\sigma \in \RelStab(S,\gamma)$ and $\sigma \in \RelStab(S',\gamma')$. Given any two such surfaces, we can glue them by identifying $\gamma = \gamma'$ and obtain a surface $S \cup_\gamma S'$. Since there is a full arc system on this surface containing the arc $\gamma$, one can take the ribbon graph dual to this arc system and get a pushout presentation
\[ \cF(S \cup_\gamma S') = \cF(S) \cup_{\cF(\gamma)} \cF(S'). \]
The relative stability conditions $\sigma,\sigma'$ have unique minimal and reduced representatives by Lemma \ref{lemma:uniqueminimal}. However they also have many minimal but non-reduced representatives.
\begin{definition}\label{def:compatibility}
A \emph{compatibility structure} between $\sigma$ and $\sigma'$ is the following data:
\begin{itemize}
    \item Minimal representatives $\tilde\sigma \in \Stab(\cF(\tilde S))$ and $\tilde\sigma' \in \Stab(\cF(\tilde S')) $ of $\sigma$ and $\sigma'$.
    \item Inclusions of surfaces
    \[ S \hookrightarrow \tilde S \hookrightarrow S \cup_\gamma S', \qquad S' \hookrightarrow \tilde S' \hookrightarrow S \cup_\gamma S', \]
\end{itemize}
such that the images of the embedded intervals in the \textsc{cosi} decompositions of $\gamma$ and $\gamma'$ agree as an arc system $\cA_\gamma$ inside of $S \cup_\gamma S'$, and the restrictions $\tilde\sigma|_{\langle \cA_\gamma \rangle}$ and $\tilde\sigma'|_{\langle \cA_\gamma \rangle}$ are the same stability condition in $\Stab(\langle \cA_\gamma \rangle)$.
\end{definition}

\section{Cutting and gluing relative stability conditions}\label{sec:cutglue}
In this section, we will explain how to cut (ordinary) stability conditions into relative stability conditions and glue relative stability conditions into (ordinary) stability conditions. This will allow us to reduce the calculations of stability conditions on general surfaces $\Sigma$ to the calculation of stability conditions on simpler surfaces. Before we present these procedures, we will need to use the following generalization of a slicing.

\begin{definition}
A \emph{pre-slicing} $\cP^{pre}$ on a category $\cC$ is a choice of full triangulated subcategories $\cP^{pre}_\phi$ for every $\phi \in \RR$, such that $\Hom(X,Y) = 0$ if $X \in \cP^{pre}_\phi$ and $Y \in \cP^{pre}_\psi$, $\phi > \psi$.
\end{definition}

\begin{remark}
This is the same data as a slicing, except that we don't require the existence of Harder-Narasimhan decompositions for objects.
\end{remark}

\begin{definition}
A \emph{pre-stability condition} on $\cC$ is the data of a central charge function $Z:K_0(\cC) \to \CC$ and a pre-slicing $\cP^{pre}$ satisfying the usual compatibility condition $Z(X)/|Z(X)| = e^{i\pi \phi}$ if $X \in \cP^{pre}_\phi$, and the support property (Definition \ref{def:supportproperty}).
\end{definition}

Let us denote by $\PreStab(\cC)$ the set of all pre-stability conditions on $\cC$. It is obvious that we have an inclusion of sets
\[ \Stab(\cC) \hookrightarrow \PreStab(\cC). \]

\subsection{Cutting stability conditions}\label{sec:cutting}
We return to the setting of a surface $\Sigma$ that is cut into $\Sigma_L,\Sigma_R$ by an embedded interval $\gamma$ supporting a rank one object $C$.

Consider a stability condition $\sigma \in \Stab(\cF(\Sigma))$. By Corollary \ref{corollary:chainofobjects}, the object $C$ has a simple \textsc{cosi} decomposition into objects $C_1,\dots,C_N$ supported on arcs $\gamma_1,\dots,\gamma_N$, which connect the marked boundary intervals $M_0,\dots, M_N$. As in subsection \ref{subsec:pushout}, there is then a full system of arcs
\[ \cA = \cA_L \sqcup \cA_\gamma \sqcup \cA_R\]
such that every arc in $\cA_L$ has a representative contained in $\Sigma_L$, every arc in $\cA_R$ has a representative contained in $\Sigma_R$, and $\cA_\gamma = \{\gamma_1,\dots,\gamma_N\}$.

Each extension between $C_i$ and $C_{i+1}$ happens either on the left (ie. by an extension map $C_i \xrightarrow{+1} C_{i+1}$) or on the right (ie. by an extension map $C_{i+1} \xrightarrow{+1} C_i$). Let $m$, $n$ be the numbers of indices with extension on the left and right, respectively, plus $2$; we have by definition $m-2+n-2=N+1=$ number of marked boundary intervals along the chain.

Then we have surfaces $\tilde\Sigma_L = \Sigma_L \cup_\gamma \Delta_m$ and $\tilde\Sigma_R = \Sigma_R \cup_\gamma \Delta_n$ such that
\[ \cF(\tilde\Sigma_L) = \langle \cA_L \sqcup \cA_\gamma \rangle, \qquad \cF(\tilde\Sigma_R) = \langle \cA_R \sqcup \cA_\gamma \rangle. \]
Consider the restrictions
\[ \sigma_L = \sigma|_{\langle \cA_L \sqcup \cA_\gamma \rangle}, \qquad \sigma_R = \sigma|_{\langle \cA_\gamma \sqcup \cA_R \rangle}, \]
that is, as in the previous section we take the data given by restricting the central charges and intersecting the slicings with each full subcategory.

\begin{lemma}\label{lemma:cut}
$\sigma_L,\sigma_R$ are stability conditions on $\cF(\tilde \Sigma_L), \cF(\tilde\Sigma_R)$.
\end{lemma}
\begin{proof}
The condition $Z(X) = m(X)\exp(i\pi \phi_X)$ on every semistable object $X$ is satisfied by construction, so we just need to check that (1) every object $X \in \cF_L$ has a HN filtration, ie. that $\cP_L$ indeed defines a slicing, and (2) the resulting pairs of central charge and slicing satisfy the support property.

It is enough to check (1)  on indecomposable objects. By geometricity, every such object $X$ is represented by an immersed curve in $\tilde\Sigma_L$ with indecomposable local system. Consider its image in $\cF(\Sigma)$ which is also an immersed curve, and its chain-of-interval decomposition under $\sigma$.

If $X$ is an interval object, then both of its ends are on marked boundary components belonging to $\tilde\Sigma_L$, and since the associated chain of intervals is isotopic to the support of $X$, if any of those intervals in in $\Sigma_R$, then the chain must cross back to $\Sigma_L$, creating a polygon of the sort prohibited by Lemma \ref{lemma:noncrossing}. And if $X$ is a circle object then it is by definition supported on a non-nullhomotopic immersed circle, so by the same argument its chain of intervals cannot cross over to $\Sigma_R$ without also creating a prohibited polygon. Thus every stable component of the HN decomposition is in $\cF_L$.

As for (2), the support property for each side follows directly from the support property for the original stability condition $\sigma$, since by definition the sets of semistable objects of $\sigma_L,\sigma_R$ are subsets of the set of semistable objects of $\sigma$, and both the central charges and norms on $K_0(\cF(\tilde \Sigma_L))_\RR, K_0(\cF(\tilde \Sigma_R))_\RR$ are defined by pullback from $\cF(\Sigma)$, so the relevant ratios are just calculated by the original ratio $|Z(-)|/\lVert [-] \rVert$ on $\cF(\Sigma)$. 
\end{proof}

We then use the inclusions of marked surfaces $\Sigma_L \hookrightarrow \tilde\Sigma_L$ and $\Sigma_R \hookrightarrow \tilde\Sigma_R$ to interpret these stability conditions as relative stability conditions:
\begin{definition}
The cutting map
\[ \cut_\gamma : \Stab(\cF(\Sigma)) \to \RelStab(\Sigma_L,\gamma) \times \RelStab(\Sigma_R,\gamma) \]
sends a stability conditions $\sigma$ as above to the image of the stability conditions $(\sigma_L,\sigma_R)$.
\end{definition}

By Lemma \ref{lemma:uniqueminimal} every element of $\RelStab$ has a unique minimal and reduced representative, so we can alternatively define the cutting map by using the `reduced restriction' of Lemma \ref{lemma:reduced}
\[ \cut_\gamma(\sigma) = (\sigma^\mathrm{red}_L, \sigma^\mathrm{red}_R). \]

\begin{lemma}\label{lemma:cutcontinuous}
The map $\Stab(\cF(\Sigma)) \xrightarrow{\cut_\gamma} \RelStab(\Sigma_L,\gamma) \times \RelStab(\Sigma_R,\gamma) $ is continuous.
\end{lemma}
\begin{proof}
We must look separately at the maps to each side; let us prove continuity of the map $\Stab(\cF(\Sigma)) \xrightarrow{\cut_L} \RelStab(\Sigma_L,\gamma)$. Recall that in subsection \ref{subsec:space} we define the topology on the $\RelStab$ spaces as the quotient topology inherited from $\SS = \bigsqcup_n \Stab(S \cup_\gamma \Delta_n)$.

Note that the construction for the map $\cut_L$ does not give a manifestly continuous map since the target $\cT = \langle \cA_L \sqcup \cA_\gamma \rangle$ changes across walls in $\Stab(\cF(\Sigma))$. We remediate this by locally defining other maps that are continuous, and which agree with $\cut_L$ after identifying by the equivalence relation $\sim$.

Let $\sigma$ be a stability condition on $\cF(\Sigma)$ such that $\sigma_L = \sigma|_{\langle \cA_L \cup \cA_\gamma \rangle}$ is a non-reduced stability condition, and let us say that under $\sigma$ the object $C$ supported on $\gamma$ has a decomposition into $C_1,\dots,C_N$ supported on embedded intervals $\gamma_1,\dots,\gamma_N$ with respective phases $\phi_1,\dots,\phi_N$. Non-reducedness means that there is some collection of indices $i$ such that $C_i,C_{i+1}$ have the same phase, and are extended on the right. For simplicity, suppose first that we have a single such index; the general case can be deduced by iterating this argument. Let us denote $C_\mathrm{bot}$ to be the object obtained by concatenating $C_1,\dots,C_i$, and $C_\mathrm{top}$ to be the object obtained by concatenating $C_{i+1},\dots,C_N$.

By standard arguments, the locus on which the objects $C_1,\dots,C_N$ are simple is open, so there is a neighborhood $U \ni \sigma$ on which all these objects are simple, and with a complex isomorphism $U \cong (K_0(\cF(\Sigma)))^\vee$. If necessary we further restrict $U$ such that on this open set the $\phi_{i-1} \neq \phi_i$ and $\phi_{i+1} \neq \phi_{i+2}$. This implies that on $U$ the chains $C_1,\dots,C_i$ and $C_{i+1},\dots,C_N$ gives \textsc{cosi} decompositions of $C_\mathrm{bot}$ and $C_\mathrm{top}$, respectively.

Consider now a fixed target category $\cT_\mathrm{fix}$ given by the target $\cT_\sigma = \langle \cA_L \sqcup \cA_\gamma \rangle$ at $\sigma$. We argue that for every stability condition $\sigma' \in U$, $\sigma'|_{\cT_\mathrm{fix}}$ is a stability condition. Note that this does not follow immediately from Lemma \ref{lemma:noncrossing} since along some chambers in $U$, the pair $C_i,C_{i+1}$ is not the \textsc{cosi} decomposition of any object so we cannot directly use the non-crossing argument.

Nevertheless, we can use a small modification of that argument. Consider some indecomposable object $X$ in the subcategory $\cT_\mathrm{fix}$; by geometricity this can be represented by an immersed curve $\xi$ to the left of the chain of intervals, and by the results of Section \ref{sec:lemmas}, $X$ has a \textsc{cosi} decomposition into intervals $\xi_1,\dots,\xi_M$ whose concatenation is isotopic to $\xi$.

Now, since both ends of $\xi$ are to the left of the $\gamma$ chain, and this chain is divided into two stable chains, extended on the left, the only way that the $\xi$ chain can cross the $\gamma$ chain is it if crosses the chain for $C_\mathrm{bot}$ or $C_\mathrm{top}$ (or both). But again this is prohibited by the noncrossing argument of Lemma \ref{lemma:noncrossing}.

Thus this defines a map $\widetilde{\cut}_\gamma: U \to \Stab(\cT_\mathrm{fix})$ which by construction is continuous and agrees with $\cut_\gamma$ on $U$; doing this for every wall gives continuity of $\cut_\gamma$.
\end{proof}

Note that by construction we have representatives $\sigma_L \in \Stab(\cF(\tilde\Sigma_L))$ and $\sigma_R \in \Stab(\cF(\tilde\Sigma_R))$ of the relative stability conditions $\sigma^\mathrm{red}_L, \sigma^\mathrm{red}_R$, and also inclusions of surfaces $\tilde\Sigma_L \hookrightarrow \Sigma$ and $\tilde\Sigma_R \hookrightarrow \Sigma$. It follows directly from the construction above that:
\begin{lemma}
This is a compatibility structure between $\sigma^\mathrm{red}_L$ and $\sigma^\mathrm{red}_R$.
\end{lemma}

\subsection{Gluing stability conditions}
As in the previous section consider a surface $\Sigma = \Sigma_L \cup_\gamma \Sigma_R$ cut into two parts by an embedded interval. Suppose we have relative stability conditions $\sigma_L \in \RelStab(\Sigma_L,\gamma)$ and $\sigma_R \in \RelStab(\Sigma_R,\gamma)$ with some compatibility structure between them (as in Definition \ref{def:compatibility}).

Unpacking this data, we have non-negative integers $m$ and $n$ and stability conditions $\sigma_L = (Z_L,\cP_L)$ on
\[ \cF_L = \cF(\tilde\Sigma_L) = \cF(\Sigma_L \cup_\gamma \Delta_m) \]
and $\sigma_R = (Z_R,\cP_R)$ on
\[ \cF_R = \cF(\tilde\Sigma_R) = \cF(\Sigma_R \cup_\gamma \Delta_n) \]
representing $\sigma_L,\sigma_R$, together with inclusions of marked surfaces $\Sigma_L \hookrightarrow \tilde\Sigma_L \hookrightarrow \Sigma$ and $\Sigma_R \hookrightarrow \tilde\Sigma_R \hookrightarrow \Sigma$.

The compatibility condition implies that the chain-of-intervals decomposition $C^L_1, \dots, C^L_N$ of the indecomposable object $C^L \in \cF_L$ supported on $\gamma \subset \tilde\Sigma_L$ and the chain-of-intervals decomposition $C^R_1,\dots,C^R_N$ of the indecomposable object $C^R \in \cF_R$ supported on $\gamma \subset \tilde\Sigma_R$ are of the same length $N$ on both sides, and that the central charges agree, ie.
\[ Z_L(C^L_i) = Z_R(C^R_i) \]
for all $i$. Also compatibility also requires that the extension maps $\eta^L_i$ and $\eta^R_i$ go the same direction, ie. either both go forward
\[ \eta^L_i \in \Ext^1(C^L_i, C^L_{i+1}) \mathrm{\ and\ } \eta^R_i \in \Ext^1(C^R_i, C^R_{i+1}) \]
or both go backward
\[ \eta^L_i \in \Ext^1(C^L_{i+1}, C^L_i) \mathrm{\ and\ } \eta^R_i \in \Ext^1(C^R_{i+1}, C^R_i), \]
so we have the relation $(m-2)+ (n-2) = N-1$ due to minimality of $\sigma_L$ and $\sigma_R$.

The compatibility structure gives an identification between the images of $C^L_1,\dots,C^L_N$ and $C^L_1,\dots,C^L_N$ inside of $\cF(\Sigma)$; we denote this full subcategory spanned by these arcs $\langle \cA_\gamma \rangle$ as in previous sections. This gives a pushout presentation
\[
\xymatrix{
\langle \cA_\gamma \rangle 	\ar[r] \ar[d]	& \cF_R \ar[d]^{j_R}  \\
\cF_L \ar[r]^{j_L}    & \cF(\Sigma)
}
\]
From this data we will produce a central charge function $K_0(\cF(\Sigma)) \to \CC$ and a pre-slicing $\cP$ on $\cF(\Sigma)$.

\subsubsection{The central charge}\label{subsec:centralcharge}
Applying the functor $K_0$ to the pushout above gives us a diagram of $\ZZ$-modules
\[
\xymatrix{
K_0(\langle \cA_\gamma \rangle) 	\ar[r] \ar[d] 	& K_0(\cF_R) \ar[d] \\
K_0(\cF_L) \ar[r]    & K_0(\cF(\Sigma))
}
\]
\begin{lemma}
This is a pushout of $\ZZ$-modules.
\end{lemma}
\begin{proof}
A diagram of this type (that is, coming from the cosheaf property of the Fukaya category) need not be \emph{a priori} a pushout, since $K_0$ does not necessarily commute with colimits. However note that in this case we have an explicit description of the $K_0$ groups in terms of $H^1$ groups because of Theorem \ref{thm:K0}, and the result follows from the fact that we are gluing along a single chain.

More explicitly, note that $K_0(\cF(S))$ for some marked surface $S$ is generated by the arcs in an arc system modulo relations coming from polygons. Completing $\cA_\gamma$ to a full arc system $\cA_L \sqcup \cA_\gamma \sqcup \cA_R$ we see that since there are no polygons crossing between the two sides of the chain, so the set of relations on $K_0(\cF(\Sigma))$ is the union of the sets of relations defining $K_0(\cF_L)$ and $K_0(\cF_R)$; this implies the statement above.
\end{proof}

By compatibility of the relative stability conditions $\sigma_L$ and $\sigma_R$, the central charges on both sides agree when restricted to $K_0(\langle \cA_\gamma \rangle)$, so we get a map $Z:K_0(\cF(\Sigma)) \to \CC$; this will be our central charge.

\subsubsection{The pre-slicing}

We will define full subcategories $\cP_\phi$ of semistable objects in two steps. Let us first define initial subcategories $\cP'_\phi$ by
\[\cP'_\phi = j_L((\cP_L)_\phi) \cup j_R((\cP_R)_\phi), \]
that is, we take the images of the semistable objects under $\sigma_L$ and $\sigma_R$ to be stable in $\cF(\Sigma)$.

Now let us algorithmically add some objects to the slicing by the following prescription. We first define a particular kind of arrangement of stable objects. Let us denote by
\[ \MM_\gamma = \{ M_0, \dots, M_N\} \subseteq \MM \]
the marked intervals appearing in the decomposition of $\gamma$, in order, that is, between $M_{i-1}$ and $M_i$ there is a stable interval object $C_i$ appearing in the decomposition of the rank 1 object supported on $\gamma$.

\begin{definition}\label{def:lozenge}
A \emph{lozenge} of stable intervals is the following arrangement on $\Sigma$:
\begin{itemize}
    \item Four marked boundaries $M_\ell, M_r, M_{up}, M_{down}$, where 
    \[ M_{up} = M_i, \quad M_{down} = M_{i-1} \]
    for some $0 \le i < N$. 
    \item A chain of intervals $\alpha_1,\dots,\alpha_a$ linking $M_\ell$ to $M_{up}$, such that $\alpha_i$ supports a stable object $A_i \in \cP'_{\phase(A_i)}$, and
    \[ \phase(A_1) \le \dots \le \phase(A_a). \]
    \item A chain of intervals $\beta_1,\dots,\beta_b$ linking $M_{up}$ to $M_r$, such that $\beta_i$ supports a stable object $B_i \in \cP'_{\phase(B_i)}$, and
    \[ \phase(B_1) \le \dots \le \phase(B_b). \]
    \item A chain of intervals $\delta_1,\dots,\delta_d$ linking $M_\ell$ to $M_{down}$, such that $\delta_i$ supports a stable object $D_i \in \cP'_{\phase(D_i)}$, and
    \[ \phase(D_1) \ge \dots \ge \phase(D_d). \]
    \item A chain of intervals $\eta_1,\dots,\eta_d$ linking $M_{down}$ to $M_r$, such that $\eta_i$ supports a stable object $E_i \in \cP'_{\phase(E_i)}$, and
    \[ \phase(E_1) \ge \dots \ge \phase(E_e). \]
\end{itemize}
such that the phases of these stable objects satisfy
\[ \phase(D_1) \le \phase(A_1) \le \phase(D_1) + 1, \quad \phase(B_1) \le \phase(A_a) \le \phase(B_1) + 1,\]
\[ \phase(B_b) \le \phase(E_e) \le \phase(B_b) + 1, \quad \phase(D_d) \le \phase(E_1) \le \phase(D_d) + 1. \]
and such that these four chain of stable intervals bound a disk containing the interval supporting the object $C_i$. This kind of arrangement is pictured in Figure \ref{fig:lozenge}.
\end{definition}

\begin{figure}[h]
    \centering
    \includegraphics[width=\textwidth]{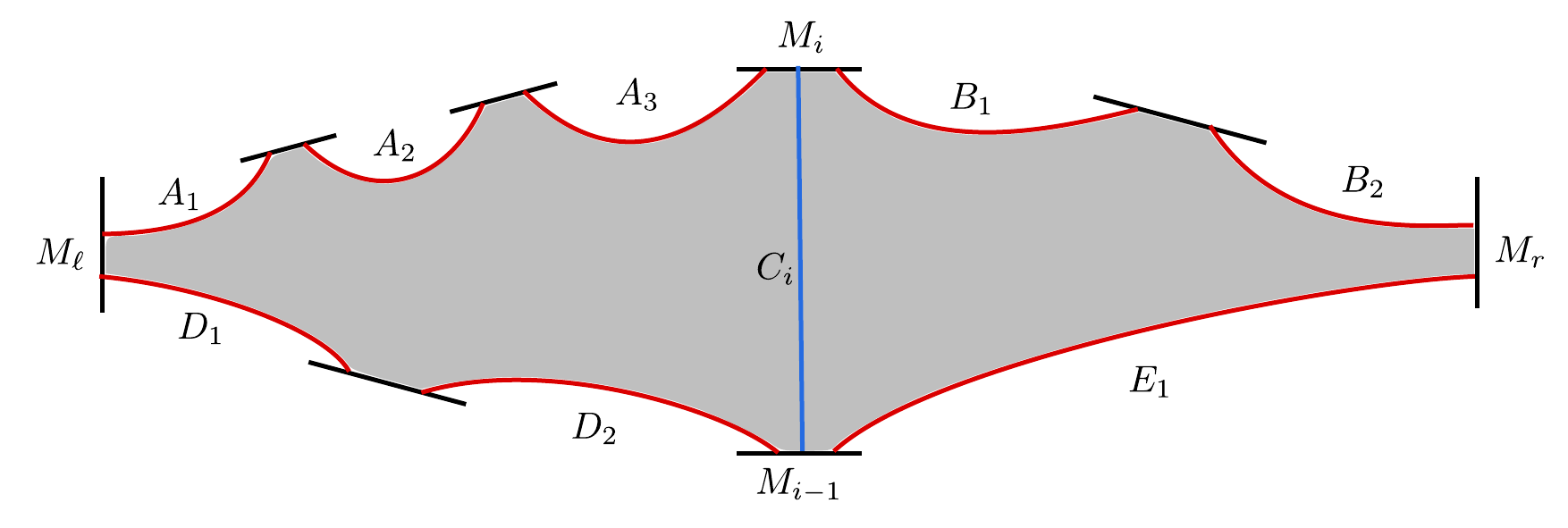}
    \caption{A lozenge of stable objects with $a=3,b=2,d=2,e=1$.}
    \label{fig:lozenge}
\end{figure}

Consider now the complex number
\[ Z(X) := \sum_i Z(A_i) + \sum_i Z(B_i) = \sum_i Z(D_i) + \sum Z(E_i), \]
which is the central charge of the object $X$ supported on the interval from $M_\ell$ to $M_r$ one gets by successive extensions of the $A_i, B_i$ or $D_i,E_i$. The equality follows from well-definedness of $Z$.

\begin{definition}
We call such a lozenge \emph{unobstructed} if there is a choice of branch of the argument function $\arg:\CC^\times \to \RR$ such that the following inequalities between the phases are satisfied:
\[ \phase(D_1) \le \arg(Z(X)) \le \phase(A_1), \quad \phase(B_b) \le \arg(Z(X)) \le \phase(E_e). \]
\end{definition}

\begin{figure}[h]
    \centering
    \includegraphics[width=\textwidth]{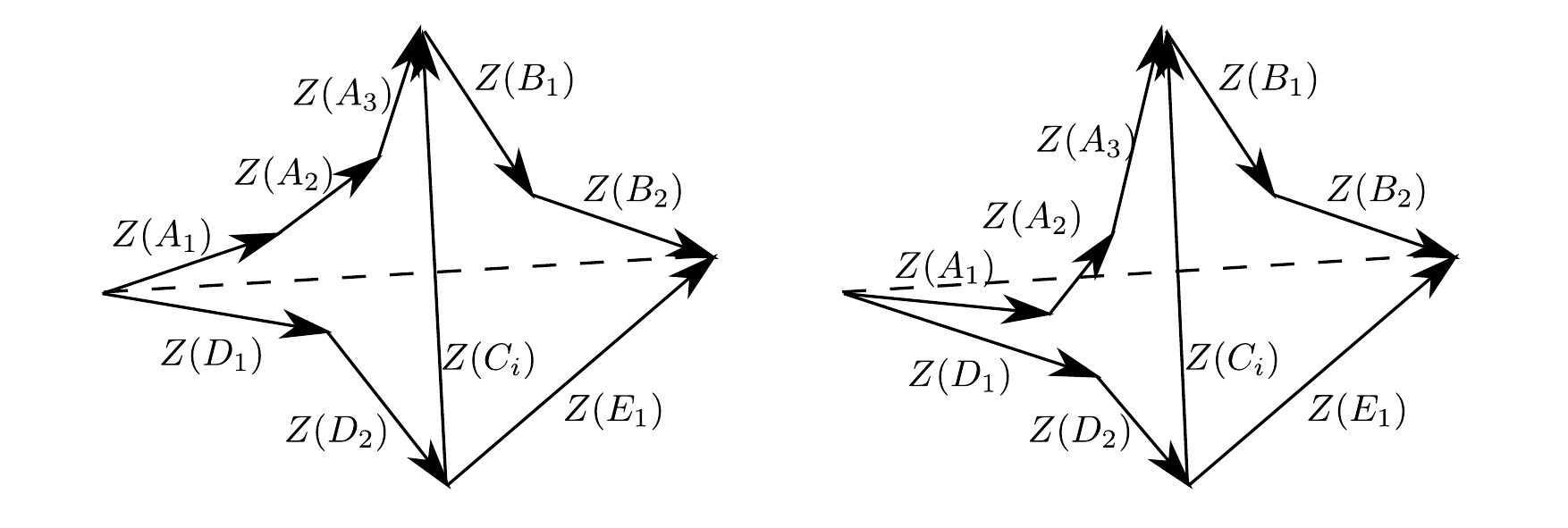}
    \caption{The central charges of the objects in an unobstructed lozenge (left) and in an obstructed lozenge (right).}
\end{figure}

It follows from the inequalities above that if a lozenge is unobstructed then there is only a single choice of $\arg(Z(X))$ satisfying the condition; let's call it $\phi_X \in \RR$. 

\begin{definition}
The preslicing $\cP$ is defined by setting $\cP_\phi$ to be the minimal additive subcategory containing all objects in $\cP'_\phi$ plus all objects $X$ of phase $\phi_X = \phi$ corresponding to unobstructed lozenges.
\end{definition}

\begin{remark}
Note that even though we termed the marked intervals to the `left' and `right' of the lozenge as $M_\ell$ and $M_r$, this \emph{does not mean} that $M_\ell$ is necessarily part of $\Sigma_L$ or that $M_r$, of $\Sigma_r$; there can be lozenges which cross back and forth between $\Sigma_L$ and $\Sigma_R$. It is still true that by definition, the interior of each lozenge must non-trivial intersection with both sides of the surface.
\end{remark}

\begin{lemma}\label{lemma:glue}
The data $Z$ and $\cP$ as above define a prestability condition on $\cF(\Sigma)$.
\end{lemma}
\begin{proof}
The compatibility between the argument of $Z$ and the phase of the subcategories $\cP$ is automatic from the definition, since every stable object either comes directly from one side or has central charge and phase defined by the formula above. So we have to prove that $\cP$ is in fact a preslicing: we must show that $\Hom(X,Y) = 0$ if $X \in \cP_{\phi_X}$ and $Y \in \cP_{\phi_Y}$ with $\phi_X > \phi_Y$, and that the resulting collection of objects satisfies the support property with respect to $Z$.

By definition, each full subcategory $\cP_\phi$ can be spanned by three full subcategories
\[ \cP^L_\phi = j_L((\cP_L)_\phi), \quad \cP^R_\phi = j_R((\cP_R)_\phi), \quad \cP^\lozenge_\phi, \]
where $\cP^\lozenge_\phi$ has all the objects of phase $\phi$ obtained from unobstructed lozenges. Note that $\cP^\lozenge_\phi$ is disjoint from the other two, but $\cP^L_\phi$ and $\cP^R_\phi$ are not disjoint; in fact their intersection is spanned by the objects supported on the chain of intervals $\{\gamma_i\}$.

Let us check vanishing of the appropriate homs. It is enough to check on stable objects. If $X,Y \in \cP^L$ then
\[ \Hom(X,Y) \neq 0 \implies \phi_X \le \phi_Y \]
automatically since they're both semistable in $\cF_L$ and $\cF_L \to \cF(\Sigma)$ is fully faithful; same for the case $X,Y \in \cP^R$. So there are four remaining cases:
\begin{enumerate}
    \item $X \in \cP^L_{\phi_X}$        and $Y \in \cP^R_{\phi_Y}$
    \item $X \in \cP^\lozenge_{\phi_X}$ and $Y \in \cP^L_{\phi_Y}$
    \item $X \in \cP^L_{\phi_X}$        and $Y \in \cP^\lozenge_{\phi_Y}$
    \item $X \in \cP^\lozenge_{\phi_X}$ and $Y \in \cP^\lozenge_{\phi_Y}$
\end{enumerate}
All the other cases can be obtained symmetrically by switching left and right. Let us treat each case separately:
\begin{enumerate}
    \item We can find representatives of $X,Y$ contained in the images of $\tilde\Sigma_L,\tilde\Sigma_R$ respectively, such that neither intersects the chain $\{\gamma_i\}$; so there are no intersections between them. The only way we can have $\Hom(X,Y) \neq 0$ is if $X$ and $Y$ are intervals sharing a common boundary component at one of the $M_i$ along the chain, with a boundary path from $X$ to $Y$.

    \begin{figure}[h]
    \centering
    \includegraphics[width=0.5\textwidth]{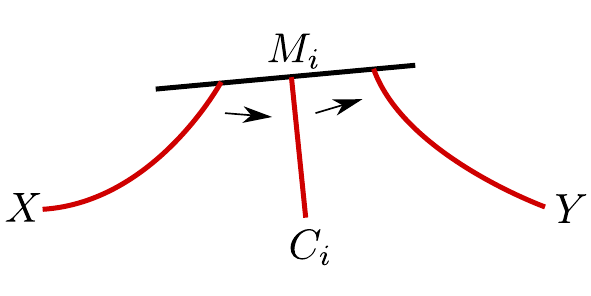}
    \end{figure}

    Consider then $C_i$ and shift its grading so that the morphism $X \to C_i$ is in degree zero; then by index arguments the morphism $C_i \to Y$ is also in degree zero. But since these three objects are stable we have
    \[ \phi_X \le \phi_{C_i} \le \phi(Y). \]

    \item Let $X$ be obtained from an unobstructed lozenge with notation as in Definition \ref{def:lozenge}, and $Y \in \cF_L$. Consider the distinguished triangle $B \to X \to A$ and let us apply the functor $\Hom(-,Y)$ to get a distinguished triangle
    \[ \Hom(A,Y) \to \Hom(X,Y) \to \Hom(B,Y). \]
    Since $Y$ comes from $\cF_L$, it has a representative that stays to the left of the chain and therefore of $B$ so by assumption we have $\Hom(B,Y) = 0$. Thus if $\Hom(X,Y) \neq 0$ then $\Hom(A,Y) \neq 0$. Since $A$ is given by the iterated extension of the $A_i$, there must be some $A_i$ with $\Hom(A_i,Y) \neq 0$; but $A_i$ and $Y$ are both in the image of $\cF_L$ we must have $\phi_{A_i} \le \phi_Y$, and also by construction $\phi_X \le \phi_{A_1}$ so we have
    \[ \phi_X \le \phi_{A_1} \le \phi_{A_i} \le \phi_Y. \]

    \begin{figure}[h]
    \centering
    \includegraphics[width=0.5\textwidth]{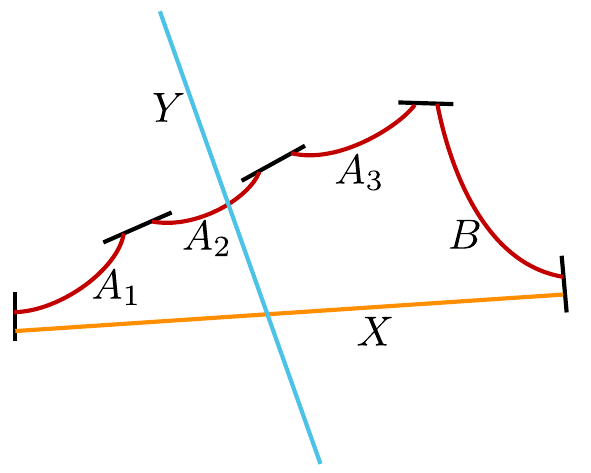}
    \end{figure}

    \item Suppose we have an unobstructed lozenge with sides $A,B,D,E$ and diagonal $Y$. A similar argument as in case (2) shows that if $\Hom(X,Y) \neq 0$, then $\Hom(X,D) \neq 0$, and then for some $i$ we have $\Hom(X,D_i) \neq 0$
    \[ \phi_X \le \phi_{D_i} \le \phi_{D_d} \le \phi_Y. \]

    \item This case can be obtained by an iterated version of the argument in case $(2)$. Let us denote the two lozenges by $A_X,B_X,D_X,E_X$ with diagonal $X$ and $A_Y,B_Y,D_Y,E_Y$ with diagonal $Y$. Suppose that $\Hom(X,Y) \neq 0$, and consider the triangle $D_X \to X \to E_X$. Consider first the case $\Hom(D_X,Y) = 0$ then $\Hom(E_X,Y) \neq 0$. Now consider the triangle $B_Y \to Y \to A_Y$. Since $E_X$ and $A_Y$ have representatives contained in the right and the left side, respectively, and don't share a boundary component we have $\Hom(E_X,A_Y) = 0$ so we must have $\Hom(E_X,B_Y) \neq 0$. But then there must be indices $i,j$ such that $\Hom((E_X)_i,(B_Y)_j) \neq 0$ so then
    \[ \phi_X \le \phi_{(E_X)_i} \le \phi_{(B_Y)_j} \le \phi_Y. \]
    The other case is $\Hom(D_X,Y) \neq 0$. Consider the triangle $B_Y \to Y \to A_Y$. By an analogous argument we can find indices $i,j$ such that
    \[ \phi_X \le \phi_{(A_X)_i} \le \phi_{(D_Y)_j} \le \phi_Y. \]
\end{enumerate}

We now turn to the support property. Let us pick a norm $\lVert \cdot \rVert$ on $K_0(\cF(\Sigma))_\RR$, and use its restrictions $\lVert \cdot \rVert_L, \lVert \cdot \rVert_R$ on $K_0(\cF_L)_\RR$, $K_0(\cF_R)_\RR$. By assumption, the stability conditions $\sigma_L,\sigma_R$ satisfy the support property\footnote{Recall that if the support condition holds for some norm, it holds for all norms, though possibly with different constants.} with some constants $c_L,c_R$, so we set $c:= \min(c_L,c_R)$.

The only new objects we added to the collection of semistable objects were the diagonals of unobstructed lozenges, so it remains to prove the support property for those. Let $X$ be such an object, with the notation of Definition \ref{def:lozenge}. From the triangle inequality we have 
\begin{align*}\lVert X \rVert &\le \sum \lVert A_j \rVert + \sum \lVert B_j \rVert, \\
\lVert X \rVert &\le \sum \lVert D_j \rVert + \sum \lVert E_j \rVert.
\end{align*}
Now, from the unobstructed condition on the complex plane of central charges, we have the following inequalities:
\begin{align*} |Z(X)| + |Z(C_i)| &\ge \sum |Z(A_j)| + \sum |Z(E_j)|, \\
|Z(X)| + |Z(C_i)| &\ge \sum |Z(B_j)| + \sum |Z(D_j)|.
\end{align*}
Note now that there are finitely many $C_i$, so we can bound $|Z(C_i)|$ from above by some constant $\Lambda > 0$ uniformly for all lozenges. We now combine the inequalities above to obtain
\begin{align*} 2(|Z(X)| + L) &\ge \sum |Z(A_j)| + \sum |Z(B_j)| + \sum |Z(D_j)| \sum |Z(E_j)| \\
&> c(\lVert A_j \rVert + \sum \lVert B_j \rVert + \lVert D_j \rVert + \sum \lVert E_j \rVert) \\
&\ge 2 c \lVert X \rVert,
\end{align*}
implying that $|Z(X)|/\lVert X \rVert > c - L/\lVert X \rVert$.

By assumption $\rk(K_0(\cF(\Sigma)) < \infty$ so we only need to worry about the support condition for objects with large norm, which is proven by the inequality above. More explicitly, for any radius $\rho > 0$, the set $\{ [X] \ |\  \lVert X \rVert \le \rho \} \subset K_0(\cF(\Sigma))$ is finite. We pick a radius $\rho > \Lambda/c$, and then have
\[  \inf_{0 \neq X \mathrm{semistable}} \frac{|Z(X)|}{\lVert X \rVert} > c' \]
where
\[ c' = \min \left( \min_{\lVert X \rVert \le \rho} \frac{|Z(X)|}{\lVert X \rVert}, c - \frac{\Lambda}{\rho} \right) > 0, \]
proving the support condition with constant $c'$.
\end{proof}

\subsection{Uniqueness of compatibility structure}
In the same setting as the previous subsection, let $\Gamma \subset \RelStab(\Sigma_L,\gamma)\times \RelStab(\Sigma_R,\gamma)$ be the locus of pairs of relative stability conditions $(\sigma_L,\sigma_R)$ such that there exists a compatibility condition between $\sigma_L$ and $\sigma_R$.
\begin{lemma}
For each $(\sigma_L,\sigma_R) \in \Gamma$, there is a unique compatibility structure between $\sigma_L$ and $\sigma_R$, up to equivalence.
\end{lemma}
\begin{proof}
Let us first prove that the numbers $m,n$ defining $\tilde\Sigma_L,\tilde\Sigma_R$ are unique. Consider the subset
\[ \MM_\sigma \subset \SS = \bigsqcup_{n \ge 2} \Stab(\cF(\Sigma_L \cup_\gamma \Delta_n) ) \]
of its minimal (but possibly not reduced) representatives. Given $\tilde\sigma \in \MM_\sigma$ we consider the \textsc{cosi} decomposition of the rank one object $C$ supported on $\gamma$ as before, and define the numbers $i(\tilde\sigma), e(\tilde\sigma)$ to be respectively the number of internal/external extensions in the $\gamma$ chain, ie. the number of indices $i$ such that the corresponding extension happens on the left/right, or equivalently by an extension map $\in \Ext^1(C_{i+1},C_i)$/$\in \Ext^1(C_i,C_{i+1})$. This defines constructible functions $i,e: \MM_\sigma \to \ZZ_{\ge 0}$ such that $i(\tilde\sigma) + e(\tilde\sigma) = N-1$, where $N-1$ is the total length of the object $C$ under $\tilde\sigma$.

We argue that the function $i$ is constant; by Lemma \ref{lemma:uniqueminimal} there is a unique minimal and reduced representative $\sigma^\mathrm{red}$ of every relative stability condition. However, reduced restriction does not change the $i$ of a stability condition, so $i(\tilde\sigma) = i(\tilde\sigma^\mathrm{red}) = i(\sigma^\mathrm{red})$ on all of $\MM_\sigma$. We define the same functions on the right side for the relative stability condition $\sigma' \in \RelStab(\Sigma_R,\gamma)$. Compatibility implies that $i(\tilde\sigma) = e(\tilde\sigma'), e(\tilde\sigma) = i(\tilde\sigma')$, but since $i$ is constant there is only one possibility for the value of $e$. Comparing with the gluing map we have $m=e(\tilde\sigma), n=e(\tilde\sigma')$.

This determines the isomorphism type of the surfaces $\tilde\Sigma_L$ and $\tilde\Sigma_R$. Consider now the inclusion of marked surfaces $j_L:\tilde\Sigma_L \hookrightarrow \Sigma_L \cup_\gamma \Sigma_R$. By definition of compatibility structure, $j_L|_{\Sigma_L}$ agrees with the inclusion $\Sigma_L \hookrightarrow \Sigma_L \cup_\gamma \Sigma_R$, so the `left part' of $j_L$ is fixed;  $j_L$ is determined up to equivalence by the images of the extra $m-2$ marked boundary intervals in the disk $\Delta_m$ attached along $\gamma$ (two of the marked boundary intervals are fixed to the ends of $\gamma$).

Analogously, $j_R$ is determined up to equivalence by the image of the extra $n-2$ marked boundary intervals of $\Delta_n$. But the images of the extra $m-2$ marked intervals under $j_L$ is contained in the image of the marked intervals coming from $\Sigma_R$ under $j_R$, so they are fixed; the same is true for the image of the extra $n-2$ marked intervals under $j_R$. Minimality implies that the subcategory $\langle \cA_L \sqcup \cA_\gamma \rangle$ is the whole category $\cF(\tilde\Sigma_R)$ so once we fix $\sigma$, the representative $\tilde\sigma$ is completely determined by its restriction to $\langle \cA_\gamma \rangle \cong \cF(\Delta_{N+1})$.

By the classification of stability conditions on the Fukaya category of a disk presented in \citep[Section 6.2]{HKK}, stability conditions on $\cF(\Delta_{N+1})$ are entirely determined by the central charges and phases of the $N+1$ intervals in the chain. Let us label the marked boundary intervals $M_0,\dots,M_N$ in sequence. We argue that the central charges and phases of the objects $C_1,\dots,C_N$ are unique using the following `zip-up' procedure. Consider first the object $C_1$; since $M_0$ is in the common image of $\Sigma_L$ and $\Sigma_R$, and $M_1$ is `internal' (in the subset counted by the $int$ function) to either of those surfaces, the interval supporting $C_1$ is contained in the image of either $\Sigma_L$ or $\Sigma_R$, so its central charge $Z(C_1)$ and phase $\phi_1$ are fixed by either $\sigma_L$ or $\sigma_R$.

Suppose without loss of generality that the interval supporting $C_1$ is in the image of $\Sigma_L$, and consider now $C_2$. There are two possibilities for $M_2$; either it is internal to $\Sigma_L$ or to $\Sigma_R$. In the former case since $M_1$ and $M_2$ are in the image of the same side $\Sigma_L$, $Z(C_2)$ and $\phi_2$ are fixed by $\sigma_L$. In the latter case, $C_2$ is not in the image of either $\Sigma_L$ or $\Sigma_R$, but we consider the concatenation $C_{1+2}$ given by extending at $M_1$; both ends of this object are in the image of $\Sigma_R$ so the central charge $Z(C_{1+2})$ of this (non-stable) object is fixed by $\sigma_R$. So $Z(C_2) = Z(C_{1+2}) - Z(C_1)$ is also fixed. Moreover, among the shifts of $C_2$, there is a unique one with the extension map at $M_1$ in the correct degree, so $\phi_2$ is also fixed. Proceeding by induction we find that all $Z(C_i),\phi_i$ are fixed by the initial data $\sigma_L,\sigma_R$.
\end{proof}

\subsection{Cut and glue maps on interval-like stability conditions}
Because of the uniqueness of compatibility structure proven above and Lemma \ref{lemma:glue}, we can define a gluing map
\[ \RelStab(\Sigma_L,\gamma) \times \RelStab(\Sigma_R,\gamma) \supset \Gamma \xrightarrow{\glue_\gamma} \PreStab(\cF(\Sigma_L \cup_\gamma \Sigma_R)) \]
that produces a prestability condition.

In general, nothing guarantees that the map $\glue_\gamma$ gives actual stability conditions, that is, whether the pre-slicings obtained by the lozenges of the previous Section give enough objects in the slicing. We now establish sufficient conditions for when this happens.

\begin{definition}
A (pre)stability condition $\sigma = (Z,\cP) \in \Stab(\cF(\Sigma))$ is \emph{interval-like} if all stable objects are embedded interval objects. We denote the subsets of all interval-like stability conditions and prestability conditions by $\Stab(\cF(\Sigma))^{int}$ and $\PreStab(\cF(\Sigma))^{int}$, respectively.
\end{definition}
Note that by Proposition \ref{prop:stableobjects}, for a stability condition $\sigma$, being interval-like is equivalent to having no stable circle objects; note that it can still have semistable circle objects but these will have decompositions in terms of interval objects of the same phase.

Let us restrict the map $\cut_\gamma$ to the locus $\Stab(\cF(\Sigma))^{int}$. By definition, relative stability conditions in the image of this locus are represented by stability conditions on surfaces $\tilde\Sigma_L$ and $\tilde\Sigma_R$ that are also interval-like; let us analogously denote that locus by $\Gamma^{int}$.

It is also clear that from the gluing prescription, one only adds interval objects (corresponding to the unobstructed lozenges), therefore the gluing map restricted to $\Gamma^{int}$ also lands in $\PreStab(\cF(\Sigma))^{int}$

\begin{theorem}\label{thm:inverses}
The composition
\[ \Stab(\cF(\Sigma))^{int} \xrightarrow{\cut_\gamma} \Gamma^{int} \xrightarrow{\glue_\gamma} \PreStab(\cF(\Sigma))^{int} \]
is equal to the canonical inclusion $\Stab(\cF(\Sigma))^{int} \hookrightarrow \PreStab(\cF(\Sigma))^{int}$.
\end{theorem}

Note that the theorem can be also stated as saying that the gluing map lands in $\Stab(\cF(\Sigma))^{int}$ and gives an right-inverse to the cutting map. It is then immediate from the definitions that this is also a left-inverse; the cutting map forgets all the stable objects coming from the lozenges so the composition
\[ \Gamma^{int} \xrightarrow{\glue_\gamma} \Stab(\cF(\Sigma))^{int} \xrightarrow{\cut_\gamma} \Gamma^{int} \]
is the identity on pairs of compatible relative stability conditions.

We will need the following lemma in the proof of \ref{thm:inverses}:
\begin{lemma} \label{lemma:lozenges}
Let $X$ be a stable interval object (under $\sigma$), with a representative that crosses the interval $\gamma$. Then there is an unobstructed lozenge (under $\sigma_L,\sigma_R$) with diagonal $X$. Conversely, the diagonal of every unobstructed lozenge is stable under $\sigma$.
\end{lemma}
\begin{proof}
Let $C_1,\dots,C_N$ be the \textsc{cosi} decomposition of the object $C$ supported on $\gamma$. Since the chain $C_i$ is homotopic to $\gamma$, we can find one with a representative intersecting $X$  transversely. Then we have $\Ext^1(C_j,X) \cong \Hom(X,C_j) \cong k$; consider the corresponding extension and cone
\[ C_j \to A \oplus E \to X, \qquad B \oplus D \to X \to C_j. \]

\begin{figure}[h]
    \centering
    \includegraphics[width=\textwidth]{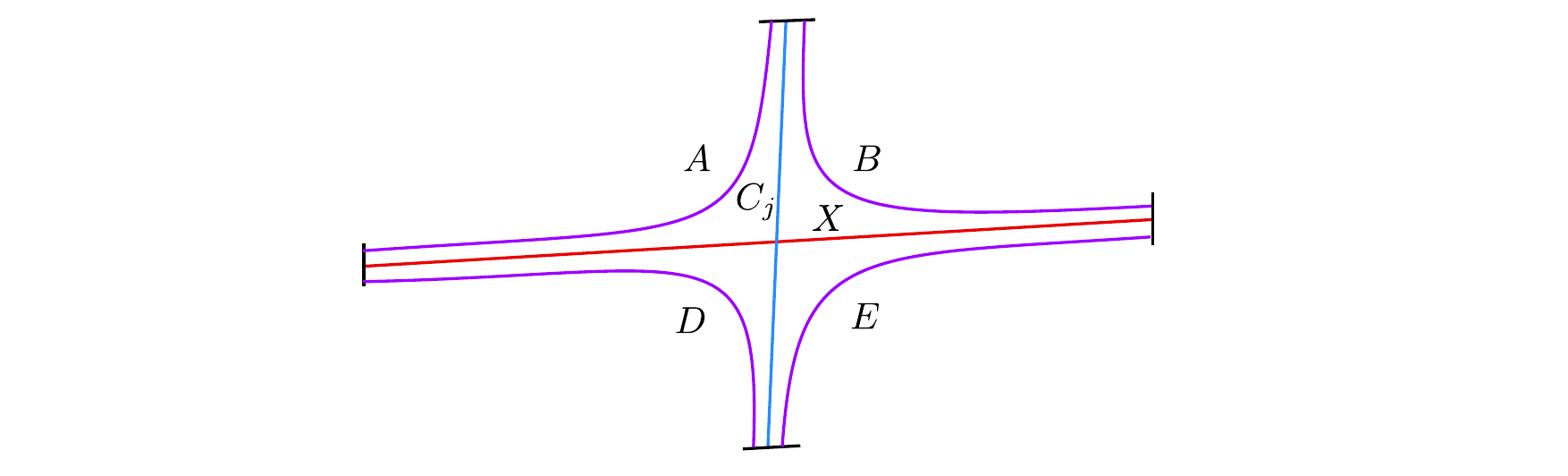}
\end{figure}

Each one of the objects $A,B,D,E$ is an embedded interval object and by Proposition \ref{prop:intervalchain} has a \textsc{cosi} decomposition; we denote the objects in these chains by $\{A_i\},\{B_i\},\{D_i\},\{E_i\}$, respectively.

We argue that $\{A_i\}$ and $\{B_i\}$ only have extensions on the right, and $\{D_i\},\{E_i\}$ only have extensions on the left. Note first that the chains of intervals $\{A_i\}$, $\{D_i\}$ and the interval $\gamma$ don't intersect mutually, since this would contradict Lemma \ref{lemma:noncrossing}. Consider the chain made up of supports of the $A_i$ and $D_i[-1]$. This chain together with $\gamma$ bounds a disk, therefore every extension is on the right; this translates to extensions on the right $\in \Ext^1(A_i,A_{i+1})$ and extensions on the left $\Ext^1(D_{i+1},D_i)$. An analogous argument applies to $B$ and $E$; note that since none of these chains crosses $\gamma$, and $\gamma$ separates $\Sigma$, they do not intersect one another.

Thus we have a lozenge whose diagonal is $X$; it remains to prove it is unobstructed. Suppose that the lozenge $A,B,D,E$ is obstructed; therefore we must have at least one of the following inequalities
\[ \phi_{A_1} \le \phi_X, \quad \phi_{D_1} \ge \phi_X, \quad \phi_{B_b} \ge \phi_X, \quad \phi_{E_e} \le \phi_X. \]

Suppose first that $\phi_{A_1} < \phi_X$. Consider then the object $X'$ given by the iterated extension of $A_2, \dots A_a,B_1,\dots B_b$, we then have a distinguished triangle
\[ X' \to X \to A_1 \]
and the map $X \to A_1$ cannot be zero since $X'$ is indecomposable (by Theorem \ref{thm:geometricity}), which cannot happen since $\phi_X > \phi_{A_1}$. The other cases are similar; moreover, the case of coinciding phases poses no further problems since we can always take $\sigma$ to be appropriately generic (since we need to be off of finitely many walls).

This proves one of the directions. For the converse, suppose that we have an unobstructed lozenge $A,B,D,E$
as above, with diagonal object $X$ which is not stable. By construction $X$ is an embedded interval, so it has a chain-of-interval decomposition $\{X_i\}$ under $\sigma$. There are two mutually exclusive cases:
\begin{enumerate}
    \item There are representatives for all the $X_i$ contained in the lozenge, ie. contained in the disk bounded by the lozenge or running along its sides.
    \item At least one of the representatives necessarily crosses out of the lozenge.
\end{enumerate}
The concatenation of the chain $\{X_i\}$ is isotopic to the object $X$. Therefore in case (2), if the chain crosses out of the lozenge along one of the sides it must cross back in, and along the \emph{same} side, since each of the objects $A,B,D,E$ cuts the surface into two. Therefore we have a configuration prohibited by Lemma \ref{lemma:noncrossing}.

As for case (1), every extension between $X_i$ and $X_{i+1}$ must happen at one of the marked components along the boundary of the lozenge. Note that even though the chain $\{X_i\}$ may not be simple (intervals could in principle double back), it must not cross itself by the same lemma, and therefore there are only two options: either $X_i$ and $X_{i+1}$ share a boundary component along the top of the lozenge (ie. along $A$ or $B$ sides) and the extension happens on the right, or it is along the bottom (ie. along $D$ or $E$ sides) and the extension happens on the left. Suppose that at least one of the intervals $X_i$ ends on the $A$ side; let $i$ be maximal among such indices. Then $X_{i+1}$ stretches between the $A$ side and another side of the lozenge, however its phase is smaller than $X_i$ so this contradicts the existence of a nontrivial extension on the right $\in \Ext^1(X_i,X_{i+1})$. The same argument can be applied along any of the other sides, in the case where no interval ends on the $A$ side. Therefore there cannot be more than one stable interval, and $X$ itself is stable.
\end{proof}

The lemma above should be interpreted as stating that the unobstructed lozenges ``see'' all the stable interval objects that were eliminated by cutting along $\gamma$.

\begin{proof}(of Theorem \ref{thm:inverses})
For clarity let us denote $\sigma = (Z, \cP) \in \Stab(\cF(\Sigma))^{int}$, $(\sigma_L,\sigma_R) = ((Z_L,\cP_L),(Z_R,\cP_R))$ for its image under the cutting map, and $\sigma_g = (Z_g,\cP_g)$ for the pre-stability condition glued out of $\sigma_L$ and $\sigma_R$. It is clear that the central charges $Z$ and $Z_g$ are the same; it is enough to check on a set of generators and we can pick the arc system $\cA_L \sqcup \cA_\gamma \sqcup \cA_R$ where the central charges agree by construction.

As for the (pre)slicings, the inclusions $\cP_g \subseteq \cP$ and $\cP \subseteq \cP_g$ follow from Lemma \ref{lemma:lozenges}.
\end{proof}

\subsection{Metric properties of the cutting map}
Recall from \cite{Bri1} that the topology on stability spaces $\Stab(\cD)$ can be defined with the use of a distance function
\[ d: \Stab(\cD) \times \Stab(\cD) \to \RR_{\ge 0} \cup \{\infty\} \]
defined on a pair of stability conditions by
\[ d(\sigma_1,\sigma_2) = \sup_{0 \neq X \in \cD} \left( |\phi^-_1(X) - \phi^-_2(X)|, |\phi^+_1(X) - \phi^+_2(X)|, \left| \log\frac{m_1(X)}{m_2(X)}\right| \right), \]
where as usual $\phi^+_i(X),\phi^-_i(X)$ are the maximum and minimum phases of the semistable components of $X$ under $\sigma_i$, and $m_i(X)$ is the mass of $X$ under $\sigma_i$, that is, the sum of $|Z(-)|$ over all semistable components of $X$ under $\sigma_i$.

The function $d$ is proven \cite[Prop.8.1]{Bri1} to give a generalized metric on $\Stab(\cD)$, ie. a metric on each connected component, with points in distinct connected components having infinite distance from each other.

Recall from Lemma \ref{lemma:uniqueminimal} that each point in a $\RelStab(\Sigma,\gamma)$ space corresponds to a unique minimal reduced stability condition $\sigma$ on some surface $\Sigma \cup_\gamma \Delta_n$ for some $n \ge 2$; moreover it also corresponds to stability conditions on larger surfaces $\Sigma \cup_\gamma \Delta_m$ with $m > n$, all related by the equivalence relation $\sim$.

We now extend the distance function $d$ to the spaces of relative stability conditions.
\begin{definition}
	The \emph{distance} between two relative stability conditions $\sigma_1,\sigma_2 \in \RelStab(\Sigma,\gamma)$ is defined as $d(\sigma_1,\sigma_2) = d(\tilde\sigma_1, \tilde\sigma_2)$ if there exists a surface $\tilde\Sigma = \Sigma \cup_\gamma \Delta_m$ with representatives $\tilde\sigma_1,\tilde\sigma_2 \in \Stab(\cF(\tilde\Sigma))$, and $+\infty$ otherwise.
\end{definition}
It follows from the definition of the topology on $\RelStab$ spaces that $d$ is finite if and only if $\sigma_1,\sigma_2$ are in the same connected component.

Given a decomposition $\Sigma = \Sigma_L \cup_\gamma \Sigma_R$, we now define a distance function on the space of compatible stability conditions $\Gamma \subset \RelStab(\Sigma_L,\gamma) \times \RelStab(\Sigma_R,\gamma)$, by setting
\[ d\left( (\sigma^L_1,\sigma^L_2), (\sigma^R_1,\sigma^R_2) \right) = \max\left( d(\sigma^L_1,\sigma^L_2), d(\sigma^R_1,\sigma^R_2) \right). \]

Consider now the cutting map of Section \ref{sec:cutting}. We now prove the following lemma.
\begin{lemma}\label{lem:distance}
	Let $\sigma_1,\sigma_2 \in \RelStab(\Sigma,\gamma)$, let us denote $\cut_\gamma(\sigma_i) = (\sigma^L_i, \sigma^R_i)$ for $i=1,2$. Then $d\left( (\sigma^L_1,\sigma^L_2), (\sigma^R_1,\sigma^R_2) \right)$ on $\Gamma$ is finite if and only if $d(\sigma_1,\sigma_2)$ on $\Stab(\cF(\Sigma))$ is finite.
\end{lemma}
\begin{proof}
	The if direction follows from the fact that the cutting map is continuous (Lemma \ref{lemma:cutcontinuous}), therefore it sends points in the same connected component to the same connected component.
	
    The only if direction is harder and relies on the support property (Definition \ref{def:supportproperty}). We split the proof in two parts; let us denote 
    \[ d'(\sigma_1,\sigma_2) = \sup_{0 \neq X \in \cD} \left( |\phi^-_1(X) - \phi^-_2(X)|, |\phi^+_1(X) - \phi^+_2(X)| \right) \]
    which depends only on the slicings, and
    \[ d''(\sigma_1,\sigma_2) =  \sup_{0 \neq X \in \cD} \left( | \log(m_1(X)/m_2(X))| \right), \]
    for the other part, also depending on the central charge. By definition, $d = \max(d',d'')$.
    
    Suppose that $d' \left( (\sigma^L_1,\sigma^L_2), (\sigma^R_1,\sigma^R_2) \right) = D' < \infty$. Let us denote by $\cP_1,\cP_2$ the slicings corresponding to $\sigma_1$ and $\sigma_2$. By \cite[Lem.6.1]{Bri1}, we have
    \[ d'(\sigma_1, \sigma_2) = \inf \{ \epsilon | \cP_1(\phi) \subseteq \cP_2([\phi-\epsilon, \phi+\epsilon]) \}, \]
    which also holds for the distances $d'$ on the relative stability spaces.
    
    For any fixed phase $\phi$, consider any semistable object $X \in \cP_1(\phi)$. There are three mutually exclusive options for $X$: either $X$ can be represented by a curve entirely contained in $\tilde\Sigma_L$, entirely contained in $\tilde\Sigma_R$, or it necessarily crosses $\gamma$. In the first two cases, $X$ can be decomposed under $\sigma_2$ by semistable objects also keeping to either side of $\gamma$ by Lemma \ref{lemma:noncrossing} and thus we have $X \in \cP_2(\phi-D',\phi+D')$ since every such object survives the cutting map. 
    
    For the other case, we have that since $X$ intersects $\gamma$ transversely, its phase satisfies $\phi \in [\phi^-_1(C) -1, \phi^+_1(C)+1]$, where $C$ is some fixed rank one object supported on the interval $\gamma$. Since $\gamma$ is contained on both categories $\cF(\tilde\Sigma_L)$ and $\cF(\tilde\Sigma_R)$, by assumption all its semistable components are contained in $\cP_2(\phi^-_1(C)-D', \phi^+_1(C)+D')$, so $X \in \cP_2(\phi^-_1(C)-D'-1, \phi^+_1(C)+D'+1)$, so it follows that $d(\sigma_1,\sigma_2) < \infty$.
    
    Let us now treat the distance function $d''$, which by assumption is finite between the two images of the cutting map. Let us prove by contradiction, assuming that $d''(\sigma_1,\sigma_2) = \infty$ on the source $\Stab(\cF(\Sigma)$ of the cutting map.
    
    Without loss of generality, we can then find an infinite sequence of objects $(X_t)_{t \in \ZZ_+}$ such that $m_1(X_t)/m_2(X_t) \to \infty$ as $t \to \infty$. Decomposing each $X_t$ into its $\sigma_2$-semistable components, we see that $m_1(X_t)/m_2(X_t)$ is bounded above by the maximum of that ratio over its semistable components; thus we can instead assume that every object $X_t$ in the sequence is $\sigma_2$-semistable.
    
    Either infinitely many of the $X_t$ are in the joint image of $\cF(\tilde\Sigma_L)$ and $\cF(\tilde\Sigma_R)$, or cross $\gamma$ non-trivially. The first two cases are impossible by the assumption that $d''\left( (\sigma^L_1,\sigma^L_2), (\sigma^R_1,\sigma^R_2) \right) < \infty$. So without loss of generality we can assume that each $X_t$ crosses the interval $\gamma$ some number $N_t > 0$ of times.
    
    Pick any grading on $\gamma$ and consider the corresponding rank one object $C$; each crossing $p$ with $X_t$ determines a map $C[s_p] \to X$ and $X \to C[s_p+1]$ for some shift $s_p \in \ZZ$. We now use all those crossings to write the exact triangle
    \[ \bigoplus_p (C[s_p]) \to X_t \to \bigoplus_I W_i, \]
    where $I$ is an indexing set with size $N_t$ (if $X_t$ is a circle) of $N_t + 1$ (if $X_t$ is an interval), and each $W_i$ is an object which stays to either side of $\gamma$. Note that in the construction above there is always a non-zero extension map $W_i \to C$; therefore, if $W_i$ is to the left, it only ends on the upper marked boundary of $\gamma$, and if to the right, on the lower marked boundary of $\gamma$.
    
    The mass function satisfies a triangle inequality with respect to exact triangles: if $A \to B \to C$ is an exact triangle, then $m(B) \le m(A) + m(C)$ \cite[Prop.3.3]{ikeda2016mass}. So we have the bounds
    \[ \sum_I m_1(W_i) \le m_1(X_t) + N_t m_1(C) \le \sum_I m_1(W_i) + 2 N_t m_1(C) \]
    and analogously for $m_2$. We then get
    \[ \frac{m_1(X)}{m_2(X) + N_t m_2(C)} \le \frac{\sum_I m_1(W_i) + N_t m_1(C)}{\sum_I m_2(W_i)}. \tag{*}\]
    
    We now use the support property of $\sigma_1$ and $\sigma_2$ to deduce from the inequality above that $\sum_I m_1(W_i)/\sum_I m_2(W_i) \to \infty$ as $t \to \infty$; for that let us choose a norm on $K_0(\cF(\Sigma))$ that will make the argument simpler. We pick a full arc system $\cA = \cA_L \sqcup \cA_\gamma \sqcup \cA_R$ on $\Sigma$ as follows: the elements of $\cA_\gamma$ are the intervals appearing in the decomposition of $\gamma$ under $\sigma_2$, every arc in $\cA_L$ only possibly shares the upper marked boundary of $\gamma$, and every arc in $\cA_R$ only possibly shares the lower marked boundary of $\gamma$. This guarantees that each object $W_i$ is either generated by the arcs in $\cA_L$ or by the arcs in $\cA_R$. We pick a norm where the images of $\cA_L,\cA_\gamma,\cA_R$ are pairwise orthogonal subspaces. Also, by the support property on $\sigma_2$ and finiteness of the rank of $K_0$, $m_1(X_t)/m_2(X_t) \to \infty$ implies $m_1(X) \to \infty$.
    
    We note that in the chosen norm we have $\lVert X_t \rVert \ge N_t \lVert C \rVert$; therefore the ratio $N_t/\lVert X_t \rVert$ stays bounded above. From the support property and the assumption $m_1(X_t)/m_2(X_t) \to \infty$, we conclude that the lhs of $(*)$ goes to $\infty$.

    Note also that on either side of $\gamma$, there must a shortest nonzero object (with respect to the chosen norm) which links an end of $\gamma$ to itself, by finiteness of the rank of $K_0$. Let $\ell > 0$ be the shortest of those lengths; then we know that $\sum_I m_2(W_i) \le \epsilon(N_t-1)\ell$ (for $X_t$ interval) $\le \epsilon N_t \ell$ (for $X_t$ circle) for some constant $\epsilon$, again from the support property. Either way, $N_t m(C)/\sum_I m_2(W_i)$ is bounded above, implying $\sum_I m_1(W_i)/\sum_I m_2(W_i) \to \infty$ as $t \to \infty$. 
    
    The mass triangle inequality and the pigeonhole principle then imply that there is a sequence of objects $W'_t$ either in $\cF(\tilde\Sigma_L)$ or $\cF(\tilde\Sigma_R)$ such that $m_1(W'_t)/m_2(W'_t) \to \infty$ as $t \to \infty$, contradicting the assumption that the images under the cutting map are at a finite distance.
\end{proof}

\section{Application of cutting/gluing relations}
In the previous section, we defined the cutting and gluing map and discussed some of its properties. We now put all those results together with calculations of a number of base cases to prove the main result of this paper, namely, that every stability condition on $\cF(\Sigma)$ is an HKK stability condition.

\subsection{The induction argument}
As before, let $\Sigma = \Sigma_L \cup_\gamma \Sigma_R$ be a marked surface divided by an embedded interval.

\begin{lemma}\label{lem:inductivestep}
If, for any extended surfaces $\tilde\Sigma_L = \Sigma_L \cup_\gamma \Delta_m$ and $\tilde\Sigma_R = \Sigma_L \cup_\gamma \Delta_n$, every stability condition on $\cF(\tilde\Sigma_L)$ and $\cF(\tilde\Sigma_R)$ is HKK, then every stability condition on $\cF(\Sigma)$ is HKK.
\end{lemma}
\begin{proof}
Recall that we know from \cite{HKK} that the subset of HKK stability conditions is an union of connected components.

Recall the cutting map:
\[ \cut_\gamma: \Stab(\cF(\Sigma)) \to \Gamma \subset \RelStab(\Sigma_L,\gamma) \times \RelStab(\Sigma_R,\gamma). \]
As before we write $\Gamma^{int} \subset \Gamma$ for the image of the locus of interval-like stability conditions $\Stab(\cF(\Sigma))^{int}$ under the cutting map.

By assumption, every pair of stability conditions on some extended surfaces $\tilde\Sigma_L,\tilde\Sigma_R$ is a pair of HKK stability conditions, so any pair representing an element of $\Gamma_{int}$ is given by two flat surfaces of the type specified in \cite[Prop.6.1]{HKK}, that is, given by an \emph{S-graph} with a central charge on each edge. This is a graph whose vertices correspond to the horizontal strips of the flat surface.

Let $(\sigma_L,\sigma_R) = \cut_\gamma(\sigma)$. By Lemma \ref{lemma:uniqueminimal} we can assume that $\sigma_L,\sigma_R$ are stability conditions on surfaces $\tilde\Sigma_L,\tilde\Sigma_R$ which are minimal and reduced with respect to $\gamma$. 

As in the proof of \cite[Prop.6.2]{HKK}, we can find two continuous families $\sigma^t_{L,R}, t \in [0,1]$ of HKK stability conditions such that $\sigma^0_{L,R} = \sigma_{L,R}$ and $\sigma^1_{L,R}$ has the property that all simple saddle connections of the corresponding flat surface are vertical, that is, all stable objects are given by intervals with phases $\phi \in \ZZ+\frac{1}{2}$. Moreover from the description of these flat surfaces by S-graphs, we can choose these two families such that $(\sigma^t_L,\sigma^t_R)$ are compatible along $\Gamma$.

From the gluing map we know that $\glue_\gamma(\sigma^1_L,\sigma^1_R)$ is an HKK stability condition on $\sigma$, so from Theorem \ref{thm:inverses}, the pair $(\sigma^1_L,\sigma^1_R)$ is the image of the cutting map, and is therefore is in $\Gamma^{int}$ since it only has stable intervals.

Now from the metric properties of $\cut_\gamma$ we know that $\sigma$ and $\glue_\gamma(\sigma^1_L,\sigma^1_R)$ are in the same component of $\Stab(\cF(\Sigma)$, but by construction $\glue_\gamma(\sigma^1_L,\sigma^1_R)$ is an HKK stability condition, and so is $\sigma$.
\end{proof}

We can use the Lemma above to run an induction argument, by using the following decomposition. Consider some general surface $\Sigma$ with genus $g$ and punctures $p_0,p_1,\dots,p_n$ with $m_0,m_1,\dots,m_n$ marked boundaries, respectively, with $m_i \ge 1$. We can then decompose the surface into a disk with some number of marked boundary intervals, possibly some annuli with two marked boundary intervals on the outer boundary circle, and possibly some punctured tori with two marked boundary intervals on the boundary circle.

\begin{figure}[h]
    \centering
    \includegraphics[width=\textwidth]{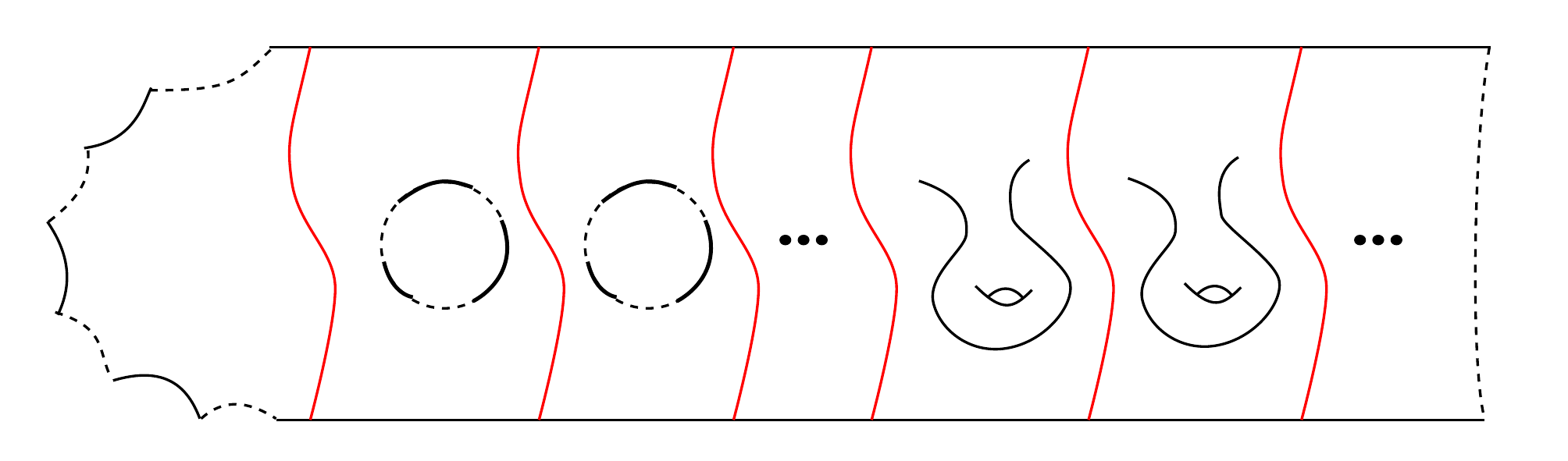}
    \caption{A decomposition of the surface $\Sigma$ into a disk, possibly several annuli and possibly several punctured tori.}
\end{figure}

We then use Lemma \ref{lem:inductivestep} as the inductive step; using as index the number $(g + n)$ of punctured tori plus annulus appearing in the decomposition above. It remains to prove in each of the three following types of base cases that all stability conditions are HKK:
\begin{enumerate}
    \item The disk $\Delta_n$ with $n \ge 2$ marked boundary intervals,
    \item The annulus $\Delta^*_{p,q}$ with $p,q$ marked boundary intervals on the outer and inner boundary circle, respectively,
    \item The punctured torus $T^*_n$ with $n$ marked boundary intervals.
\end{enumerate}

\subsection{Calculations}\label{sec:calculations}
By the main theorem of \citep{HKK} (Theorem 5.3) the locus of HKK stability conditions in $\Stab(\cF(\Sigma))$ is a union of connected components. Thus, if every stability condition can be continuously deformed into an HKK stability condition, then all stability conditions are HKK stability conditions.

We will use this strategy for the three base cases; in fact we will prove that every stability condition can be continuously deformed to a stability condition with finite heart. This argument already appears for the case of the disk and the annulus in \citep{HKK}; we will reproduce it in greater detail so that its use in the context of the punctured torus is clearer.

\subsection{Finite-heart stability conditions}
The definitions and lemmas here seem to be standard in the literature to some extent and may appear with different formulations; for clarity we will assemble them here.

\begin{definition}
A stability condition $\sigma \in \Stab(\cD)$ is \emph{finite-heart} if the corresponding heart $\cH$ is a finite abelian category, ie. a finite length abelian category that furthermore only has finitely many isomorphism classes of simple objects.
\end{definition}

Note that finite-length only means that every object is finite-length but those lengths could be unbounded; this doesn't happen in the cases we care about because of the following standard fact.
\begin{lemma}
If $\cH$ is finite-length and $\rk(K_0(\cH)) = \rk(K_0(\cD)) < \infty$ then $\cH$ is finite, and in particular the number of isomorphism classes of simple objects is equal to $\rk(K_0(\cD))$.
\end{lemma}

We have the following criterion to determine when some stability condition is finite-heart, based on the set of stable phases $\Phi \in S^1$, ie. the set of phases of stable objects.

\begin{lemma}\label{lemma:gap}
If $\Phi$ has a gap around zero (ie. $S^1 \setminus \Phi$ contains an open interval $I \ni 0$) and $K_0(\cD) < \infty$ then $\sigma$ is finite-heart.
\end{lemma}

\begin{remark}
This fact is used in \citep{HKK} but left unstated. The clear statement and proof of this lemma were informed to me by F. Haiden.
\end{remark}
\begin{proof}
Note that $\phi$ is symmetric under a $\ZZ_2$ rotation so having a gap around zero means that $\Phi$ is contained in a strict cone in the upper half-plane. Thus there is $K > 0$ such that $|\Im(Z(X))| > K \cdot |\Re(Z(X))| $ for any semistable object $X$. We will argue that the set of semistable imaginary parts
\[ \{ \Im(Z(E)) | 0 \neq E \in \cP_\phi, \phi \in \RR \} \]
is discrete. Suppose that there is an accumulation point, which without loss of generality we assume to be $a>0$; we can then pick a sequence of pairwise non-isomorphic semistable objects $\{E_n\}$ such that $\lim_{n \to \infty} |\Im(Z(E_n)) - a| = 0$; in particular for $\delta > 0$ we can pick the sequence such that $|\Im(Z(E_n)) - a| < \delta$ for every $n$, so picking $0 < \delta < a$ gives $|\Re(Z(E_n))| < K(a + \delta)$

But since $\Lambda$ is finite rank and the $E_n$ are all distinct, we have $\lim_{n\to \infty} \lVert E_n \rVert = \infty$. We then have
\[ |Z(E_n)| < |\Im(Z(E_n))| + |\Re(Z(E_n))| \le (K+1)|\Re(Z(E_n))| \le
(K+1)K(a + \delta). \]
So we have $\lim_{n\to \infty} \frac{|Z(E_n)|}{\lVert E_n \rVert} = 0$ contradicting the support condition.

So since the set of imaginary parts of objects in the heart $\cH$ is discrete and bounded below by zero, any strictly descending chain of objects is finite, and therefore $\cH$ is finite-length, and thus $\sigma$ is finite-heart by the assumption $\rk(K_0(\cD)) < \infty$.
\end{proof}

Using the formalism of S-graphs presented in Section 6 of \citep{HKK}, one can prove the following lemma (which is implicitly used in the proofs of Theorems 6.1 and 6.2 of that same paper)
\begin{lemma}
If $\sigma$ is a finite-heart stability condition on $\cF(\Sigma)$ then it is an HKK stability condition.
\end{lemma}

For each of the three base cases, we will see that every stability condition can be deformed to a finite-heart stability condition.

\subsection{The disk}(Section 6.2 of \citep{HKK})
We have $\cF(\Delta_n) \cong \Mod(\AA_{n-1})$, which up to shift has finitely many indecomposable objects. Thus any heart is a finite abelian category, and every stability condition is finite-heart and therefore HKK.

\subsection{The annulus}
There are two different kinds of annulus; one where the nontrivial circle is gradable, ie. has index zero, and one where it has index nonzero. Consider first the annulus $\Delta^*_{p,q,(m)}$ with $p,q$ marked boundary components and grading $m\neq 0$ around the circle.

We argue that the set of stable phases is finite. Let us fix some embedded interval object $I_0$ to have winding number zero, and measure the winding number of every other interval or circle with reference to it. By the classification of objects, there are only finitely many primitive (ie. non multiple) classes in $K_0(\cF(\Delta^*_{p,q,(m)}))$ whose winding number is less than some fixed $N$ in absolute value, so if there are infinitely many non-isomorphic stable objects there must be a sequence of stable objects $X_i$ with winding number $\to \infty$.

Consider some object $X_i$ with winding number $N_i$ which intersects $I_0$ transversely $N_i$ many times. Since the circle has index $m\neq 0$, this contributes classes to both $\Ext^*(I_0,X_i)$ and $\Ext^*(X_i,I_0)$ in a range spanning $(m-1) N_i$ degrees. But this is impossible as $N_i \to \infty$ since the stable components of $I_0$ have a minimum and maximum phase.

Consider now the annulus with zero grading. We have $\cF(\Delta^*_{p,q,(0)}) \cong \Mod(\tilde\AA_{p+q-1})$. So we have $\Gamma = K_0(\cF(\Delta^*_{p,q,(0)})) = \ZZ^{p+q}$, and denote by $S \subset \Gamma$ the subgroup generated by the circle around the annulus. Let $E \subset \Gamma$ be the set of classes of indecomposable objects. By the classification of objects $E/S$ is finite so the only possible accumulation point in the set of stable phases $\Phi$ is $\arg(Z(S))$. After a rotation (which can be arbitrarily small) we can guarantee that $\Phi$ has a gap around zero and apply Lemma \ref{lemma:gap}.

\subsection{The punctured torus}
The calculation of this case is new. From the cutting procedure we know that only need to consider the punctured torus $T^*_n$ with $n \ge 2$ marked boundary components. In fact there are many inequivalent such punctured tori, with different gradings. Let us pick simple closed curves $L$ and $M$ as longitude and meridian, and denote by $i_L,i_M$ the index of the grading along them. By picking different curves we get indices differing by an action of $\SL(2,\ZZ)$ so the set of distinct graded punctured tori is $\ZZ^2/\SL(2,\ZZ)$. The orbits of $\SL(2,\ZZ)$ on $\ZZ^2$ are labelled by gcd, so each orbit contains a unique pair of the form $(0,m)$.

Let us fix a grading such that $(i_L,i_M) = (0,m)$. It will be important for us to know what are the circle objects. The classes in $\pi_1(T^*)$ which are representable by simple closed curves are the curves winding $(p,q)$ times around the longitude and meridian, with $\text{gcd}(p,q) = 1$, plus the curve $MLM^{-1}L^{-1}$, ie. the circle around the puncture.

For any of these tori, the index of the circle around the puncture is always $2$ for topological reasons (it bounds a punctured torus) so this curve is never gradable. On the torus with $(i_L,i_M) = (0,m \neq 0)$ torus the index of the $(p,q)$ curve is $mq \neq 0$ if $q \neq 0$, so all of the embedded circle objects are supported on the longitude $L$. On the torus with $(i_L,i_M) = (0,0)$, every simple closed curve is gradable and supports embedded circle objects.

\begin{remark}
This is the fundamental reason why the calculation for the $(0,0)$ will be more involved than the case of the annulus; in that case the lattice spanned by the circle objects inside of $K_0(\cD)$ is rank one, so there can be at most one direction of phase accumulation. In the punctured torus, the central charges of stable objects could in principle occupy every direction of the lattice, making $\Phi$ dense; we will prove that this doesn't happen generically.
\end{remark}

\subsubsection{The $(0,m \neq 0)$ torus}
Let us denote $\cD = \cF(T^*_{n,(0,m)})$ where $n$ is the number of marked boundaries. This case will be very similar to the index zero annulus. There is only one type of embedded circle object $L$, since no other circles are gradable. Let $\Gamma = K_0(\cD)$ and $E \subset \Gamma$ be the set of classes of stable objects.

We argue that the set $E/\langle L \rangle$ is finite. Suppose otherwise, and note that by the classification of embedded curves, the number of embedded curves with winding numbers $(p,q)$ with $|q \le N|$ is infinite, but they form finitely many orbits in $K_0(\cD)$ under the action of the subgroup $\langle L \rangle$. Thus, if we have an infinite sequence of stable objects $\{E_i\}$ with winding numbers $(p_i,q_i)$ and pairwise distinct classes $[E_i] \in K_0(\cD)/\langle L \rangle$, there is a subsequence with $\lim_{i\to\infty} |q_i| = \infty$.

This is impossible in any stability condition. Note that an object with winding $q_i$ along the meridian intersects $L$ transversely $|q_i|$ times; but since $m \neq 0$ the difference in degree between each two consecutive intersections is $|m|$, so the amplitude of nonzero degrees in both $\Hom(E_i, L)$ and $\Hom(L,E_i)$ is $m(q_i-1)$. Since $|q_i| \to \infty$ we can find stable objects $E_i$ with arbitrarily large amplitude morphisms in both directions which is impossible since $L$ has some HN decomposition with finitely many semistable factors, having a minimum and a maximum phase.

From the fact that $E/\langle L \rangle$ is finite we can proceed as in the annulus case, and after an infinitesimal rotation we can guarantee that any stability condition has an gap in $\Phi$.

\subsubsection{The $(0,0)$ torus}
Let us denote $\cD = \cF(T^*_{n,(0,0)})$, where $n$ is the number of marked boundaries. We will first need some facts about $K_0(\cD)$. By Theorem 5.1 of \citep{HKK} there is an isomorphism
\[ K_0(\cF(\Sigma,M)) = H_1(\Sigma,M;\ZZ_\tau), \]
where $\ZZ_\tau$ is the $\ZZ$-local system  associated to the orientation double cover of the foliation. In our case, since we are looking at the foliation with $(0,0)$ winding, $\ZZ_\tau$ is trivial.

Let us pick an explicit set of generators of $K_0(\cD)$ as below: first choose a basis of $H_1(T,\ZZ)$ and a labeling $M_1,\dots,M_N$ of the marked boundary components.
The classes $[L]$ and $[M]$ are represented by circles around the longitude and meridian, and $[E_i], i = 1,\dots,N$ is represented by intervals that connect adjacent $M_i$ and $M_{i+1}$ along the boundary. Consider the object $X$ winding around the longitude with ends at $M_1,M_N$. Extending it by $E_1,\dots,E_{n-1}$ and by $E_n$ both give $L$, so in $K_0$ we have $\sum_{i=1}^n [E_i] = 0$

So the classes $[L],[M],[E_1],\dots, [E_{n-1}]$ give a basis of $K_0(\cD)$. Since every immersed curve has well-defined winding numbers, we have a projection map
\[ w: K_0(\cD) {\longrightarrow} \ZZ^2 \cong \ZZ\cdot[L] \oplus \ZZ\cdot[M] \]
taking a curve of $(p,q)$ winding numbers to $p[L] + q[M]$. The following lemma tells us that the distribution of stable phases is not essentially changed by $w$.

\begin{lemma}
For any sequence of stable objects $\{X_k\}$ (with all $X_k$ pairwise distinct) if $\lim_{k\to \infty} \arg(Z(X_k))$ exists then
\[ \lim_{k\to\infty} \arg(Z(w([X_k]))) = \lim_{k\to \infty} \arg(Z(X_k)). \]
\end{lemma}
\begin{proof}
By the classification of indecomposables, $X$ is represented by some circle or interval with winding $(p,q)$. If $X$ is a circle we already have $w([X]) = [X]$. Given embedded interval with boundaries on $M_1,M_i$, one can express it as the concatenation of $p$ copies of the interval winding along the longitude with both ends at $M_1$ (whose class is $[L]$), $q$ copies of the interval winding along the meridian with both ends at $M_1$ (whose class is $[M]$) and a chain of intervals $E_1,\dots, E_{i-1}$ connecting $M_1$ to $M_i$. This chain can wind around the circle any number of times, but since $\sum_{j=1}^n [E_j] = 0$, its class is always $[E_1] + \dots + [E_i]$. Applying $|Z(\cdot)|$, since this sum is bounded above we have
\[ |Z(X) - Z(w(X))| \le C \]
for some fixed constant $C$.

Consider now the stable objects $X_k$. Without loss of generality suppose that $\lim_{k\to \infty} \arg(Z(X_k)) = 0$ (ie. the positive real direction). These objects can be represented by embedded intervals; note that there are finitely many embedded intervals with fixed winding numbers. Thus in the infinite sequence of distinct objects $\{X_k\}$ we must have $p_k^2 + q_k^2 \to \infty$ so therefore $|Z(X_k)| \to \infty$ and $\Re(Z(X_k)) \to +\infty$.

The triangle inequality,
\[ |Z(X_k)| - C \le |Z(w(X_k))| \le |Z(X_k)| + C \]
also implies similar inequalities for the real and imaginary parts. Since $|\Re(Z(X_k))| \to \infty$ we have
\[
    \lim_{k\to\infty} \frac{|\Im(Z(w(X_k)))|}{|\Re(Z(w(X_k)))|} \le
    \lim_{k\to\infty} \frac{|\Im(Z(X_k))| + C}{|\Re(Z(X_k))| - C}
    = \lim_{k\to\infty} \frac{|\Im(Z(X_k))|}{|\Re(Z(X_k))|} = 0.
\]
so $\lim_{k\to \infty} \arg(Z(w(X_k))) = \lim_{k\to \infty} \arg(Z(X_k)) = 0$.
\end{proof}

\begin{corollary}\label{cor:wphases}
If the set of stable phases $\Phi$ is dense in $S^1$ then the set
\[ \Phi_w = \left\{ \arg(Z(w(X))) \mid X \mathrm{stable}  \right\}\]
is also dense in $S^1$
\end{corollary}

We will also need to know a bit more about which objects necessarily intersect transversely.
\begin{lemma}
Consider two embedded objects  $X$ and $Y$ with winding numbers $(p_X,q_X)$ and $(p_Y,q_Y)$, respectively. If $|p_X q_Y - q_X p_Y| \ge 2$ then $X$ and $Y$ intersect transversely.
\end{lemma}
\begin{proof}
If $X$ is a circle with $(p_X,q_X) = (1,0)$, then any $Y$ with $|q_Y| \ge 1$ intersects $X$ transversely; if $X$ is an embedded interval with $(p_X,q_X) = (1,0)$, then circles with $|q_Y| \ge 1$ intersect $X$ transversely but intervals with $|q_Y| = 1$ may not. On the other hand, winding more times around the meridian by requiring $|q_Y| \ge 2$ necessarily causes a transverse intersection. Applying the right element of $\SL(2,\ZZ)$ that sends $(1,0) \mapsto (p_X,q_X)$ gives the statement of the lemma.
\end{proof}

The following lemma gives an existence result for a certain kind of stable object.
\begin{lemma}
Let $\sigma \in \Stab(\cD)$ be a stability condition on $\cD = \cF(T^*_N)$. Then there is some stable object represented by an embedded interval with nonzero winding and ends at different marked boundaries.
\end{lemma}
\begin{proof}
Suppose otherwise; by the classification of embedded curves, there are three remaining possibilities for a stable object:
\begin{enumerate}
    \item A semistable circle with winding $\neq (0,0)$,
    \item A semistable interval with winding $\neq (0,0)$ both ends on the same marked boundary,
    \item A semistable interval with $(0,0)$ winding and ends possibly on different marked boundaries.
\end{enumerate}

Two objects of type (2) ending on the same marked boundary $M$ will have extension morphisms between them, but we argue that if they have different classes in $K_0(\cD)$ these morphisms cannot appear in the HN decomposition of any object. By keeping track of the grading with respect to the $(0,0)$ grading on the torus, we note that if we grade the intervals such that $\deg(f) = 1$, then $\deg(g) = 0$. Thus $\phi_B \le \phi_A$ and by genericity $\phi_B \neq \phi_A$ since $[A] \neq [B]$, so $f \in \Ext^1(A,B)$ cannot appear in the HN decomposition.

Thus every interval with winding $(p,q)$, $\gcd(p,q) = 1$ and ends on the same marked boundary must be semistable, since there is no way to express it as a valid extension of the objects above. We argue that this is impossible in a generic stability condition. Take for example the semistable interval $J$ with winding $(1,0)$ and both ends on some marked boundary $M$, and consider another embedded interval $J'$ with winding $(0,1)$, with ends on $M$ and $M' \neq M$.
\begin{figure}[h]
    \centering
    \includegraphics[width=\textwidth]{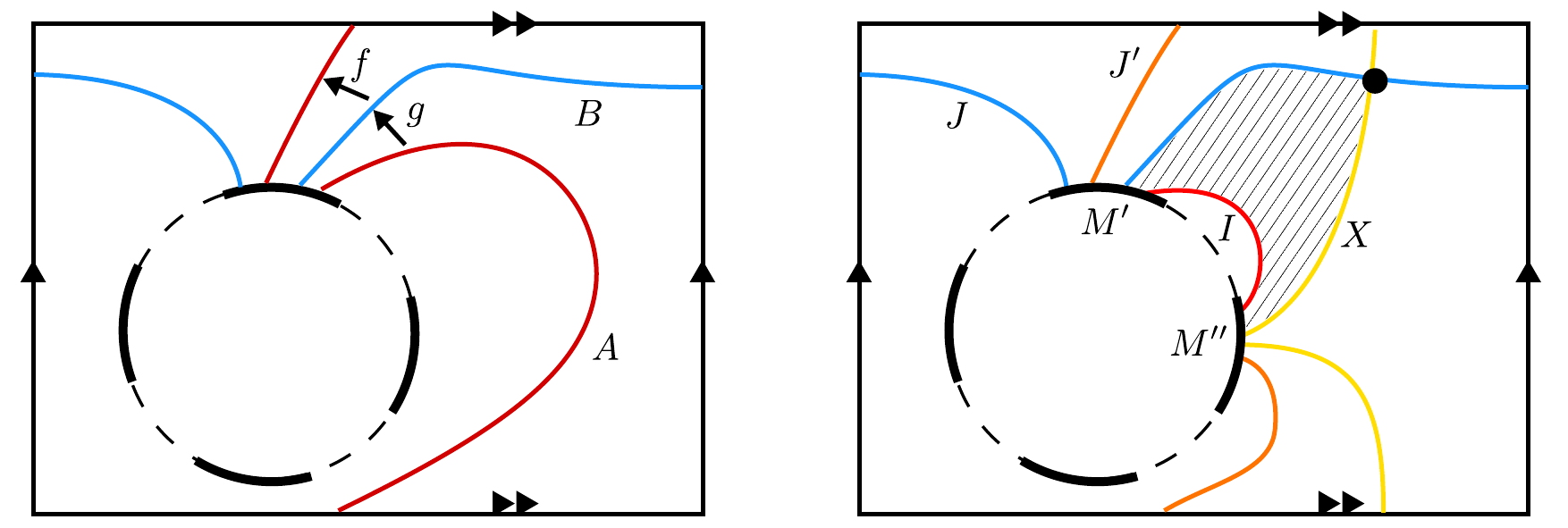}
    \caption{Left: two stable objects $A,B$ of type (2). Right: one stable object $J$ of type (2) and one (not semistable) embedded interval $J'$, in whose decomposition some object $I$ of type (3) must appear, causing a prohibited polygon (shaded) to appear.}
\end{figure}
By assumption, $J'$ is not semistable so it must have a chain-of-intervals decomposition with at least two distinct phases; consider the interval objects in this chain that end at $M$; since the other end of the chain is at another marked boundary, among these objects there must be at least one semistable interval $J'_0$ of type (3) above (ie. with zero winding). We see immediately that such an interval has an essential transversal intersection with $J$; therefore the rest of the chain (after $J'_0$) must cross $J$ as well. But this configuration is prohibited by Lemma \ref{lemma:noncrossing}.

So there must be some semistable interval object $I'$ with nonzero winding and ends on different marked boundary intervals. If $I$ is not stable, consider its Jordan-H\"older filtration into stable objects; among these there must be one stable interval object $I$ connecting two distinct marked boundaries.
\end{proof}

Using the lemmas above, in the following calculation we show that an adequately generic stability condition does not have dense phases in $S^1$.
\begin{lemma}
Let $\sigma \in \Stab(\cD)$ be a stability condition on $\cD = \cF(T^*_N)$. Then possibly after a infinitesimal deformation the set of stable phases $\Phi$ has a gap, ie. $S^1 \setminus \Phi$ contains an open interval.
\end{lemma}
\begin{proof}
By the previous lemma, there must be some stable interval $I$ with nontrivial winding and ends on distinct marked boundary components. Applying an appropriate $\SL(2,\ZZ)$ automorphism, we can assume this stable interval $I$ has winding numbers $(1,0)$, ie. winds around the longitude once. Let $L$ be the rank one trivial circle object also with winding number $(1,0)$.

The subset of $\Stab(\cD)$ where $I$ is stable is open by standard results \citep{bridgeland2015quadratic} so there is a neighborhood $U$ of $\sigma$ where $I$ is stable. From the description of $K_0(\cD)$ we know that $[I] \neq [L]$, so $Z(I),Z(L)$ are not parallel in the complement of a codimension one wall. Thus, possibly after an infinitesimal deformation inside of $U$, we can guarantee that $I$ is stable and $Z(I),Z(L)$ have different arguments.

Consider the trivial rank one objects $L$ and $M$ (which may or may not be stable) supported along the longitude and meridian, with gradings so that
\[ \deg(M,I) = \deg(M,L) = 0 \]
and for simplicity let us rotate and scale the stability condition so that $Z(L) = 1$. Since $[L] \neq [M]$ and we fixed $Z(L) \in \RR$, for a generic stability condition we must have $Z(M) \notin \RR$. Let us treat the case $\Im(Z(M)) > 0$ first; the other case follows from an analogous argument.

Suppose now that $\Phi$ is dense in $S^1$; by Lemma \ref{cor:wphases} ,$\Phi_w$ is dense too. For a choice of winding numbers $(p,q)$, let us denote by
\[ \cX_{p,q} = \{(p',q') \mid q > 0, |p q' - q p'| \ge 2 \} \subset \ZZ^2\]
the set of winding numbers whose objects necessarily intersect transversely with objects of winding number $(p,q)$, with positive winding around the meridian.

The set $\cX_{1,0}$ corresponding to $I$ is given by $q \ge 2$; so at infinity $\cX_{1,0}$ approaches a sector (with angle $\pi$). Remember that for any $N$ there are only finitely many indecomposable objects with winding satisfying $p^2 + q^2 \le N$. By density of $\Phi_w$ we can find some stable object $X_0$ with winding numbers $(p_0,q_0) \in \cX_{1,0}$.

\begin{figure}[h]
    \centering
    \includegraphics[width=\textwidth]{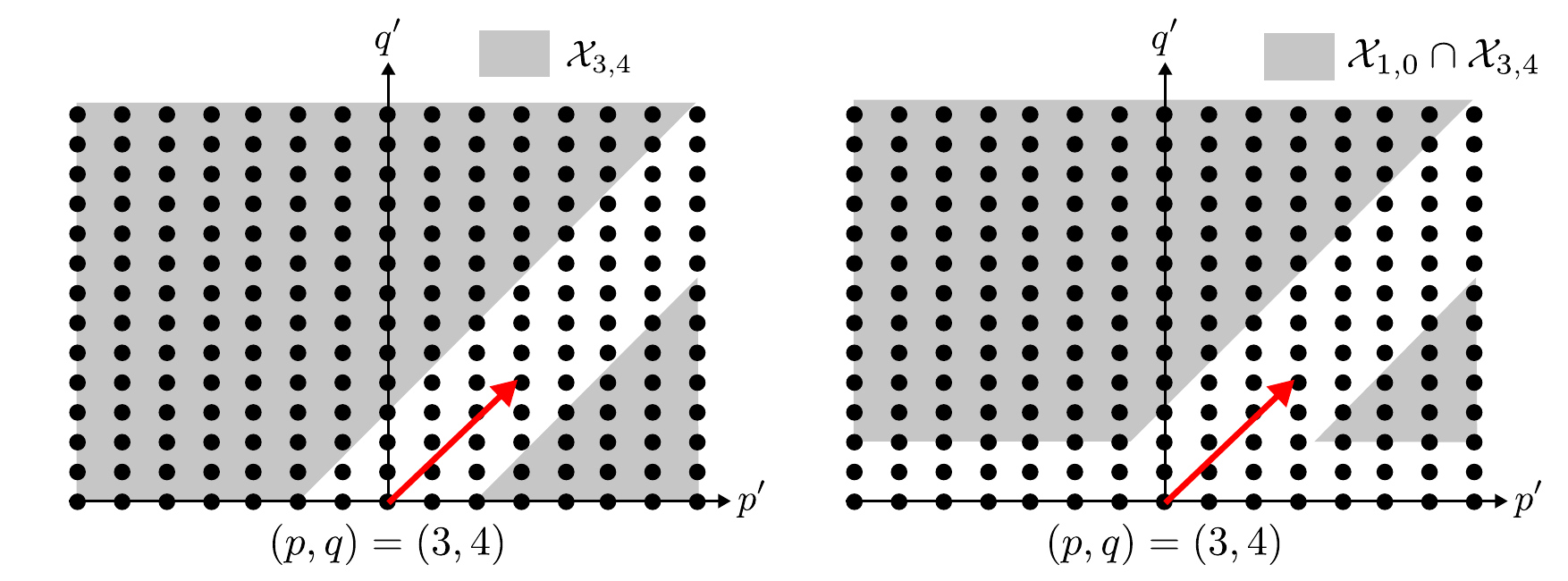}
    \caption{Left: the set $\cX_{p,q}$ for $(p,q) = (3,4)$ is composed of the $\ZZ^2$ dots inside of the shaded area. Note that all these sets have two parts, each of which at infinity approaches a sector with finite angle. Right: after the first iteration we consider $\cX_{1,0} \cap \cX_{3,4}$. Note that after any number of iterations the each side of this set still approaches a sector with finite angle at infinity.}
\end{figure}

Consider now the set $\cX_{1,0} \cap \cX_{p_0,q_0}$; this set is composed of lattice points inside of two components of a subset of $\RR \times \RR_+$. At infinity, the right component approaches a sector with angle spanning $(0,\arctan(q_0/p_0))$ and the left component approaches a sector at $(\arctan(q_0/p_0),\pi)$. Note that here we are choosing $\arctan$ to be valued between $0$ and $\pi$. Using density, let us pick some object $X_1$ with $(p_1,q_1)$ in the right component, and $X_{-1}$ with $(p_{-1},q_{-1})$ in the left component. Note that since the sectors span positive angles we can pick these objects with $q_1, q_{-1}$ arbitrarily large; since
\[ |\Im(Z(X)) - \Im(Z(w(X))| = |\Im(Z(X)) - q_X \Im(Z(M))|\]
is bounded for any indecomposable object $X$ we can also guarantee that $\Im(Z(X_1))$ and $\Im(Z(X_{-1}))$ are positive.

We would like to iterate this process; at the $n$th step we will have objects $\{X_k\}_{-n \le k \le n}$ with winding numbers $(p_k, q_k)$ running clockwise in angle, ie. $0 \le \arctan(q_k/p_k) \le \pi$ is decreasing. The set
\[ \cX_{1,0} \cap \cX_{p_{-n},q_{-n}} \cap \dots \cap \cX_{p_0,q_0} \cap \dots \cap \cX_{p_n,p_q}  \]
at infinity approaches two sectors at $(0,\arctan(q_n/p_n))$ and $(\arctan(q_{-n}, p_{-n}), \pi)$; since each of these sectors has nonzero angle we can use density and repeat the process by picking stable objects $X_{-n-1}, X_{n+1}$ in each sector, also both with central charge with positive imaginary part. Also from density of $\Phi$ it follows that we can pick objects such that
\[ \lim_{k\to +\infty} \arctan(q_k/p_k)) = 0, \quad \lim_{k \to -\infty} \arctan(q_k/p_k)) = \pi. \]

Iterating to infinity we get stable objects $\dots,X_{-1},X_0,X_1,\dots$ all mutually transversely intersecting, that also transversely intersect $I$ as well. Taking appropriate shifts we can guarantee that all these objects have phases $0 \le \phi_k \le 1$. We then get that
\[ \lim_{k\to +\infty} \phi_k = 0, \quad \lim_{k \to -\infty} \phi_k = 1. \]

\begin{figure}[h]
    \centering
    \includegraphics[width=\textwidth]{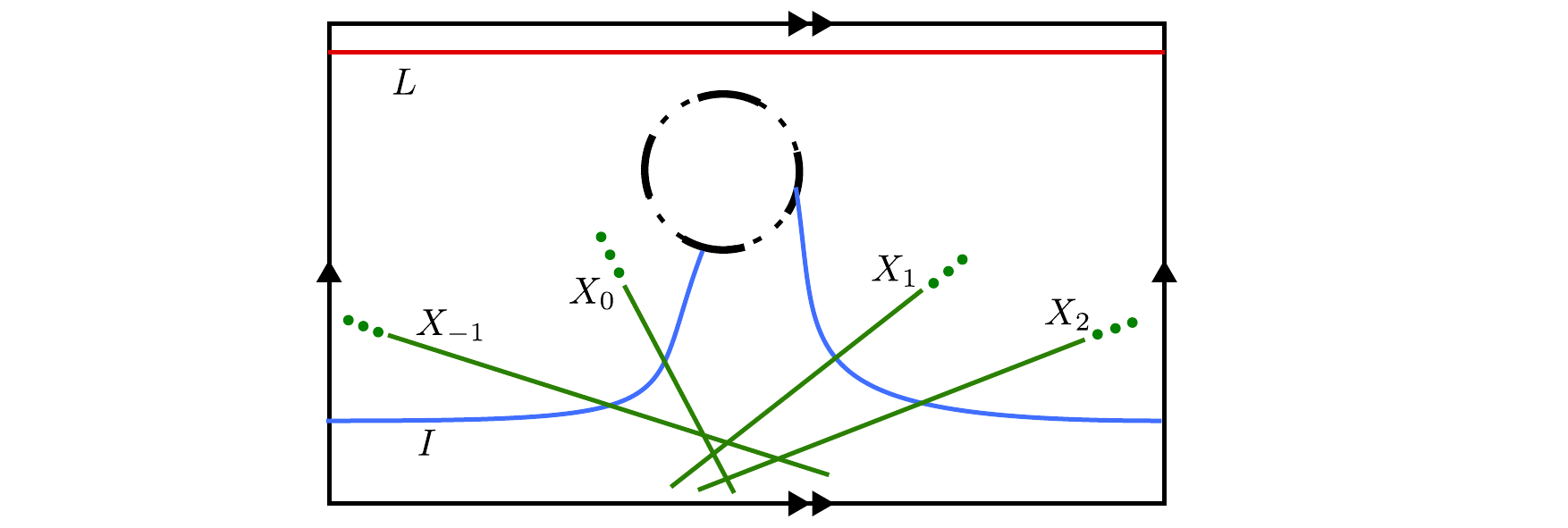}
    \caption{Stable circle $L$ and stable interval $I$ with ends on different boundary components, together with transversely intersecting stable objects $X_i, i \in \ZZ$.}
\end{figure}

Let $d_k$ be the degree of the intersection between $X_k$ and $I$, and $f_k$ be the degree of the intersection between $X_k$ and $X_{k+1}$. Let us shift $I$ such that $d_0 = -1$. The triangles with sides $X_k,X_{k+1},I$ give the relations $d_k = d_{k+1} + f_k$. Since all the objects are stable we have inequalities for the phases
\[ \phi_k \le \phi_I + d_k \le \phi_k + 1, \qquad \phi_k \le \phi_{k+1} + f_k \le \phi_k + 1. \]
But we chose the shifts such that all the $\phi_k$ are in $(0,1)$, so we must have $f_k = 0$ for all $k$, and therefore $d_k = -1$ for all $k$, so $\phi_k -1 \le \phi_I \le \phi_k$.

Taking the two limits $k\to +\infty$ and $k \to -\infty$ gives us $\phi_I = \phi_L = 0$ which contradicts the genericity of $\sigma$.
\end{proof}

\section{Conclusions}\label{sec:future}
The calculations for the three base cases above show that the every generic stability condition on those categories is an HKK stability condition; because of \citep[Theorem 5.3]{HKK} the image of the moduli of HKK stability conditions in $\Stab(\cD)$ is an union of connected components, so for all these cases there are only HKK stability conditions.

Together with the inductive step given by Lemma \ref{lem:inductivestep}, this proves Theorem \ref{thm:noexotic}: every stability condition on a graded surface $\Sigma$ is an HKK stability condition, ie. given by a quadratic differential with essential singularities.

\subsection{Future directions}

It is likely that this kind of construction could be extended beyond Fukaya categories of surfaces; this motivates many possible directions of future study. One obvious such direction is towards extending the definition of relative stability conditions to wrapped Fukaya categories of higher-dimensional symplectic manifolds, which appear in the work of Abouzaid \citep{abouzaid2010open,abouzaid2012wrapped} and others. Kontsevich \citep{kontsevich1994homological} conjectured that the wrapped Fukaya category of a Weinstein manifold in any dimension can be calculated from a cosheaf of categories on its skeleton; this has been recently proven by the work of Ganatra, Pardon and Shende \citep{ganatra2018microlocal,ganatra2018structural,ganatra2017covariantly}. The description can be made more explicit by working with constructible sheaves \citep{nadler2009constructible}, and the work of Nadler \citep{nadler2017arboreal,nadler2016wrapped} furnishes combinatorial models for these cosheaves of categories. This particular model applies to Weinstein manifolds with appropriately generic `arboreal' skeleta, and the local data are given by quiver representation categories. Comparing to the results of this paper, this model appears very suitable to the application of relative stability conditions, since the study of stability conditions on quiver representation categories is in general much simpler than on `more geometric' categories.

This paper also opens up the possibility of using these sheaf-theoretic techniques to address some questions about dynamics on surfaces; the work of Dimitrov, Haiden, Katzarkov and Kontsevich \citep{dimitrov2014dynamical,dimitrov2015bridgeland,dimitrov2016bridgeland} investigates the relation between dynamical systems on surfaces and stability conditions on their Fukaya category. The relation between Teichm\"uller theory and stability conditions was already noted in \citep{bridgeland2015quadratic,gaiotto2013wall}, and in particular there is a close relation between the set of stable phases $\Phi$ (which we analyze for some cases in Section \ref{sec:calculations}) and measures of dynamical entropy for categories. For now, the possible applications of our methods to such questions remain topic of current and future investigations.

\end{document}